\documentclass[EJP]{ejpecp} 

\usepackage[T1]{fontenc}
\usepackage[utf8]{inputenc}

\usepackage[shortlabels]{enumitem}
\usepackage{bm}
\usepackage[capitalize]{cleveref}

\SHORTTITLE{A decorated tree approach to random permutations in substitution-closed classes}

\TITLE{A decorated tree approach to random permutations in substitution-closed classes} 

\AUTHORS{%
  Jacopo~Borga\footnote{Institut für Mathematik, Universität Zürich.
    \EMAIL{jacopo.borga@math.uzh.ch}}
  \and
  Mathilde~Bouvel\footnote{Institut für Mathematik, Universität Zürich.
  	\EMAIL{mathilde.bouvel@math.uzh.ch}}
  \and
  Valentin~Féray\footnote{Institut für Mathematik, Universität Zürich.
  	\EMAIL{valentin.feray@math.uzh.ch}}
  \and
  Benedikt~Stufler\footnote{Institut für Mathematik, Universität München.
  	\EMAIL{stufler@math.lmu.de}}
  }


\AMSSUBJ{60C05,05A05} 

\SUBMITTED{April 15, 2019} 
\ACCEPTED{May 26, 2020} 

\ARXIVID{1904.07135} 

\VOLUME{0}
\YEAR{2016}
\PAPERNUM{0}
\DOI{10.1214/YY-TN}

\ABSTRACT{We establish a novel bijective encoding that represents permutations as forests of decorated (or enriched) trees. This allows us to prove local convergence of uniform random permutations from substitution-closed classes satisfying a criticality constraint. It also enables us to reprove and strengthen permuton limits for these classes in a new way, that uses a semi-local version of Aldous' skeleton decomposition for size-constrained  Galton--Watson trees.}

\newcommand{\MMM}{\mathcal{M}}   
\newcommand{\SSS}{\mathcal{S}}   
\newcommand{\Sall}{\SSS_{\text{all}}}  
   
\newcommand{\DDD}{\mathcal{D}}   
\newcommand{\GGG}{\mathcal{G}}   
\newcommand{\NN}{\mathbb N}

\newcommand{\esper}{\mathbb{E}}

\newcommand{\One}{\mathbb{1}}
\newcommand{\nonp}{{\scriptscriptstyle{\mathrm{not}}{\oplus}}}

\newcommand{\cR}{\mathcal{R}}
\newcommand{\lufT}{\mathfrak{T}^{\bullet,\text{\tiny luf}}}
\newcommand{\lufTD}{\mathfrak{T}^{\bullet,\text{\tiny luf}}_{\mathcal D}}
\newcommand{\lufPT}{\mathfrak{P}^{\bullet,\text{\tiny luf}}} 
\newcommand{\Proba}{\mathbb{P}}
\newcommand{\E}{\mathbb{E}}
\newcommand{\V}{\mathbb{V}}

\newcommand{\RR}{\mathbb{R}}
\newcommand{\Vint}{V_{\text{int}}}

\newcommand{\Seq}{ \textsc{Seq}}

\newcommand{\PackedTree}{P}

\newcommand{\Uinf}{\mathcal{U}_{\infty}^{\bullet}}
\newcommand{\Sr}{\SG_{\bullet}}
\newcommand{\Sri}{\tilde{\SG_{\bullet}}}
\newcommand{\Z}{\mathbb{Z}}
\newcommand{\Asi}{A_{\nu,i}}
\newcommand{\leqsi}{\preccurlyeq_{\nu,i}}
\newcommand{\coc}{c\text{-}occ}
\newcommand{\setTkt}{\mathcal T_{k,\Omega}^{[t]}}
\newcommand{\setTktn}{\mathcal T_{k,\Omega}^{[t_n]}}
\newcommand{\setToh}{\mathcal T_{1,\{0\}}^{[h]}}
\newcommand{\setTth}{\mathcal T_{2,\{0\}}^{[h]}}

\def\N{N}

\newcommand\restr[2]{{%
		\left.\kern-\nulldelimiterspace %
		#1 %
		\right|_{#2} %
}}%
\newcommand{\Av}{\mathrm{Av}}
\newcommand{\cC}{\mathcal{C}}
\newcommand{\myvec}[1]{\vec{#1}} 
\newcommand{\rv}[1]{\mathbf{#1}} 
 
\newcommand{\cS}{\mathcal{S}}
\newcommand{\cP}{\mathcal{P}}
\newcommand{\convdis}{\,{\buildrel \mathrm{d} \over \longrightarrow}\,}

\newcommand{\convp}{\,{\buildrel \mathrm{P} \over \longrightarrow}\,}
\newcommand{\eqdist}{\,{\buildrel \mathrm{d} \over =}\,}

\newcommand{\SG}{\mathfrak{S}}

\newcommand{\Prb}[1]{\mathbb{P}\left(#1\right)}

\DeclareMathOperator{\socc}{occ}
\DeclareMathOperator{\occ}{\widetilde{occ}}

\DeclareMathOperator{\pat}{pat}

\DeclareMathOperator{\Occ}{Occ}

\DeclareMathOperator{\size}{size}
\DeclareMathOperator{\dec}{dec}

\DeclareMathOperator{\Sh}{Sh}
\DeclareMathOperator{\Lab}{Lab}
\DeclareMathOperator{\height}{ht}
\DeclareMathOperator{\CanTree}{CT} 
\DeclareMathOperator{\Pack}{PA} 
 
\DeclareMathOperator{\DF}{DF} 
\DeclareMathOperator{\DT}{DT} 
\DeclareMathOperator{\RP}{RP} 
\DeclareMathOperator{\sgn}{sgn} 

\begin{document}



\section{Introduction}
\label{sec:intro}

\subsection{Uniform random permutations in classes: some background and overview of our results}
We assume some familiarity of the reader with basic definitions of permutation patterns and permutation classes,
\emph{i.e.,} what is a pattern, an occurrence and a consecutive occurrence, a class, its basis, \dots~If needed,
the definitions of these notions are given at the end of the introduction.

Permutation classes are classically studied from an enumerative point of view,
\emph{i.e.,} one wants to compute the number of permutations of any fixed size in a given class
or the generating function of the class (possibly refining according to some statistics).
In recent years, there has also been an increasing interest
in the behaviour of a large typical permutation taken in a given permutation class.
We refer for example to \cite{Borga2019,MR3632417,MR3704772,MR3894923, doi:10.1002/rsa.20806,Ja_multiple,MR3176717,pinsky} 
for results on random $\tau$-avoiding permutations with $\tau$ of size $3$.
Other specific classes (or sets of permutations) have been studied: permutations avoiding a monotone pattern of any size \cite{HRSinpreparation}, separable permutations \cite{MR3813988}, 
square permutations \cite{borga2019almost,borga2019square}, doubly alternating Baxter permutations \cite{MR3238333}. 
The recent paper \cite{bassino2017universal} by Bassino, Bouvel, F{\'e}ray, Gerin, and  Maazoun uses singularity analysis methods 
to study random permutations from {\em substitution-closed} classes satisfying some analytic assumptions.

The present work takes a probabilistic approach to the analysis of random permutations from substitution-closed classes.
We establish a novel encoding of these permutations as decorated conditioned monotype Galton--Watson forests, hence integrating them into the framework of random enriched trees and tree-like structures introduced by Stufler~\cite{stufler2016limits,StEJC2018}.
This yields a unified and powerful way for describing their asymptotic shape on a global and local scale.
\begin{itemize}
	\item 
	As a first application we give a new proof of the main scaling limit result of \cite{bassino2017universal, MR3813988} by using an extension of Aldous' skeleton decomposition~(see \cite{MR1207226} for Aldous' original
	statement, and \cref{le:semilocal} for our extension).
	This new proof %
	works under weaker conditions
	and makes transparent the connection to random trees which was suggested, but unclear, in \cite{bassino2017universal}
	(see in particular Remark 1.11 or the beginning of Section 1.7 there).
	In particular,
	our proof yields a probabilistic interpretation of the conditions under which
	this scaling limit result holds (see \cref{ssec:intro_Conditions}).
	\item 
	Our second main contribution is a novel quenched local limit for random permutations from substitution-closed classes. Here we use fringe subtree count asymptotics and the  skeleton decomposition to describe a concentration phenomenon for consecutive patterns. This notion of convergence has recently been introduced by Borga in \cite{Borga2019}, where such limits were proven for random permutations avoiding patterns of length $3$.
\end{itemize}

The rest of the introduction defines substitution-closed classes and
provides details on our results and on the approach used in this paper.

\subsection{Substitution of permutations and closed classes}
To define the substitution operation, it is convenient to think of permutations 
as diagrams. That is, if $n$ denotes the size of a permutation~$\nu$, we may identify $\nu$ with the set of points $(i,\nu(i))$ (for $i$ in $[n]$).
The substitution $\theta[\nu^{(1)},...,\nu^{(d)}]$, where $\theta, \nu^{(1)}, \ldots, \nu^{(d)}$ are permutations and $d$ is the size of $\theta$,
is then obtained as follows. For each $i$,
we first replace the point $(i,\theta(i))$ with the diagram of $\nu^{(i)}$.
Then rescaling the rows and columns yields the diagram of a bigger permutation,
which is by definition $\theta[\nu^{(1)},...,\nu^{(d)}]$.
A permutation of size greater than $2$ is called {\em simple}
if it cannot be obtained as the substitution of smaller permutations.
An example of substitution is given in \cref{fig:Subs}.
\begin{figure}[htbp]
	\centering
	\includegraphics[height=2.5cm]{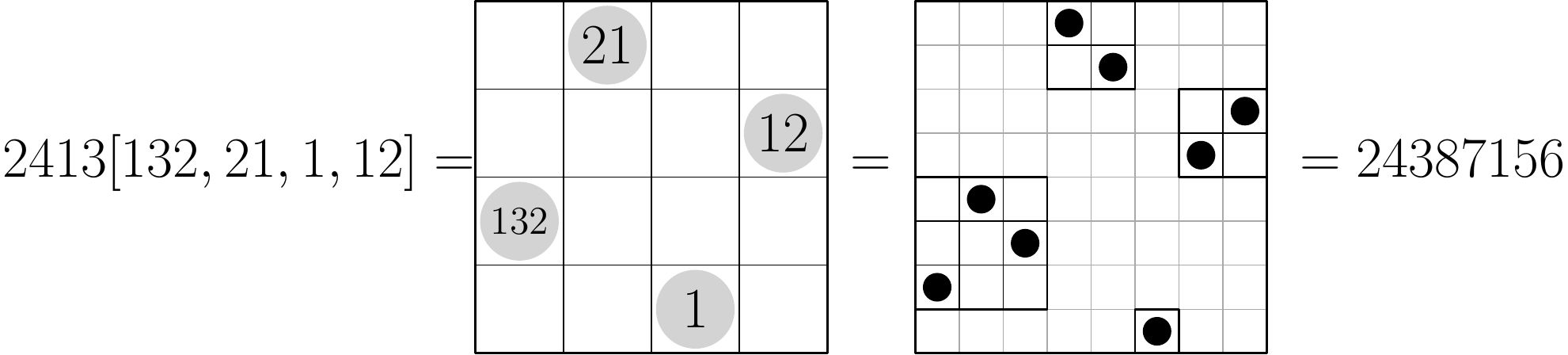}
	\caption{Example of substitution of permutations}
	\label{fig:Subs}
\end{figure}

As said above, this article considers {\em substitution-closed classes} of permutations,
\emph{i.e.,} classes $\mathcal C$ such that $\theta,  \nu^{(1)},...,\nu^{(d)} \in \mathcal C$
implies $\theta[\nu^{(1)},...,\nu^{(d)}] \in \mathcal C$.
Alternatively, a class is substitution-closed
if and only if its basis (\emph{i.e.,} the avoided patterns defining the class)
consists only of simple permutations.
In particular, there are uncountably many substitution-closed permutation classes.
Due to their nice combinatorial structure (see \cref{sec:bijections}),
substitution-closed permutation classes are a nice general framework,
where to investigate the properties of uniform random elements.

We note that a substitution-closed class $\mathcal C$ is entirely determined by the set $\mathcal S$
of simple permutations in it (see~\cref{prop:treesOfSubsClosedClasses}).
We consider this set $\mathcal S$ as the data of our problem, 
and the goal is, under various conditions on $\mathcal S$
to obtain convergence results for
uniform random permutations in the class $\mathcal C$.
These conditions will typically be expressed in terms of the generating functions
of $\mathcal S$, that we conveniently also denote $\mathcal S$.
From Stanley-Wilf-Marcus-Tard\"os' theorem \cite{MR2063960},
it always has a positive radius of convergence $\rho_\cS>0$
(except in the trivial case where $\mathcal C$ is the set $\SG$ of all permutations, 
which we exclude from now on; permutation classes different from $\SG$ are called \emph{proper}).

\subsection{Permuton convergence of substitution-closed classes}
The notion of permutons was introduced in \cite{MR2995721} to describe limits of permutation sequences.
Formally, a permuton is a probability measure on the unit square $[0,1]^2$,
whose projection on each axis is the Lebesgue measure on $[0,1]$
(we say that the measure has uniform marginals).
Permutations can be seen as permutons by considering the rescaled diagrams;
we will denote $\mu_\nu$ the permuton associated with the permutation $\nu$.
The weak topology on measures gives then a natural meaning to the convergence
of a sequence of permutations to a given permuton.
A nice feature is that the convergence in terms of permutons
is equivalent to the {\em convergence of pattern proportions}.
We refer to \cite[Section 2]{bassino2017universal} for details.

Some specific permutons have been described as limits of permutation classes, as in~\cite[Chapter 6]{bevanThesis}, \cite{borga2019square} and~\cite{MR3813988,bassino2017universal,bassino2019finitelyGenerated}.
Among these, the {\em biased Brownian separable permuton} $\bm{\mu}^{(p)}$ of parameter $p$
is a random permuton, constructed from a Brownian excursion and independent signs associated
with its local minima, see Maazoun~\cite{maazoun17BrownianPermuton}.
It was proved in \cite{bassino2017universal, MR3813988}  that this
is a universal limiting object for substitution-closed permutations classes,
in the sense that uniform random permutations in many substitution-closed classes
converge to $\bm{\mu}^{(p)}$, for some $p$.
In this article, we give a new proof of this theorem that is based on an extension of Aldous' skeleton decomposition~\cite{MR1207226} and the framework of random enriched trees and tree-like structures~\cite{stufler2016limits,StEJC2018}.
\begin{theorem}
	\label{thm:scaling_intro}
	Let $\bm{\nu}_n$ be the uniform $n$-sized permutation from a proper substitution-closed class of permutations $\cC$. Suppose that
	\begin{equation}
	\label{eq:S_ExpMoments}
	\cS'(\rho_\cS) > \frac{2}{(1 +\rho_\cS)^2} -1,
	\end{equation}
	or
	\begin{equation}
	\label{eq:S_FiniteVariance}
	\cS'(\rho_\cS)  = \frac{2}{(1 +\rho_\cS)^2} -1 \qquad \text{and} \qquad 	\cS''(\rho_\cS)  < \infty.
	\end{equation}
	Then
	\[
	\bm{\mu}_{\bm{\nu}_n} \convdis \bm{\mu}^{(p)},
	\]
	with $\bm{\mu}^{(p)}$ denoting the biased Brownian separable permuton 
	with an explicit parameter given by \cref{eq:parm_p} page \pageref{eq:parm_p}. This includes the case of uniform separable permutations, for which $\cS = \emptyset$ and $p=1/2$.
\end{theorem}

Specifically, the result~\cite[Thm. 1.2]{MR3813988} corresponds to the special case where $\cS= \emptyset$, and \cite[Thm. 1.10]{bassino2017universal} corresponds to the special case where \cref{eq:S_ExpMoments} is satisfied. The result~\cite[Thm. 7.8]{bassino2017universal} corresponds to the case where \cref{eq:S_FiniteVariance} is satisfied and additionally $\cS(z)$ is amendable to singularity analysis.

\subsection{Local convergence: a concentration phenomenon for substitution-closed classes}
In addition to scaling limits,
our decorated tree approach also allows us to obtain local limit results 
for uniform random permutations in substitution-closed classes.
For this, we use a local topology for permutations recently defined by Borga
in \cite{Borga2019}. This topology is the analogue of the celebrated Benjamini--Schramm convergence
for graphs, in the sense that we look at the neighbourhood of a {\em random element} of the permutation.
Pleasantly, convergence for this local topology is equivalent to 
the {\em convergence of consecutive pattern proportions}.

For convenience, we present our results in term of consecutive patterns.
For a permutation $\nu$ and a pattern $\pi$, 
we denote by $\coc(\pi,\nu)$ the number of consecutive occurrences of a pattern $\pi$ in $\nu$;
for instance, for $\pi=21$ (resp. $\pi=321$), these are the number of descents (resp. double-descents)
in the permutation. 
\begin{theorem}
	\label{thm:local_intro}
	Let $\mathcal{C}$ be a proper substitution-closed permutation class and assume that
	\begin{equation}
	\label{eq:S_TypeI}
	\cS'(\rho_\cS) \ge \frac{2}{(1 +\rho_\cS)^2} -1.
	\end{equation}
	For each $n \in \NN$, we consider a uniform random permutation $\bm{\nu}_n$ of size $n$
	in $\mathcal{C}$.
	Then, for each pattern $\pi \in \mathcal C$,
	there exists $\gamma_{\pi,\mathcal C}$ in $[0,1]$ such that
	\begin{equation}
	\tfrac1n \, \coc(\pi,\bm{\nu}_n) \stackrel{P}{\longrightarrow}  \gamma_{\pi,\mathcal C}.
	\end{equation}
\end{theorem}
We note that the hypothesis made in this theorem is slightly weaker than that for scaling limits. The theorem shows the convergence
of all random variables $\tfrac1n \, \coc(\pi,\bm{\nu}_n)$
to deterministic constants,
revealing a "concentration" phenomenon in substitution-closed class
under hypothesis \eqref{eq:S_TypeI}.
The constants $\gamma_{\pi,\mathcal C}$ can be constructed from local limits of 
conditioned Galton--Watson trees around a random leaf, see \cref{sec:local}
and in particular \cref{rk:gammas}.
They depend both on the pattern $\pi$ and on the class $\mathcal C$.

\subsection{Proof methodology}
Start with a permutation $\nu$ of size $n \ge 2$.
If it is not simple nor monotone\footnote{By definition, for $k \ge 2$, there are exactly two monotone permutations
	of size $k$: the monotone increasing one $12 \cdots k$ and the monotone decreasing one $k \cdots 21$.}, it can be written as $\theta[\nu^{(1)},...,\nu^{(d)}]$, for some smaller
permutations $\theta, \nu^{(1)},...,\nu^{(d)}$.
We can iterate this decomposition on $\theta, \nu^{(1)},...,\nu^{(d)}$: as long as they are not simple nor monotone, we decompose them further through substitution.
The result is a representation of $\nu$ as a tree with $n$ leaves,
whose internal vertices are decorated by monotone or simple permutations.
We call positive (resp. negative) a decoration 
with a monotone increasing (resp. decreasing) permutation.
From a result of Albert and Atkinson \cite{albert2005simple},
the decorated tree representation of a permutation is unique
if we forbid the children of a vertex with a positive (resp. negative) decoration  
to have themselves a positive (resp. negative) decoration.
The resulting tree is called the {\em canonical tree} of the permutation.
Details on this construction, standard in the permutation pattern literature,
are given in \cref{sec:DecoTrees}.
\medskip

From a probabilistic point of view, one wants to consider the tree associated with a uniform
random permutation in a class $\mathcal C$ and possibly to recognize some standard tree models.
One can show that, if the class is substitution-closed,
the associated random canonical tree is a {\em multitype} random Galton--Watson tree
with some specific offspring distribution conditioned on having $n$ leaves.
The need to have several types comes from the condition 
on positive and negative decorations:
this forces us to consider children of positively and negatively decorated vertices
to be of a different type from other vertices in the tree.

Results on conditioned multitype Galton--Watson trees do exist in the literature:
in particular, there are some scaling limit results under finite or infinite variance assumptions
\cite{MR3748121,deraphelis2017,MR2469338}, and local limit results around the root
\cite{MR3803914,MR3769811} for such trees.
Nevertheless these results do not cover our needs.
\begin{itemize}
	\item For the scaling limit results on permutations,
	we need information on the type and outdegree of the closest common ancestors of randomly selected leaves
	(while tree scaling limit results only give information on the genealogy of such leaves).
	\item For the local limit results on permutations,
	we need some local limit results around a random leaf, and not around the root. For studying local convergence of random separable permutations we additionally require joint convergence with the parity of the height of the leaf.
\end{itemize}
We therefore do not use this encoding as multitype Galton--Watson trees,
but rather provide a novel encoding of random permutations in substitution-closed classes
as {\em decorated monotype Galton--Watson forests}. That is, random plane forests where each vertex is enriched with an independent local structure. This integrates the random permutations naturally into the framework of random tree-like structures~\cite{stufler2016limits}.
\medskip

To identify permutations with decorated forests, 
we first note that a generic permutation is the $\oplus$-sum of an ordered sequence of $\oplus$-indecomposable permutations, 
\emph{i.e.,} of permutations which cannot be obtained as a substitution $12[\pi^{(1)},\pi^{(2)}]$
(see \cref{Th:AlbertAtkinson} below). 
We then associate to each of these $\oplus$-indecomposable permutations its canonical tree.
To those trees, we apply a packing procedure.
This packing procedure merges vertices decorated 
with a simple permutation with its children having a positive decoration.
As a consequence, we do not need anymore to distinguish between positive and negative decorations.
The resulting tree, called packed tree of the ($\oplus$-indecomposable) permutation,
is still a decorated tree with $n$ leaves,
but the decorations are now more complicated objects than permutations,
being themselves trees of permutations 
(called $\cS$-gadget below, see \cref{sect:Packed_decomposition_trees}
for details).
The advantage of this new representation is that there is no condition
on the decoration of a vertex, depending on the one of its parent. 
As a result of this construction, any permutation is represented as an ordered sequence of decorated trees,
\emph{i.e.,} an ordered decorated forest, without any constraint on the decorations (\cref{te:bijection}).
We note that this representation is a bijection from the set of all permutations
to ordered decorated forests, and could thus be of interest,
independently from its application to the study of random elements in substitution-closed classes done here.
\medskip

To study random permutations of size $n$ taken uniformly at random 
in a substitution-closed class $\mathcal C$, 
we use a result on convergent Gibbs partitions (see Stufler~\cite[Thm. 3.1]{stufler2016gibbs})
to prove that the associated ordered decorated forest contains a giant tree of size $n-O_p(1)$ (\cref{prop:giant_comp_perm} page \pageref{prop:giant_comp_perm}).
It is therefore enough to study a random decorated \emph{tree} with $n$ leaves.
Such trees have the same distribution as a {\em monotype} Galton--Watson tree $\bm T^\xi_n$ 
with a specific offspring distribution $\xi$ conditioned on having $n$ leaves.
We can therefore use results or techniques on monotype Galton--Watson trees,
which are much more developed than in the multitype case.
\begin{itemize}
	\item
	In particular, to find the scaling limit of our permutations, there are some results on the genealogy and the outdegree of common ancestors
	of randomly chosen vertices (this is implicit in the original paper of Aldous, see~\cite[Eq. (49)]{MR1207226}).
	We will refer to this as {\em Aldous' skeleton decomposition}.
	In this article we will need an extension of this, 
	considering also local neighbourhood of the common ancestors (being therefore {\em semi-local})
	and allowing to condition on the number of leaves instead of the number of vertices.
	\cref{le:semilocal} provides a general result to this effect,
	allowing to condition on the number of vertices
	with arity in any given set $\Omega \subseteq \mathbb \NN_0$ satisfying $\mathbb{P}(\xi \in \Omega)>0$.
	\item
	The literature also contains results on the number of (extended) fringe subtrees of $\bm T^\xi_n$ (and related models) 
	isomorphic to a given tree \cite{MR1102319,holmgren2017,janson2012simply,stufler2016limits,stufler2019,stufler2019offspring}.
	When $\xi$ is critical, such results may be translated to local limit results for $\bm T^\xi_n$, pointed at a random leaf
	(see \cref{prop:quenched_conv_tree}).
	We shall however need and will prove a slightly stronger result when $\xi$ is critical and additionally has finite variance,
	taking also into account the parity of the height of the pointed leaf
	(see \cref{prop:quenched_conv_tree_Plus_Signs} page \pageref{prop:quenched_conv_tree_Plus_Signs}).
\end{itemize}
\medskip

The last step of the proofs (both in the scaling and local limit cases)
is to translate the results on the packed trees to results on the permutation 
$\bm \nu_n$ itself.
A difficulty here arises from the identification of positive and negative decorations
in the packing construction. To invert this construction,
and recover the correct signs on the decorated trees, we need to determine the distance to the closest
ancestor decorated with a simple permutation.
When $\mathcal S \ne \emptyset$, this ancestor is at a stochastically bounded distance, so that
this inversion procedure is still local.
However, when $\mathcal S= \emptyset$, \emph{i.e.,} in the case of {\em separable permutations},
there is no such ancestor and we need to go all the way to the root to invert the packing construction.
This creates an extra difficulty, that we overcome by using a local limit theorem for the length of ``bones'' in the skeleton decomposition.\medskip

\subsection{Interpretation of the various assumptions on $\cS$}~
\label{ssec:intro_Conditions}
Our assumptions on $\cS$ might seem artificial but they are in fact very natural,
after having introduced the above representation of permutations
as decorated conditioned Galton--Watson forests.
Namely
\begin{itemize}
	\item \cref{eq:S_TypeI} is equivalent to the fact that the Galton--Watson tree model is critical;
	\item \cref{eq:S_ExpMoments} asks in addition that the offspring distribution has small exponential moments;
	\item finally, \cref{eq:S_FiniteVariance} means that the offspring distribution has no exponential moments,
	but finite variance.
\end{itemize}
Such hypotheses are classical in the analysis of conditioned Galton--Watson trees, and give a probabilistic meaning to the conditions used in~\cite{bassino2017universal}. In terms of substitution-closed classes, the small exponential moment condition
is satisfied for most classes in the literature, see the discussion in 
\cite[Section 1.4]{bassino2017universal}.
Although general classes satisfying \cref{eq:S_TypeI} but not \cref{eq:S_ExpMoments} have also been studied in previous works~\cite[Sec. 7]{bassino2017universal}, we do not know at present whether such classes exist. 

\medskip

There is however at least one class not satisfying \cref{eq:S_TypeI}: the class $\Av(2413)$.
The packed forest associated with a uniform random permutation $\bm \nu_n$ in this class
has the distribution of a decorated conditioned Galton--Watson forest 
with a {\em subcritical offspring distribution}.
It will therefore contain with high probability a unique vertex with macroscopic degree (see \cite{janson2012simply,MR2764126,MR3335012,stufler2019offspring}).
This  vertex is decorated with a large simple permutation $\bm \alpha_n$ in the class
and the scaling (resp. local) limit of $\bm \nu_n$ could be described \emph{if} one knew that of~$\bm \alpha_n$.
In the current state of the art,
studying a uniform random {\em simple} $\bm \alpha_n$ in $\Av(2413)$
does not seem to be a simpler problem than the original one of studying $\bm \nu_n$, hence this approach appears to be ineffective for $\Av(2413)$.
\bigskip

\subsection{Outline of the paper}
The paper is organized as follows. The rest of the introduction sets up some notation.
\cref{sec:bijections} presents the combinatorial construction used in this paper,
that is the canonical tree and packed forest associated with a permutation.
\cref{sec:indecomposable_GW} identifies the packed forest associated
with a uniform random permutation
in a substitution-closed class as a conditioned monotype Galton--Watson forest.
We also discuss the existence of a giant tree in such a forest.
In \cref{sec:skeleton}, we state and prove our improvement of Aldous' skeleton decomposition.
The last two sections are devoted to the proofs of the main theorems:
\cref{sec:scaling} for the scaling limit result (\cref{thm:scaling_intro}) and \cref{sec:local} for the local limit result (\cref{thm:local_intro}).

\subsection{Permutation patterns and permutation classes: basic definitions and notation}
We let $\NN_0 = \{0, 1, \ldots\}$ denote the collection of non-negative integers and $\NN = \{1,2 \ldots\}$ the collection of strictly positive integers. For any $n\in\NN,$ we denote the set of permutations of $[n]:=\{1,2,\dots,n\}$ \label{def:SG} by $\SG^n.$ We write permutations of $\SG^n$ in one-line notation as $\nu=\nu(1)\nu(2)\dots\nu(n).$ For a permutation $\nu\in\SG^n$ the \emph{size} $n$ of $\nu$ is denoted by $|\nu|.$ We let $\SG\coloneqq\bigcup_{n\in\NN}\SG^n$ be the set of finite permutations. We write sequences of permutations in $\SG$ as $(\nu_n)_{n\in\NN}.$ 

We will often view a permutation $\nu$ as its diagram, which is (as said earlier -- see also the right part of \cref{fig:Subs}) the set of points of the Cartesian plane at coordinates $(j,\nu(j))$.

If $x_1\dots x_n$ is a sequence of distinct numbers, let $\text{std}(x_1\dots x_n)$ be the unique permutation $\pi$ in $\SG^n$ that is in the same relative order as $x_1\dots x_n,$ \emph{i.e.}, $\pi(i)<\pi(j)$ if and only if $x_i<x_j.$
Given a permutation $\nu\in\SG^n$ and a subset of indices $I\subseteq[n]$, let $\text{pat}_I(\nu)$ be the permutation induced by $(\nu(i))_{i\in I},$ namely, $\text{pat}_I(\nu)\coloneqq\text{std}\big((\nu(i))_{i\in I}\big).$
For example, if $\nu=87532461$ and $I=\{2,4,7\}$ then $\text{pat}_{\{2,4,7\}}(87532461)=\text{std}(736)=312$. 

Given two permutations, $\nu\in\SG^n$ for some $n\in\NN$ and $\pi\in\SG^k$ for some $k\leq n,$ we say that $\nu$ contains $\pi$ as a \textit{pattern} (and we write $\pi\leq\nu$) 
if $\nu$ has a subsequence of entries order-isomorphic to $\pi,$ that is, if there exists a \emph{subset} $I\subseteq[n]$ such that $\text{pat}_I(\nu)=\pi$.
Denoting $i_1, i_2, \dots, i_k$ the elements of $I$ in increasing order, the subsequence $\nu(i_1) \nu(i_2) \dots \nu(i_k)$ is called an \emph{occurrence} of $\pi$ in $\nu$. 
In addition, we say that $\nu$ contains $\pi$ as a \textit{consecutive pattern} if $\nu$ has a subsequence of adjacent entries order-isomorphic to $\pi$, 
that is, if there exists an \emph{interval} $I\subseteq[n]$ such that $\text{pat}_I(\nu)=\pi$. 
Using the same notation as above, $\nu(i_1) \nu(i_2) \dots \nu(i_k)$ is then called a \emph{consecutive occurrence} of $\pi$ in $\nu$. 
All along the article, for any integers $a,b\in\Z$ (resp. $n \in \N$), the interval $[a,b]$ (any interval $I \subseteq [n]$) is to be understood as an integer interval, \emph{i.e.}, an interval contained in $\Z$.
For real numbers $a\leq b$, we use the same notation $[a,b]$ to denote the interval $[a,b] \subseteq \mathbb{R}$ 
\begin{example} 
	The permutation $\nu=1532467$ contains $1423$ as a pattern but not as a consecutive pattern and $321$ as consecutive pattern. Indeed $\text{pat}_{\{1,2,3,5\}}(\nu)=1423$ but no interval of indices of $\nu$ induces the permutation $1423.$ Moreover, $\text{pat}_{[2,4]}(\nu)=\text{pat}_{\{2,3,4\}}(\nu)=321.$
\end{example}

We say that $\nu$ \emph{avoids} $\pi$ if $\nu$ does not contain $\pi$ as a pattern. We point out that the definition of $\pi$-avoiding permutations refers to patterns and not to consecutive patterns. Given a set of patterns $B\subseteq\SG,$ we say that $\nu$ \emph{avoids} $B$ if $\nu$ avoids $\pi,$ for all $\pi\in B$. We denote by $\text{Av}^n(B)$ the set of $B$-avoiding permutations \ of size $n$ and by $\text{Av}(B)\coloneqq\bigcup_{n\in\NN}\text{Av}^n(B)$ the set of $B$-avoiding permutations of arbitrary size. 

A \emph{permutation class} $\mathcal{C}$ is a set of permutations closed under the operation of pattern-containment (\emph{i.e.,} if $\nu\in\mathcal{C}$ and $\pi\leq\nu$ then $\pi\in\mathcal{C}$). We recall that every permutation class can be rewritten as a family of pattern-avoiding permutations, \emph{i.e.,} for every permutation class $\mathcal{C}$ there exists a set of patterns $B\subseteq\SG$ such that $\mathcal{C}=\text{Av}(B).$ Note that if one permutation of $B$ is contained in another then we may remove the larger one without changing the family. Thus we may take $B$ to be an \emph{antichain}, meaning that no element of $B$ contains any others. In the case that $B$ is an antichain we call it the \emph{basis} of this family. 
We note that the basis of a class may be finite or infinite. 

\subsection{Probabilistic notation}
In order to avoid any confusion, we write random quantities using \textbf{bold characters} to distinguish them from deterministic quantities. 
Moreover, given a random variable $\bm{X},$ we denote with $\mathcal{L}(\bm{X})$ its law. Unless otherwise stated, all limits are taken as $n \to \infty$. 
Given a sequence of random variables $(\bm{X}_n)_{n\in\NN}$ we write $\bm{X}_n\convdis\bm{X}$ to denote  convergence in distribution and $\bm{X}_n\convp\bm{X}$ to denote  convergence in probability. 
We let $O_p(1)$ represent an unspecified
random variable $\bm{Y}_n$ of a stochastically bounded sequence $(\bm{Y}_n)_n$. 

Besides, the indicator of an event $A$ is denoted $\One[A]$.
Finally, the expression {\em with high probability} means {\em with probability tending to 1}
(without precision on the speed of convergence).

\subsection{Index of notation}

\cref{table:notation} summarizes some notational conventions and frequently used terminology in this paper.
In general, for a combinatorial class denoted by a curly letter, e.g. $\cP$,
we use the same letter $\cP(z)$ for its generating series, $\rho_{\cP}$ for the radius of convergence of $\cP(z)$, 
a standard uppercase letter $P$ for an object in the class,
and a lowercase letter $p_n$ with index $n \ge 0$ for the number of objects of size $n$. 
Bijections between classes will be denoted by two upper case letters, e.g. $\CanTree$ (page \pageref{def:cantree}) building the canonical tree of a permutation. 

\begin{table}[ht]
	\begin{center}
		\begin{tabular}{ll}
			$\mathfrak{S}^n$ & the set of permutations of size $n,$ page \pageref{def:SG}\\[2pt]
			$\cC$ & a substitution-closed class of permutations, page \pageref{def:substclosed}\\[2pt]
			$\cS$ & the subset of simple permutations in $\cC$, page \pageref{prop:treesOfSubsClosedClasses}\\[2pt]
			$\mathcal{T}$ & the class of canonical trees associated with $\cC$, page \pageref{def:ct}\\[2pt]
			$\cP$ & the class of packed trees associated with $\cC$, page \pageref{le:bij_perm_tree}\\[2pt]
			$\bm{\nu}_n$ & the uniform random $n$-sized permutation from $\cC$, page \pageref{def:bmnu}\\[2pt]
			$\bm{P}_n = (\bm{T}_n,\bm{\lambda}_{\bm{T}_n})$ & the uniform random packed tree with $n$ leaves, page \pageref{def:Pn}\\[2pt]
			$\mathfrak{T}^\bullet$ & the collection of (possibly infinite) pointed plane trees, page \pageref{def:mfTb}\\[2pt]
			$\lufT$ & the collection of (possibly infinite) {locally and upwards finite}\\& pointed plane trees, page \pageref{def:lufT}\\[2pt]
			$\lufTD$ & the collection of (possibly infinite) {locally and upwards finite} \\&decorated pointed plane trees, page \pageref{def:luftd}\\[2pt]
			$\lufPT$ & the collection of (possibly infinite) {locally and upwards finite }\\& pointed packed trees, page \pageref{def:pluft}%
		\end{tabular}
	\end{center}
	\caption{Table of the main notation.} \label{table:notation}
\end{table}

\section{A novel encoding of permutations as forests of decorated trees}
\label{sec:bijections}

In this section we show that any substitution-closed class of permutations may be bijectively encoded as a forest of trees decorated (or enriched) with local structures. This goal is achieved in Theorem~\ref{te:bijection}, in Subsection~\ref{sec:final}.

\subsection{Basics on combinatorial classes and decorated trees}

In this paper, we only consider rooted (a.k.a.\ planted) plane trees;
{\em plane} means that the children of a given vertex are endowed with a linear order.
Throughout the paper, the \emph{outdegree} $d^+_T(v)$ (or $d^+(v)$ when there is no ambiguity) of a vertex $v$ in a tree $T$ is the number of its children (which is sometimes called arity in other works). 
Note that it may be  different from the graph-degree:
the edge to the parent (if it exists) is not counted in the outdegree. 
We consider both finite and infinite trees. 
We say a tree is {\em locally finite}, if all its vertices have finite degree. 
A vertex of $T$ is called a {\em leaf}, if it has outdegree zero. The collection of non-leaves (also called \emph{internal vertices}) is denoted by $\Vint(T)$.
The \emph{fringe subtree} of a tree $T$ rooted at a vertex $v$ is the subtree of $T$ containing $v$ and all its descendants. 
We will also speak of {\em branch} attached to a vertex $v$ for a fringe subtree rooted at a child of $v$.

Any plane tree may be encoded in a canonical way as a subtree of the Ulam--Harris tree  $\mathcal{U}_{\infty}$. 
The vertex set of  $\mathcal{U}_{\infty}$ is given by the collection of all finite sequences of positive integers, 
and the offspring of a vertex $(i_1, \ldots, i_k)$ is given by all sequences $(i_1, \ldots, i_k, j)$, $j \ge 1$. The root of $\mathcal{U}_{\infty}$ is the unique sequence of length $0$.

Moreover, most trees considered here carry some additional structures on their vertices from a \emph{combinatorial class}.
Let $\DDD$ be a set and $\size:\DDD \to \NN_0$ be a map
from $\DDD$ to the set of non-negative integers, associating to each object in $\DDD$ its size. 
We say $\DDD$ is an (unlabelled) combinatorial class,  if for any $n \in \NN_0$ the number $d_n$ of $n$-sized objects in $\DDD$ is finite. This allows us to form the \emph{generating series}
\begin{align}
\DDD(z) = \sum_{n \ge 0} d_n z^n.
\end{align}
Note that we use the same curvy letter $\mathcal{D}$ for the class and its generating series.
This should hopefully not lead to confusions.
Two combinatorial classes $\mathcal{D}_1, \mathcal{D}_2$ are considered \emph{isomorphic} if there is a size-preserving bijection between the two,
or equivalently if they have the same generating series.

Various standard operations are available for combinatorial classes. 
For example, whenever $\DDD$ has no objects of size $0$,
we can form the combinatorial class $\Seq(\DDD)$,
which is the collection of finite sequences of objects from $\DDD$. 
The size of such a sequence is defined to be the sum of sizes of its components. 
We may also consider the subclass $\Seq_{\ge 1}(\DDD) \subset \Seq(\DDD)$ of non-empty sequences.

\begin{definition} \label{defn:decorated_tree}
	Let $\DDD$ be a combinatorial class.
	A $\DDD$-decorated (or $\DDD$-enriched) tree is a rooted locally finite plane tree $T$,
	equipped with a function $\dec:\Vint(T) \to \DDD$ from the set of internal vertices of $T$ to $\DDD$
	such that the following holds:
	{\em for each $v$ in $\Vint(T)$, the outdegree of $v$ is exactly $\size(\dec(v))$.} 
\end{definition}

This is a (planar) variant of Labelle's \emph{enriched trees}~\cite{MR642392}, which have been studied in~\cite{StEJC2018,stufler2016limits} from a probabilistic viewpoint. %

\subsection{Substitution decomposition and canonical trees}
\label{sec:DecoTrees}
We recall classical concepts related to permutation classes, including for expository purposes the concepts sketched in the introduction.

\begin{definition}%
	Let $\theta=\theta(1)\cdots \theta(d)$ be a permutation of size $d$, and let $\nu^{(1)},\dots,\nu^{(d)}$ be $d$ other permutations. 
	The \emph{substitution} of $\nu^{(1)},\dots,\nu^{(d)}$ in $\theta$,
	denoted by $\theta[\nu^{(1)},\dots,\nu^{(d)}]$, 
	is the permutation of size $|\nu^{(1)}|+ \dots +|\nu^{(d)}|$ 
	obtained by replacing each $\theta(i)$ by a sequence of integers isomorphic to $\nu^{(i)}$ while keeping the relative order induced by $\theta$ between these subsequences.\\
\end{definition}

Examples of substitution (see \cref{fig:sum_and_skew} below) are conveniently presented representing permutations by their diagrams: 
the diagram of $\nu = \theta[\nu^{(1)},\dots,\nu^{(d)}]$ is obtained by inflating each point $\theta(i)$ of $\theta$ by a square containing the diagram of $\nu^{(i)}$. 
Note that each $\nu^{(i)}$ then corresponds to a \emph{block} of $\nu$, a block being defined as an interval of $[|\nu|]$ which is mapped to an interval by $\nu$. 

Throughout this article, the increasing permutation $12 \ldots k$ will be denoted by $\oplus_k$,
or even $\oplus$ when its size $k$ can be recovered from the context:
this is the case in an inflation $\oplus[\nu^{(1)},\dots,\nu^{(d)}]$ where the size of $\oplus$ is the number $d$ of permutations inside the brackets.
Similarly, we denote the decreasing permutation $k \ldots 21$ by $\ominus_k$,
or $\ominus$ when there is no ambiguity.

\begin{figure}[htbp]
	\centering
		\includegraphics[width=6.5cm]{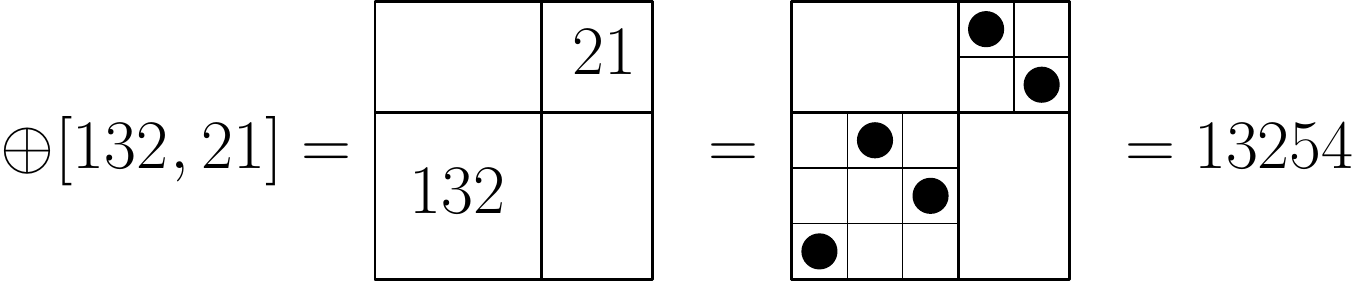} \qquad 
		\includegraphics[width=6.5cm]{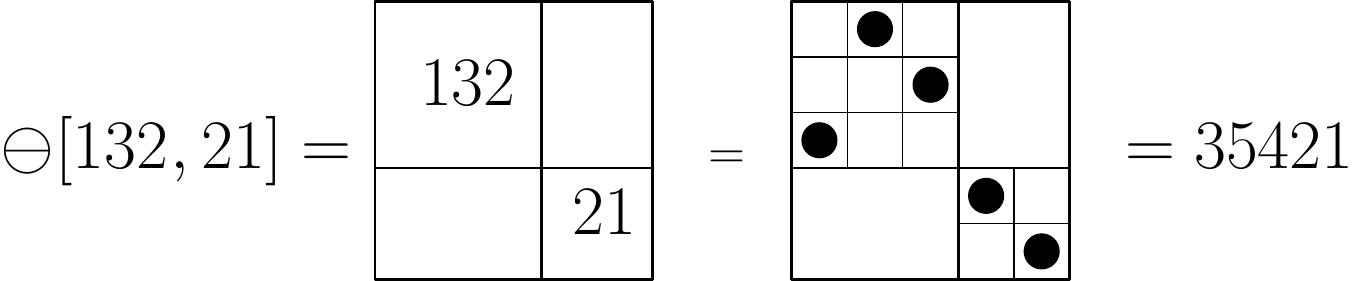}\medskip\\
		\includegraphics[width=95mm]{2413Sum}
	\caption{Substitution of permutations.\label{fig:sum_and_skew}}
\end{figure}

Permutations can be decomposed in a canonical way using recursively the substitution operation.
To explain this, we first need to define several notions of indecomposable objects.
\begin{definition}%
	A permutation $\nu$ is \emph{$\oplus$-indecomposable} (resp. $\ominus$-indecomposable) 
	if it cannot be written as $\oplus[\nu^{(1)},\nu^{(2)}]$ (resp. $\ominus[\nu^{(1)},\nu^{(2)}]$).
	
	A permutation of size $n > 2$ is \emph{simple} if it contains no nontrivial block, i.e., if it does not map any nontrivial interval 
	(i.e., a range in $[n]$ containing at least two and at most $n-1$ elements) onto an interval.
\end{definition}
For example, $451326$ is not simple as it maps the interval $[3,5]$ onto the interval $[1,3]$. 
The smallest simple permutations are $2413$ and $3142$ (there is no simple permutation of size $3$). 
We denote by $\Sall$ the set of simple permutations.
\medskip

\begin{remark}
Usually in the literature, the definition of a simple permutation requires $n \geq 2$ instead of $n > 2$, so that $12$ and $21$ are considered to be simple.
However, for decomposition trees, $12$ and $21$ do not play the same role as the other simple permutations, that is why we do not consider them to be simple.
\end{remark}

\begin{theorem}[Decomposition of permutations]\label{Th:AlbertAtkinson}
	Every permutation $\nu$ of size $n\geq 2$ can be uniquely decomposed as either:
	\begin{itemize}
		\item $\alpha[\nu^{(1)},\dots,\nu^{(d)}]$, where $\alpha$ is a simple permutation (of size $d\geq 4$),
		\item $\oplus[\nu^{(1)},\dots,\nu^{(d)}]$, where $d\geq 2$ and $\nu^{(1)},\dots,\nu^{(d)}$ are $\oplus$-indecomposable,
		\item $\ominus[\nu^{(1)},\dots,\nu^{(d)}]$, where $d\geq 2$ and $\nu^{(1)},\dots,\nu^{(d)}$ are $\ominus$-indecomposable.
	\end{itemize}
\end{theorem}

\begin{remark}
	The above theorem is essentially Proposition 2 in \cite{albert2005simple}, presented with a slightly different point of view.
	The decomposition according to \cref{Th:AlbertAtkinson} is obtained from the one of \cite[Proposition 2]{albert2005simple} 
	by merging maximal sequences of nested substitutions in $12$ (resp. $21$) into a substitution in $\oplus$ (resp. $\ominus$).
	For example, the second item above for $d=4$ corresponds to $12[\nu^{(1)}, 12[\nu^{(2)},12[\nu^{(3)},\nu^{(4)}]]]$ with the notation of \cite{albert2005simple}. 
	With this obvious rewriting, the statements of \cite[Proposition 2]{albert2005simple} and of \cref{Th:AlbertAtkinson} are trivially equivalent.  
\end{remark}

This decomposition theorem can be applied recursively
inside the permutations $\nu^{(i)}$ appearing in the items above, 
until we reach permutations of size $1$. 
Doing so, a permutation $\nu$ can be naturally encoded by a rooted labelled plane tree $\CanTree(\nu)$ as follows. \label{def:cantree}
(The notation $\CanTree(\nu)$ stands for \emph{canonical tree} -- see \cref{defintro:CanonicalTree}.)
\begin{itemize}
	\item If $\nu=1$ is the unique permutation of size $1$,
	then $\CanTree(\nu)$ is reduced to a single leaf.
	\item If $\nu=\beta[\nu^{(1)},\dots,\nu^{(d)}]$, where $\beta$ is either a simple permutation
	or the increasing (resp. decreasing) permutation (denoted by $\oplus$, resp. $\ominus$),
	then $\CanTree(\nu)$ has a root of degree $d$
	labelled by $\beta$ and the subtrees attached to the root
	are $\CanTree(\nu^{(1)})$, \ldots, $\CanTree(\nu^{(d)})$ (in this order from left to right).
\end{itemize}

From the above theorem, the decomposition $\nu=\beta[\nu^{(1)},\dots,\nu^{(d)}]$ exists
and is unique if $|\nu|~\ge~2$. Moreover, $\nu^{(1)},\dots,\nu^{(d)}$ have size smaller than $\nu$
so that this recursive procedure always terminates and its result is unambiguously defined.
The map $\CanTree$ is therefore well-defined.
An example of this construction is shown on \cref{fig:ExCanTree}.
\begin{figure}[htbp]
	\centering
		\includegraphics[height=5cm]{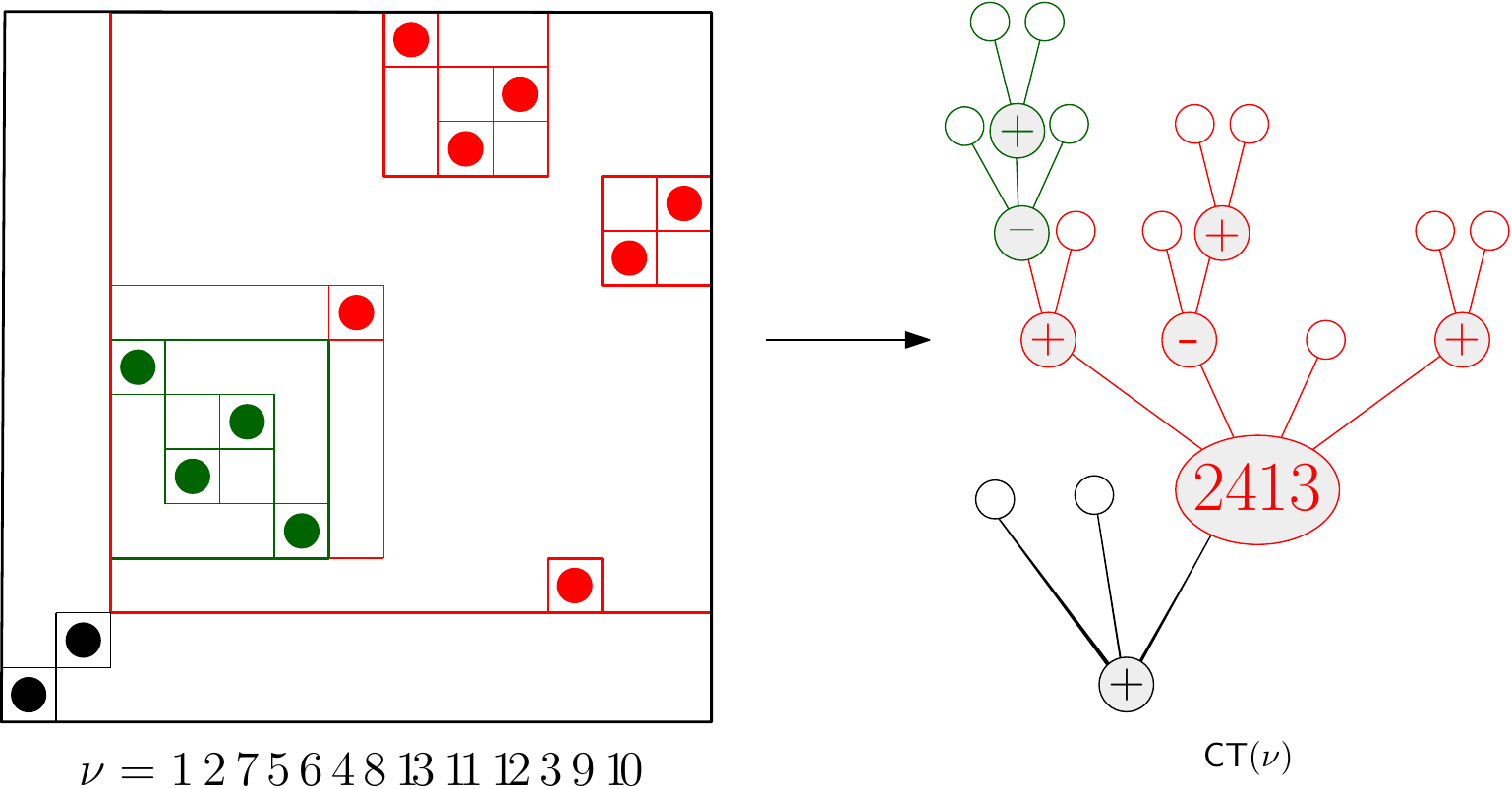}
	\caption{A permutation $\nu$ and its decomposition tree $\CanTree(\nu)$.
		To help the reader understand the construction, we have coloured accordingly
		some blocks of $\nu$ and some subtrees of $\CanTree(\nu)$.}
	\label{fig:ExCanTree}
\end{figure}

Since the labels of the vertex record the permutation $\beta$
in which we substitute, it is clear that $\CanTree$ is injective. 
Moreover, its inverse (once restricted to $\CanTree(\mathfrak{S})$) 
is immediate to describe, simply by performing the iterated substitutions recorded in the tree. 
We are just left with identifying the image set of $\CanTree$.
Recall that $\Sall$ denotes the set of all simple permutations, and let $\MMM$ be the set of all monotone (increasing or decreasing) permutations of size at least $2$. Denote $\widehat{\Sall}:=\Sall \cup \MMM$.
\begin{definition}\label{defintro:CanonicalTree}
	A \emph{canonical tree}  is an $\widehat{\Sall}$-decorated tree
	such that we cannot find two adjacent vertices both decorated
	with increasing permutations (i.e., with $\oplus$) or both decorated with decreasing permutations (i.e., with $\ominus$).
\end{definition}
Canonical trees are also known in the literature under several names:
decomposition trees, substitution trees,\ldots
We choose the term {\em canonical}
to be consistent with \cite{bassino2017universal}.
The following is an easy consequence of \cref{Th:AlbertAtkinson}.
\begin{proposition}\label{prop:Can_tree}
	The map $\CanTree$ defines a size-preserving bijection from the set of
	all permutations to the set of all canonical trees,
	the size of a tree being its number of \emph{leaves}.
\end{proposition}

\begin{remark}
	\label{rk:CT-1}
	We note that the inverse map $\CanTree^{-1}$, which builds a permutation from a canonical tree performing nested substitutions, 
	can obviously be extended to all $\widehat{\Sall}$-decorated trees, regardless of whether they contain $\oplus - \oplus$ or $\ominus - \ominus$ edges. 
	However, $\CanTree^{-1}$ is no longer injective on this larger class of ``non-canonical'' decomposition trees.
\end{remark}

We will be interested in the restriction of $\CanTree$
to some permutation class.
The following condition ensures that its image
has a nice description.

\begin{definition}%
	\label{def:substclosed}
	A permutation class $\mathcal{C}$ is \emph{substitution-closed}
	if  for every $\theta, \nu^{(1)},\dots,\nu^{(d)}$ in $\mathcal{C}$ it holds that
	$\theta[\nu^{(1)},\dots,\nu^{(d)}] \in \mathcal{C}$.
\end{definition}
\begin{proposition}
	\label{prop:treesOfSubsClosedClasses}
	Let $\mathcal{C}$ be a substitution-closed permutation class, and assume\footnote{Otherwise, 
		$\mathcal{C} \subseteq \{12\ldots k : k \geq 1 \}$ or 
		$\mathcal{C} \subseteq \{k\ldots 21 : k \geq 1 \}$ and these cases are trivial.} that $12,21 \in \mathcal{C}$. 
	Denote by $\mathcal{S}$ the set of simple permutations in $\mathcal{C}$. 
	The set of canonical trees encoding permutations of $\mathcal{C}$ 
	is the set of canonical trees with decorations in
	$\widehat{\SSS}:=\SSS \cup \MMM$. 
\end{proposition}

\begin{proof}
	First, if a canonical tree contains a vertex decorated by a simple permutation $\alpha \notin \mathcal{S}$,
	then the corresponding permutation $\nu$ contains the pattern $\alpha \notin \mathcal{C}$, and hence $\nu \notin \mathcal{C}$. 
	Second, by induction, all canonical trees with decorations in $\widehat{\SSS}$
	encode permutations of $\mathcal{C}$, because $\mathcal{C}$ is substitution-closed. 
	If necessary, details can be found in~\cite[Lemma 11]{albert2005simple}. 
\end{proof}

\subsection{Packed decomposition trees}
\label{sect:Packed_decomposition_trees}

From now until the end of the article we fix a substitution-closed class $\mathcal{C}$ such that $12,21\in \mathcal{C}$ and we denote with  $\mathcal{S}$ the set of simple permutations in $\mathcal{C}$. 
The assumption that we are working in $\mathcal{C}$ rather than in the set of all permutations is however often tacit: 
for example, we simply refer to canonical trees instead of canonical trees with decorations in $\widehat{\SSS}=\SSS \cup \MMM$.  We let \label{def:ct}$\mathcal{T}$ denote the collection canonical trees with decorations in $\widehat{\SSS}$, and $\mathcal{T}_{\nonp} \subset \mathcal{T}$ the subset of canonical trees with a root that is \textbf{not} labelled $\oplus$.

In this section we introduce a new family of trees called ``packed trees''
and describe a bijection between the collection $\mathcal{T}_{\nonp} \subset \mathcal{T}$ and packed trees.
Packed trees are decorated trees,
whose decorations are themselves trees.
Let us define these decorations, that we call {\em gadgets}.

\begin{definition}\label{def:S_plus_decoration}
	An $\SSS$-gadget is an $\widehat{\SSS}$-decorated tree of height at most $2$ such that: 
	\begin{itemize}
		\item The root is an internal vertex decorated by a simple permutation;
		\item The children of the root are either leaves or decorated by an increasing permutation.  
	\end{itemize}	
	The size of a gadget is its number of leaves. 
\end{definition}
We denote with $\GGG(\mathcal{S})$ the set of $\mathcal{S}$-gadgets. An example of size 7 is shown on \cref{fig:S/+-decoration}.

\begin{figure}[htbp]
	\centering
		\includegraphics[height=2.5cm]{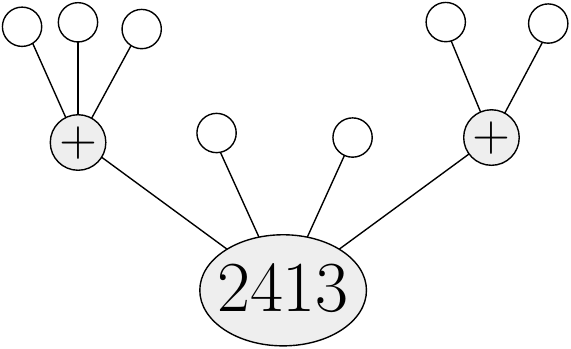}
	\caption{An $\mathcal{S}$-gadget.}
	\label{fig:S/+-decoration}
\end{figure}

Finally, let $\widehat{\GGG(\SSS)} = \GGG(\mathcal{S})\cup\{\circledast_k, k \ge 2\}$,
where, for each integer $k\ge 2$, the object $\circledast_k$ has size~$k$. 
To shorten notation,  $\widehat{\GGG(\SSS)}$ is sometimes denoted $\mathcal{Q}$ in the following. 

\begin{definition}
	An \emph{$\mathcal{S}$-packed tree} is a $\widehat{\GGG(\SSS)}$-decorated tree,
	its size being its number of leaves.
\end{definition}

\begin{remark}
	We will often refer to $\mathcal{S}$-packed trees simply as packed trees since in our analysis, the substitution-closed class $\mathcal{C}$ and 
	its set of simple permutations $\mathcal{S}$ will be fixed.
\end{remark}

An example of packed tree is shown on the right-hand side of \cref{fig:packed_tree_bij}.

\begin{remark}
	Note that in \cref{fig:packed_tree_bij} the subscript $k$ is not reported in the vertices decorated by an element in $\{\circledast_k, k \ge 2\}.$ 
	Indeed it can be easily recovered by counting the number of children of the vertex. 
\end{remark}

We now describe a bijection between canonical trees with a root that is not labelled with $\oplus$ 
and packed trees. Given a tree $T\in\mathcal{T}_{\nonp}$ the corresponding packed tree $\Pack(T)$ is obtained modifying $T$ as follows. 
\begin{itemize}
	\item For each internal vertex $v$ of $T$ labelled by a simple permutation, we build an $\mathcal{S}$-gadget $G_v$
	whose internal vertices are $v$ and the $\oplus$-children of $v$, 
	the parent-child relation in $G_v$ and the left-to-right order between children are inherited from the ones in $T$, 
	and we add leaves so that the outdegree of each internal vertex is the same in $G_v$ as in $T$.
	Then, in $\Pack(T)$, we merge $v$ and the $\oplus$-children of $v$ into a single vertex decorated by $G_v$. 
	\item The remaining vertices of $T$, decorated by $\ominus_k$ or $\oplus_k$,
	are decorated by $\circledast_k$ instead.
\end{itemize}
An example is given on \cref{fig:packed_tree_bij}.
As a preparation for the inversion procedure, let us note the following:
if a vertex $\tilde{v}$ in $\Pack(T)$
has a decoration $\circledast_k$ and his parent is decorated by an $\SSS$-gadget,
then the corresponding vertex $v$ in $T$ had decoration $\ominus_k$.
Indeed, a vertex decorated by $\oplus$ which is the child of a vertex $v$ labelled by a simple permutation is included in $G_v$,
and canonical trees do not contain $\oplus-\oplus$ edges.

\begin{figure}[htbp]
	\centering
	\includegraphics[height=7cm]{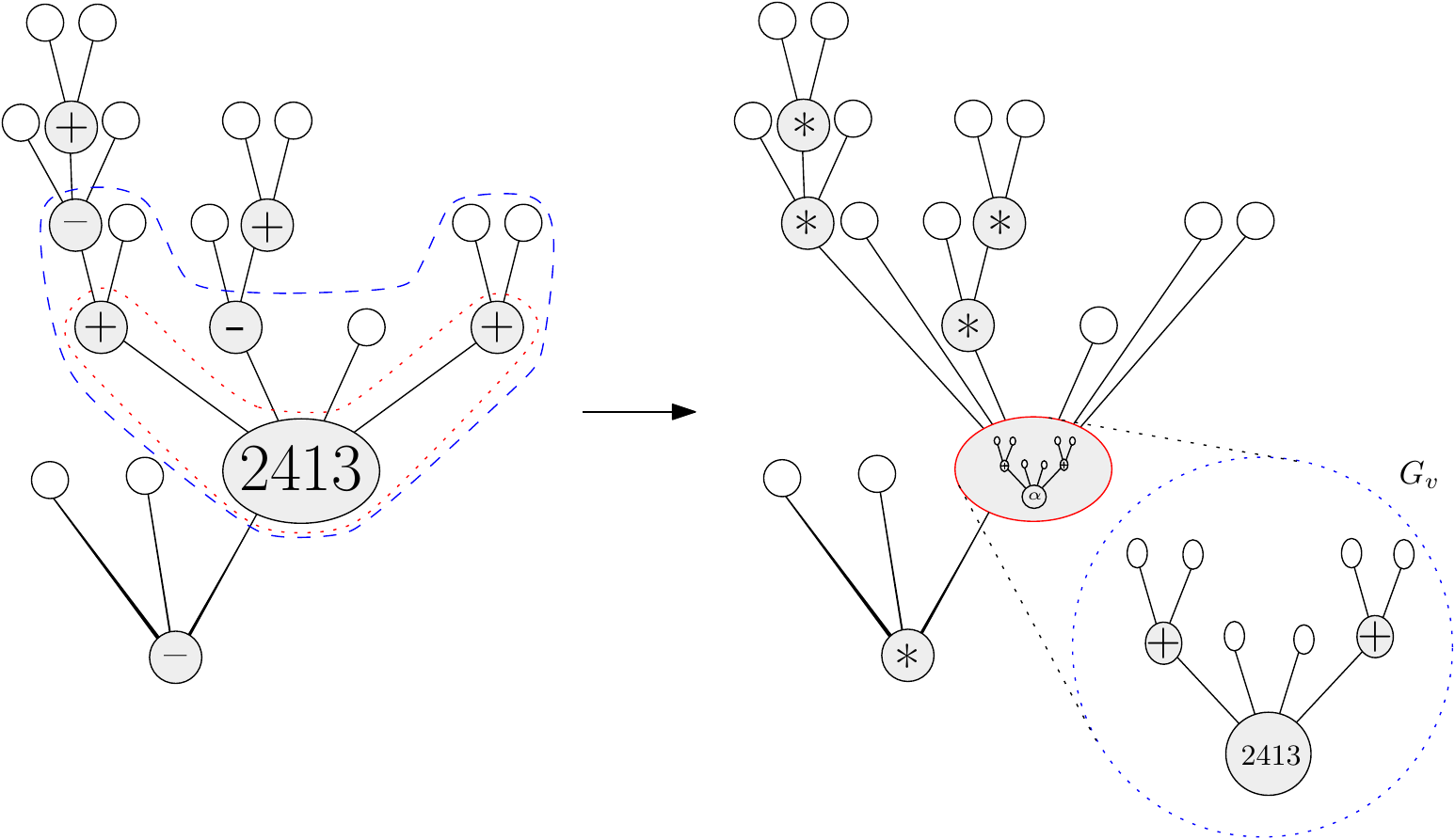}
	\caption{A canonical tree $T$ with the corresponding packed tree $\Pack(T)$. The red dotted line highlights the vertices in $T$ that are merged in the red vertex in $\Pack(T).$ The blue dashed line highlights the subtree in $T$ that determines the decoration $G_v$ in the tree $\Pack(T)$.}
	\label{fig:packed_tree_bij}
\end{figure}
\begin{proposition}\label{prop:pack_tree}
	The map $\Pack$ defines a size-preserving bijection from the set $\mathcal{T}_{\nonp}$
	of canonical trees with a root that is not labelled $\oplus$ to the set of $\SSS$-packed trees.
\end{proposition}
\begin{proof}
	We need just to show that the previous construction is invertible. 
	Given a packed tree $\PackedTree$, the corresponding tree $T$ of $\mathcal{T}_{\nonp}$ (such that $\PackedTree=\Pack(T)$) is obtained by modifying $\PackedTree$ as follows.
	\begin{itemize}
		\item For each internal vertex $\tilde{v}$ of $\PackedTree$ decorated by an $\mathcal{S}$-gadget $G$, 
		we replace $\tilde{v}$ by $G$, merging the leaves of $G$ with the children of $\tilde{v}$, respecting their order. 
		Namely, when doing this replacement, the root of the $i$-th subtree attached to $\tilde{v}$ (from left to right) 
		is merged with the $i$-th leaf of $G$ (also from left to right). 
		\item We replace each decoration $\circledast_k$ with either $\ominus_k$ or $\oplus_k$, with the following rule. 
		If $\tilde{v}$ is the root of $\PackedTree$ or the child of a vertex decorated by an $\mathcal{S}$-gadget, 
		it receives label $\ominus_k$. 
		Otherwise, if $\tilde{v}$ is the child of a vertex also decorated by some $\circledast$, then 
		we label $\tilde{v}$ in the only way that prevents the creation of 
		$\oplus - \oplus$ or $\ominus - \ominus$ edges. 
	\end{itemize}
	This shows that $\Pack$ defines a bijection.
\end{proof}

\begin{remark}
	\label{rk:Leaves_Elements}
	If $T$ (or $P$) is a tree  with $n$ leaves,
	we can label its leaves with number from $1$ to $n$
	using a depth-first traversal of the tree from left to right.
	Then the $i$-th leaf
	of the canonical or packed tree associated to a permutation
	$\nu$ corresponds to the $i$-th element in the one-line notation
	of $\nu$.
	We will use this identification between leaves and elements of the permutations
	later in the article.
\end{remark}

\subsection{Permutations are forests of decorated trees}
\label{sec:final}

Summing up the results obtained in the previous sections (in particular in \cref{prop:Can_tree,prop:treesOfSubsClosedClasses,prop:pack_tree}),
we obtain a bijective encoding of $\oplus$-indecomposable permutations in $\mathcal{C}$:

\begin{lemma}
	\label{le:bij_perm_tree}
	The map 
	\[
	\Pack\circ\CanTree: \mathcal{C}_{\nonp} \to \cP
	\]
	is a size-preserving bijection from the set $ \mathcal{C}_{\nonp}$
	of all $\oplus$-indecomposable permutations in $\mathcal{C}$
	to the set $\cP$ of all $\cS$-packed trees.
\end{lemma}

By Theorem~\ref{Th:AlbertAtkinson}, any $\oplus$-decomposable permutation corresponds uniquely to a sequence of at least two $\oplus$-indecomposable permutations.
Hence any permutation corresponds bijectively to a non-empty sequence of $\oplus$-indecomposable permutations. 
If we apply the bijection $\Pack\circ\CanTree$  to each we obtain a plane forest of packed trees. 
That is, it is an element of the collection $\Seq_{\ge 1}(\cP)$ of non-empty ordered sequences 
of $\widehat{\mathcal{G}(\mathcal{S})}$-decorated trees. 
We define the size of such a forest to be the total number of leaves.
The function that maps a permutation of $\mathcal{C}$ to the corresponding forest of packed trees is denoted by $\DF$
($\DF$ stands for decorated forest).
Summing up:

\begin{theorem}
	\label{te:bijection}
	The function 
	\[
	\DF: \mathcal{C} \to \Seq_{\ge 1}(\cP)
	\]
	is a size-preserving bijection between the substitution-closed class of permutations $\mathcal{C}$ 
	and the collection of forests of packed trees.
\end{theorem}

\subsection{Reading patterns in trees}
\label{ssec:patterns_subtrees}
Let us consider a permutation $\nu$ in $\mathcal C_{\nonp}$ and the associated canonical and packed trees:
$T=\CanTree(\nu)$ and $P=\Pack(T)$.
Let $I$ be a subset of $[n]$. 
Using \cref{rk:Leaves_Elements},
$I$ can be seen as a subset of the leaves of $T$ (or $P$).
The purpose of this section is to explain how to read out the pattern $\pi = \pat_I(\nu)$ on the trees $T$ or $P$.

Let us first note that a pattern $\pi = \pi(1) \dots \pi(k)$ is entirely determined when we know,
for each $i_1<i_2$, whether $\pi(i_1) \pi(i_2)$ forms an \emph{inversion} 
(\emph{i.e.}, an occurrence of the pattern $21$) or a non-inversion (occurrence of $12$). 
Therefore, to read patterns on $T$ (or $P$),
we should explain how to determine,
for any two leaves $\ell_1$ and $\ell_2$ of $I$, whether
the corresponding elements of $\nu$ form an inversion or not
(in the sequel, we will simply say that $\ell_1$ and $\ell_2$ form an inversion,
and not refer anymore to the corresponding elements of $\nu$). 
\medskip

Looking at $T$, this is rather easy.
We consider the closest common ancestor of $\ell_1$ and $\ell_2$, call it $v$.
By definition, $\ell_1$ and $\ell_2$ are descendants of different children of $v$, 
say the $i_1$-th and $i_2$-th.
Then the following holds: $\ell_1$ and $\ell_2$ form an inversion in $\nu$
if and only if $i_1$ and $i_2$ form an inversion in the decoration $\beta$ of $v$.
\medskip

Let us now look at $P$. We consider the closest common ancestor $u\in P$ 
of $\ell_1$ and $\ell_2$ and as before,
we assume that $\ell_1$ and $\ell_2$ are descendants of the $i_1$-th and $i_2$-th children
of $u$. 
Note that, in the packing bijection, the vertex $u$
corresponds to $v$ (the common ancestor of $\ell_1$ and $\ell_2$ in $T$)
potentially merged with other vertices.

Consider first the case that $u$ is decorated by an $\cS$-gadget $G$.
Then $G$ contains the information of the decoration of all vertices merged into $u$,
including $v$.
Therefore, whether $\ell_1$ and $\ell_2$ form an inversion in $\nu$
can be determined by looking at the $i_1$-th and $i_2$-th leaves 
of the gadget $G$ (see the example below).

If on the contrary $u$ is not decorated by an $\cS$-gadget but by a $\circledast$, 
we need to determine whether $v$ is decorated with $\oplus$ (implying that $\ell_1$ and $\ell_2$ form a non-inversion) or $\ominus$ (resp., an inversion). 

Assume first that there is a closest ancestor $u'$ of $u$ 
that is decorated with an $\cS$-gadget. 
In this case, we claim that
$v$ is decorated by $\ominus$ if $d(u,u')$ is odd, and it is decorated by $\oplus$ if $d(u,u')$ is even.
Indeed, decorations $\oplus$ and $\ominus$ alternate,
and, by construction of the packing bijection,
the vertex just above an $\cS$-gadget is decorated by a $\ominus$.

It remains to analyse the case where $u$ is decorated by $\circledast$, 
as well as all vertices on the path from $u$ to the root $r$ of $P$. 
By construction, this implies that the root of $T$ is decorated by $\ominus$. 
So, using again the alternation of $\oplus$ and $\ominus$ in $T$, 
the decoration of $v\in T$ is $\ominus$ if $d(u,r)$ is even, 
and $\oplus$ if $d(u,r)$ is odd.
\medskip

We note in particular that the pattern induced by a set $I$ of leaves in $P$
is determined by any fringe subtree containing all leaves of $I$
and {\em rooted at any vertex decorated with an $\cS$-gadget}.

\begin{example}\label{ex:reconstruction_pattern}
	Let  $\nu = 13 \, 12 \, 5\, 3\, 4\, 2\, 6\, 11\, 9\, 10\, 1\, 7\,8 $ be a permutation in $\mathcal C_{\nonp}$ with associated canonical and packed trees $T=\CanTree(\nu)$ and $P=\Pack(T)$ shown in Fig.~\ref{fig:reconstruction_pattern}. 
	We explain in the following example how to read out in $P$ the pattern induced by the leaves $\ell_1$, $\ell_2$ and $\ell_3$. 
	
	The closest common ancestor $u\in P$ of $\ell_1$ and $\ell_2$ is decorated with a $\circledast$, 
	which is at distance $1$ from its closest ancestor decorated with an $\cS$-gadget.
	We can conclude that the leaves $\ell_1$ and $\ell_2$ induce an inversion
	(the closest ancestor $v$ of $\ell_1$ and $\ell_2$ in $T$ carries a $\ominus$ decoration).
	
	Now consider $\ell_1$ and $\ell_3$. Their closest common ancestor $u'$ in $P$
	is decorated with an $\cS$-gadget. 
	Note $\ell_1$ and $\ell_3$ are descendants of the first and fifth
	children of this $\cS$-gadget; the corresponding leaves of the $\cS$-gadget
	have the vertex decorated by $2413$ as common ancestor and are attached to the branches
	corresponding to $2$ and $3$.
	We deduce that $\ell_1$ and $\ell_3$ do not form an inversion in $\nu$.
	Similarly, $\ell_2$ and $\ell_3$ do not form an inversion either in $\nu$.
	
	Putting all together, the pattern induced by $\ell_1$, $\ell_2$ and $\ell_3$ is $213$.
	Let us check that it is indeed the case, by reading this pattern on the permutation.
	These three leaves correspond to the 4th, 6th and 12th elements of the permutation respectively,
	which have values 3, 2 and 7. The induced pattern is indeed $213$.
\end{example}
\begin{figure}[htbp]
	\centering
	\includegraphics[height=6cm]{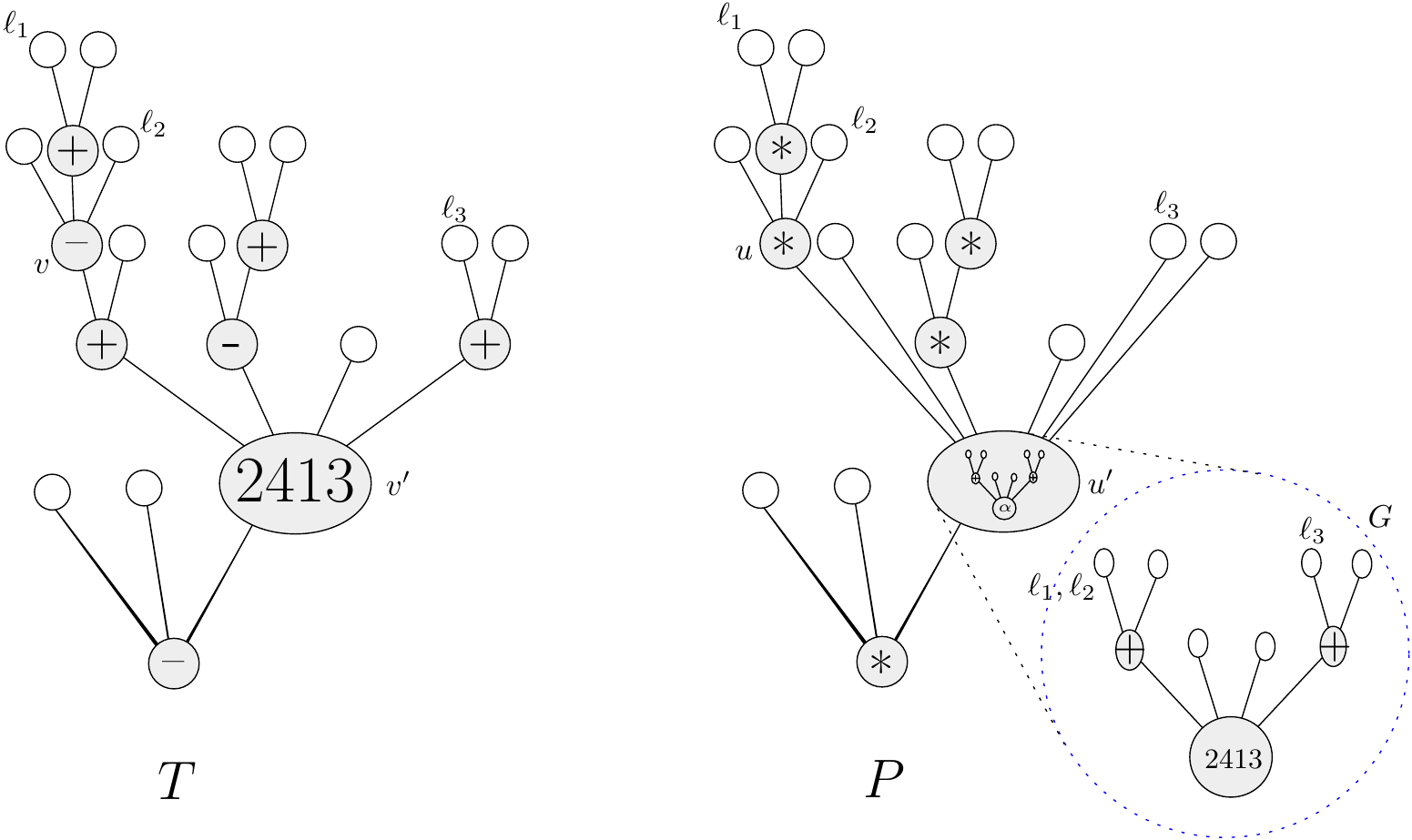}
	\caption{Reading patterns from trees -- see \cref{ex:reconstruction_pattern}}
	\label{fig:reconstruction_pattern}
\end{figure}

\section{Random permutations and conditioned Galton--Watson trees}
\label{sec:indecomposable_GW}

{\em 
	Throughout this section and the rest of the paper we assume that $\mathcal{C}$ is a \emph{proper} substitution-closed class of permutations, that is we exclude the case where $\mathcal{C}$ is the class of all permutations. To avoid trivial cases, we furthermore assume that $12, 21 \in \mathcal{C}$.}

\medskip

Theorem~\ref{te:bijection} allows us to see a uniform random  permutation $\bm{\nu}_n$ of size $n$ in the substitution-closed permutation class $\mathcal{C}$ as a uniform random forest of packed trees with $n$ leaves.\label{def:bmnu}
In the present section we apply Gibbs partition methods~\cite{stufler2016gibbs} to show that a giant component with size $n - O_p(1)$ emerges, and the small fragments admit a limit distribution. 
This goal is achieved in Proposition~\ref{prop:giant_comp_perm}. 
Since the size of the small fragments is stochastically bounded, this reduces the study of $\bm{\nu}_n$ to that of a uniform random packed tree with $n$ vertices. 
The strength of this approach is that we do not need to make any additional assumptions on the class $\mathcal{C}$. 

\subsection{Enumerative observations}
\label{sec:enum}

Theorem~\ref{te:bijection} implies that the generating series of the class $\mathcal{C}$ satisfies
\begin{align}
\mathcal{C}(z) = \frac{\cP(z)}{1- \cP(z)}.
\end{align}

From the definition of packed trees, we deduce the following equation for their generating series:
\begin{align}
\label{eq:P}
\cP(z) = z + \mathcal Q( \cP(z)),
\end{align}
where $\mathcal Q(u)=\widehat{\mathcal{G}(\mathcal{S})}(u)$ is defined as the generating function of $\widehat{\mathcal{G}(\mathcal{S})}$.
Via basic algebraic manipulations, we rewrite this as
\begin{align}
\label{eq:A}
\cP(z) = z \left( 1+ \frac{\mathcal Q( \cP(z))}{\cP(z)-\mathcal Q( \cP(z))} \right) = z \left( \frac{\cP(z)}{\cP(z)-\mathcal Q( \cP(z))} \right)
=z\, \mathcal{R}( \cP(z)),
\end{align}
with
\begin{align}
\label{eq:R}
\mathcal{R}(u) = 
\frac{1}{1 - \mathcal Q(u)/u}.
\end{align}
By definition, an $\SSS$-gadget is described by a simple permutation of size say $k$,
and $k$ elements, which are either atoms (elements of size one) or increasing permutations
of size at least two.
Therefore
\[
\mathcal{G}(\mathcal{S})(z) = \mathcal{S}\big( z + \tfrac{z^2}{1-z} \big) = \mathcal{S}\big(\tfrac{z}{1-z} \big),
\]
and consequently, 
\begin{equation}
\label{eq:G}
\mathcal{Q}(z) = \widehat{\mathcal{G}(\mathcal{S})}(z)  = \frac{z^2}{1-z} + \mathcal{S}\left(\frac{z}{1-z}\right).
\end{equation}
Since we assumed that $\mathcal{C}$ is proper, 
a celebrated result by  Marcus and Tardos~\cite{MR2063960} states
that the generating series $\mathcal{C}(z)$ has positive radius of convergence. 
Hence the same holds for $\mathcal{S}(z)$, and consequently, for $\mathcal{Q}(z)$ and $\mathcal{R}(z)$.
A general result on solutions of implicit equations (such as \eqref{eq:A})~\cite[Lem. 3.3]{stufler2016gibbs}
implies that the $n$-th coefficient $p_n$ of $\cP(z)$ satisfies the subexponentiality condition
\begin{equation}
\label{eq:subexp}
\frac{p_n}{p_{n+1}} \to \rho_{\cP} \qquad \text{and} \qquad \frac{1}{p_n} \sum_{i+j = n} p_i p_j \to 2 \cP(\rho_{\cP}) < \infty,
\end{equation}
as $n \to \infty$, with $0 < \rho_{\cP}< \infty$ denoting the radius of convergence of $\cP(z)$. This even implies
\begin{align}
\label{eq:subc}
\cP(\rho_{\cP}) < 1.
\end{align}
Indeed, if  $ 1 \le \cP(\rho_{\cP}) <\infty$, then there would exist a number $0 \le t \le \rho_{\cP}$ with $\cP(t) =1$ and hence $\widehat{\mathcal{G}(\mathcal{S})}(\cP(t)) = \infty$ by \eqref{eq:G}. But this is not possible by \cref{eq:A,eq:R}.

\cref{eq:subexp,eq:subc} allow us apply \cite[Thm. 4.8, 4.30]{MR3097424} (or \cite[Thm. 1]{MR0348393}), yielding that the number $c_n$ of $n$-sized permutations in $\mathcal{C}$ satisfies
\begin{equation}
c_n \sim \frac{p_n}{(1 - \cP(\rho_{\cP}))^{2}}.
\end{equation}

\begin{remark}
	\cref{eq:P} identifies the class $\cP$ as so-called $\mathcal{Q}$-enriched parenthesizations. A classical bijection due to Ehrenborg and Méndez  \cite{ehrenborg1994schroder} consequently allows us to identify the class~$\cP$ with the class of $\cR$-enriched trees. The recursive equation $\cP(z) = z\cR(\cP(z))$ with $\cR$ given in \cref{eq:R} is actually a consequence of this general bijection.
\end{remark}

\subsection{A giant $\oplus$-indecomposable component}

\label{subsec:Large_Indec_Component}

Let $\nu$ be a permutation in the proper substitution-closed class of permutations $\mathcal C$.
From \cref{Th:AlbertAtkinson}, we know that 
\begin{itemize}
	\item either $\nu$ is $\oplus$-indecomposable,
	\item or $\nu$ can be uniquely written as $\nu=\oplus[\nu^{(1)},\ldots,\nu^{(d)}]$, 
	where $d \ge 2$ and $\nu^{(1)},\ldots,\nu^{(d)} \in \mathcal{C}_{\nonp}$, the set of $\oplus$-indecomposable permutations of $\mathcal{C}$.
\end{itemize}
In the first case, we set $d=1$ and $\nu^{(1)}=\nu$ for convenience. Recall that \cref{le:bij_perm_tree} allows us to identify the classes $\mathcal{C}_{\nonp}$ and $\cP$. The subexponentiality condition~\eqref{eq:subexp} allows us to apply the Gibbs partition result~\cite[Thm. 3.1]{stufler2016gibbs} to obtain the following result
(only the first part will be useful in this paper, but we state the whole version for completeness):
\begin{proposition}\label{prop:giant_comp_perm}
	Let $\bm \nu_n$ be a uniform random permutation of size $n$ in $\mathcal C$ and define $\bm d$, $\bm{\nu^{(1)}}$, \ldots, $\bm{\nu^{(d)}}$ as above.
	Let $\bm{m}$ be the smallest  index such that $|\bm{\nu^{(m)}}|=\max(|\bm{\nu^{(1)}}|,\ldots, |\bm{\nu^{(d)}}|)$. Then $\bm{\nu^{(m)}}$ has size $n - O_p(1)$, and conditionally on its size,
	$\bm{\nu^{(m)}}$ is uniformly distributed among all $|\bm{\nu^{(m)}}|$-sized  $\oplus$-indecomposable permutations in $\mathcal{C}$.

	Moreover, the other components converge jointly in distribution:
	\[
	((\bm{\nu^{(1)}}, \ldots, \bm{\nu^{(\bm{m}-1)}}), ( \bm{\nu^{(\bm{m}+1)}}, \ldots, \bm{\nu^{(\bm d)}})) \convdis ((\bar{\bm \nu}_1, \ldots, \bar{\bm \nu}_{\rv{G}_1}),(\tilde{\bm \nu}_1, \ldots, \tilde{\bm \nu}_{\rv{G}_2}))
	\]
	with $\rv{G}_1, \rv{G}_2$ denoting i.i.d.\ geometric random variables with distribution
	\[
	\Proba(\rv{G}_i = k) = \cP(\rho_{\cP})^k (1 - \cP(\rho_{\cP})), \quad k \ge 0,
	\] 
	and $\bar{\bm{\nu}}_i$, $\tilde{\bm{\nu}}_i$, $i \ge 1$, denoting independent copies of a Boltzmann-distributed random object $\bm{\nu}$ with distribution given by
	\[
	\Proba(\bm{\nu} = \nu) = \rho_{\cP}^{|\nu|} / \cP(\rho_{\cP}).
	\]
\end{proposition}

\begin{remark}
	We excluded the case of uniform unrestricted $n$-sized permutations. In this case, it is well-known that the permutation is with high probability $\oplus$-indecomposable. 
	This follows for example from~\cite[Cor. 6.19]{stufler2016limits} in the tree literature or from~\cite[Thm 3.4]{PerfectSorting} in the permutation literature.
\end{remark}

\subsection{From permutations to simply generated trees}
\label{subsec:PackedTrees_GW}

Proposition~\ref{prop:giant_comp_perm} and Lemma~\ref{le:bij_perm_tree} 
reduce the study of the proper substitution-closed class $\mathcal{C}$ to the study of the class $\cP$ of packed trees.
In this section, we explain how a random tree in $\cP$ can be seen as a random simply generated tree with random decorations. 
This result may be seen as a special case of a sampling procedure~\cite[Sec. 6.4]{stufler2016limits}
for general enriched trees with a fixed number of leaves (so called enriched Schr\"oder parenthesizations),
but we present it in our specific setting to make the article more self-contained.
\medskip

We can describe a packed tree $\PackedTree$ as a pair $(T,\lambda_T)$
where $T$ is a rooted plane tree and $\lambda_T$ is a map
from the internal vertices of $T$ 
to the set $\mathcal{Q} = \widehat{\GGG(\mathcal{S})}$ which records the decorations of the vertices. 

In order to sample a uniform packed tree with $n$ leaves, we first simulate a random rooted plane tree $\bm{T}_n$ 
and then a random decoration map $\bm{\lambda}_{\bm{T}_n}$ as follows.

Define the weight-sequence $\myvec{q} = (q_k)_{k \ge 0}$,
where, for $k \ge 2$, $q_k$ denotes the $k$-th coefficient of the generating series $\mathcal Q(z)=\widehat{\GGG(\mathcal{S})}(z)$,
while we set $q_0=1$ and $q_1=0$.
We consider the simply generated tree $\bm{T}_n$ (with $n$ leaves) associated 
with weight-sequence $\myvec{q},$ \emph{i.e.,} by definition, $\bm{T}_n$ is a random rooted plane tree such that
\begin{equation}\label{eq:simply_gen_distrib}
\Proba(\bm{T}_n=T)=\frac{\prod_{v\in T} q_{d^+(v)}}{Z_n}=\frac{\prod_{v\in \Vint(T)} q_{d^+(v)}}{Z_n},
\end{equation}
for all rooted plane trees $T$ with $n$ leaves (we recall that $\Vint(T)$
denotes the set of internal vertices of $T$). 
Here, $Z_n$ is the \emph{partition function} given by 
\[
Z_n=\sum_{T}\prod_{v\in T} q_{d^+(v)},
\] 
where the sum runs over all rooted plane trees with $n$ leaves.
For a general introduction about simply generated trees see \cite[Section 2.3]{janson2012simply}.

Then, given a rooted plane tree $T$, let $\bm{\lambda}_T$ be the random map such that for all internal vertices $v$ of~$T$,
\begin{align}\label{eq:decoration_distrib}
\Proba(\bm{\lambda}_T(v)=Q)=\frac{1}{q_{d^+_T(v)}} \quad\text{for all $Q \in \mathcal{Q}$ with $|Q| = d^+_T(v)$},
\end{align}
independently of all other choices. Namely, the decoration of each internal vertex $v$ of $T$ gets drawn uniformly at random among all $d^+_T(v)$-sized decorations in $\mathcal{Q}$, independently of all the other decorations. 

\begin{lemma}
	\label{lem:unif_packed_tree} \label{def:Pn}
	The random packed tree $\bm{P}_n = (\bm{T}_n,\bm{\lambda}_{\bm{T}_n})$ is uniform among all the packed trees with $n$ leaves.
\end{lemma}
\begin{proof}
	Let $\PackedTree=(T,\lambda)$ be a packed tree with $n$ vertices. Then
	\begin{equation}
	\begin{split}
	\Proba\big((\bm{T}_n,\bm{\lambda}_{\bm{T}_n})=(T,\lambda)\big)&=\Proba\big((\bm{T}_n,\bm{\lambda}_{\bm{T}_n})=(T,\lambda)\big|\bm{T}_n=T\big)\cdot\Proba(\bm{T}_n=T)\\
	&=\left(\prod_{v\in \Vint(T)}\frac{1}{q_{d^+_T(v)}}\right)\cdot\left(\frac{\prod_{v\in T}q_{d^+_T(v)}}{Z_n}\right)=\frac{1}{Z_n},
	\end{split}
	\end{equation}
	where in the second equality we use \cref{eq:decoration_distrib,eq:simply_gen_distrib}. 
\end{proof}

\subsection{Random packed trees as conditioned Galton--Watson trees}

Building on \cref{lem:unif_packed_tree}, in what follows we explain how to sample a uniform packed tree with $n$ leaves
as a randomly decorated  Galton--Watson tree conditioned on having $n$ leaves.
Again, we refer to~\cite[Sec. 6.4]{stufler2016limits}
for a discussion in a more general context.
\medskip

Let $\rho_q$ denote the radius of convergence of the generating series ${\mathcal Q}(z)$.
As observed in Section~\ref{sec:enum}, it holds that $\rho_q > 0$.
As we shall see, this implies that $\bm{T}_n$ has the distribution of a Galton--Watson tree conditioned of having $n$
leaves, whose offspring distribution $\xi$ is defined below
(for similar discussion with fixed number of vertices,
see \cite[Section 4]{janson2012simply}).

The offspring distribution $\xi$ is given by 
\begin{align}
\label{eq:offspring_distribution_packed_tree}
\begin{cases}
\Proba(\xi = 0) = a\\
\Proba(\xi = 1) = 0\\
\Proba(\xi = k) = q_k t_0^{k-1} \text{ for } k  \geq 2.
\end{cases}
\end{align}
with $a, t_0>0$ constants that are defined as follows. 
If $\lim_{z \nearrow \rho_q} {\mathcal Q}'(z) \ge 1$, let $0<t_0\le \rho_q$ be the unique number with ${\mathcal Q}'(t_0) = 1$. 
If the limit is less than $1$, then set $t_0 = \rho_q$.
Finally set $a = 1 - \sum_{k \ge 2} q_k t_0^{k-1} >0$.

Note that the tilting in \cref{eq:offspring_distribution_packed_tree} previously appeared in \cite[Proposition 2]{MR3378819} (see also the discussion above Corollary 1 in the same paper).

We note that $\xi$ is always aperiodic since $q_k>0$ for $k \ge 2$ (because of the $\circledast$ decorations).
Moreover, we have
\begin{align}
\E[\xi] = \mathcal Q'(t_0) \le 1,
\end{align}
so that the Galton--Watson tree $\bm T^\xi$ of offspring distribution $\xi$ is either subcritical or critical.
It is a simple exercise to check that $\bm T^\xi$, conditioned on having $n$ leaves,
has the same distribution as the simply generated tree $\bm{T}_n$ defined by \cref{eq:simply_gen_distrib}.

To end this section, we characterize when this Galton--Watson tree model is critical.
Below, we write $\cS'(\rho_\cS)$ for $\lim_{z \nearrow \rho_\cS} \cS'(z)$, noting that this limit may be infinite. 

\begin{proposition}
	\label{prop: offspring_distr_charact}
	It holds that $\E[\xi] = 1$ if and only if
	\begin{align}
	\label{eq:type1}
	\cS'(\rho_\cS) \ge \frac{2}{(1 +\rho_\cS)^2} -1.
	\end{align}
	In this case, $t_0 = \kappa/(1+\kappa)$ for the unique number $0  < \kappa \le \rho_\cS$ with $\cS'(\kappa) = 2/(1+\kappa)^2 -1$, and
	\begin{align}
	\V[\xi] = \kappa (1+\kappa)^3 \cS''(\kappa) + 4 \kappa.
	\end{align}
\end{proposition}
For the convenience of the reader, we note that the relation between $t_0$ and $\kappa$ can be rewritten as $\kappa = \frac{t_0}{1-t_0}$.

\begin{proof}
	It holds that
	\begin{align*}
	\mathcal{Q}'(z) = \frac{\cS'(z/(1-z)) + z(2-z)}{(1-z)^{2}}
	\end{align*}
	We perform the formally substitution $z=y/(1+y)$ (which implies $z = \rho_q \Leftrightarrow y=\rho_{\cS}$). This yields
	\begin{equation}
	\label{eq:Qp}
	\mathcal{Q}'(z) = (1+y)^2 \cS'(y) + y^2 + 2y=1+(1+y)^2\left[ \cS'(y) +1 -\frac{2}{(1+y)^2} \right].
	\end{equation}
	Recall that $\E[\xi]=1$ if and only if $\lim_{z \nearrow \rho_q} {\mathcal Q}'(z) \ge 1$.
	Since $\rho_q=\frac{\rho_{\mathcal S}}{1+\rho_{\mathcal S}}$, this shows the first part of the statement.
	The formula for $t_0$ also follows from \eqref{eq:Qp} and the definition of $t_0$.
	Finally, if $\E[\xi]=1$, then 
	\begin{align*}
	\V[\xi] 	&= \E[\xi(\xi-1)] = t_0 \mathcal{Q}''(t_0)  \\
	&= \kappa (1 + \kappa)^3 \cS''(\kappa) + 2 \kappa (1 + \kappa)^2 (\cS'(\kappa) +1) \\
	&= \kappa (1+\kappa)^3 \cS''(\kappa) + 4 \kappa.\qedhere
	\end{align*}
\end{proof}

%
%
%
%
%


\section{Semi-local convergence of the skeleton decomposition}
\label{sec:skeleton}

The previous section establishes a connection between uniform random permutations
and conditioned Galton-Watson trees.
In this section, we provide a convergence result for skeletons induced by marked vertices in such trees.
The application to permutations will be discussed in further sections.\medskip

Aldous~\cite[Eq. (49)]{MR1207226} showed that 
the subtree spanned by a fixed number of random marked vertices 
in a large critical Galton--Watson tree admits a limit distribution. 
Here, we extend this \emph{skeleton decomposition} 
so that it additionally describes the asymptotic  local structure in $o(\sqrt{n})$-neighbourhoods around the marked vertices
and their pairwise closest common ancestors.
Note also that Aldous works with Galton--Watson trees conditioned
on having $n$ vertices,
while we more generally consider
Galton--Watson trees conditioned
on having $n$ vertices with out-degree in a given set $\Omega$
(see \cite{MR3335013} or \cite{MR2946438} for scaling limit results under such conditioning). 

\subsection{Extracting the skeleton with a local structure}

Let $k \ge 1$ denote a fixed integer and $T$ a (rooted) plane tree. 
We choose an ordered sequence $\myvec{v} = (v_1, \ldots, v_k)$ of vertices in $T$
(possibly with repetitions) that we call \emph{marked vertices}.
The goal of this section is to associate to this data some
object recording:
\begin{itemize}
	\item the genealogy between the marked vertices;
	\item the local structure around the \emph{essential vertices} of $T$, 
	which we define as the root of $T$, the marked vertices $v_1, \ldots, v_k$ and their pairwise closest common ancestors;
	\item the distances in the original tree between these vertices.
\end{itemize}
The reader can look at \cref{fig:skeleton} to see the different steps of the construction.

\begin{figure}[htbp]
	\centering
	\includegraphics[width=10cm]{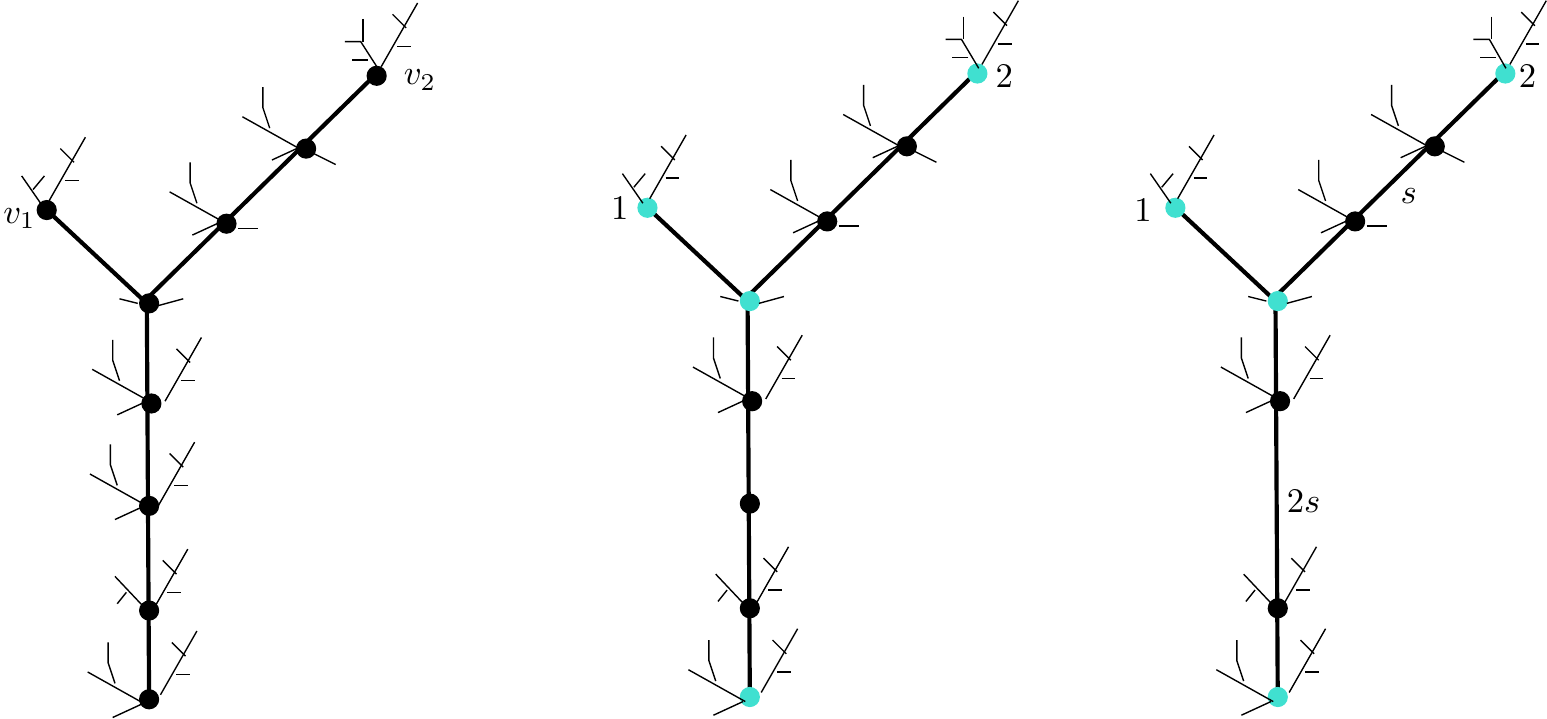}
	\caption{A tree $T$ with two marked vertices $v_1$ and $v_2$. 
		In the left-most picture, the subtree $R(T, \myvec{v})$ is represented in bold,
		while branches attached to its corner are drawn with thinner lines.
		The middle picture represent $R^{[1]}(T, \myvec{v})$:
		the essential vertices are in blue, and only one vertex of $R(T, \myvec{v})$
		is at distance more than 1 from the closest essential vertex.
		The branches attached to that vertex do not belong to $R^{[1]}(T, \myvec{v})$.
		The right-most picture represent $s.R^{[1]}(T, \myvec{v})$.
		In particular, observe that the two middle edges of the path 
		between the root and the branching vertex
		have been contracted into a single edge with label $2s$.}
	\label{fig:skeleton}
\end{figure}

The first step is to consider
the subtree $R(T, \myvec{v})$ consisting of the vertices $\myvec{v}$ and all their ancestors.
For each $1 \le j \le k$ the vertex $v_j$ in $R(T, \myvec{v})$ receives the label $j$. 
Note that the tree $T$ may be constructed from the skeleton $R(T, \myvec{v})$ by attaching an ordered sequence of branches (rooted plane trees) at each corner of $R(T, \myvec{v})$.  Here we have to consider the corner below the root-vertex twice, since branches at this corner may either be added to the left or to the right of $R(T, \myvec{v})$. 

The second step is to remove the vertices of $T$ 
which lie outside of the skeleton $R(T, \myvec{v})$ 
and are ``far'' from the essential vertices. 
For convenience, we call distance of any branch $B$ (grafted on $R(T, \myvec{v})$) 
from a vertex  $w \in R(T, \myvec{v})$ the distance in $R(T, \myvec{v})$ 
from $w$ to the corner where $B$ is attached. %
For any integer $t \ge 0,$ we let $R^{[t]}(T, \myvec{v})$ denote the subtree of $T$ that contains $R(T, \myvec{v})$ and all branches grafted on $R(T, \myvec{v})$ that have distance at most $t$ from at least one essential vertex. In particular, $R^{[t]}(T, \myvec{v})$ contains all vertices of  $T$ that lie at distance at most $t$ from the essential vertices.

The final step of the construction is to shrink the paths of $R^{[t]}(T, \myvec{v})$ 
consisting of the vertices whose attached branches have been removed in step 2. 
Indeed, we are interested in a scenario where the distance between any two  essential vertices is much larger than $2t$. Consider two essential  points $x \ne y$ that are connected by a path not containing other essential vertices. Assume that $x$ lies on the path from the root to $y$. If the distance between $x$ and $y$ is larger than $2t$, then the path joining $x$ and $y$ consists of a starting segment of length $t$ that starts at $x$, a \emph{middle segment} of positive length, and an end segment of length $t$ that ends at $y$.
By construction, the branches attached to inner vertices of the middle segment of $R(T,\myvec{v})$
do not appear in $R^{[t]}(T,\myvec{v})$.
For any real number $s>0$, we let $s.R^{[t]}(T, \myvec{v})$ 
denote the result of contracting each middle segment 
to a single edge that receives a label given 
by the product of $s$ and the number of deleted vertices in this segment. %

\subsection{The space of skeletons with a local structure}
\label{ssec:skeleton_space}

In the following, we will need to be more precise about 
the space in which $s.R^{[t]}(T, \myvec{v})$ lives and the topology we consider on it.
In the above construction, $s.R^{[t]}(T, \myvec{v})$ is a tree with $k$ distinguished vertices
with outdegree in $\Omega$,
where at most $2k-1$ edges have a (length-)label. 
Moreover, the distances between successive essential vertices
are at most $2t+1$ (we say that two essential vertices are successive if the path going from one to the other
does not contain any other essential vertex). 
The set of trees (without edge-labels) with $k$ marked distinguished vertices with outdegree in $\Omega$
such that the above distance condition holds is denoted $\setTkt$.
Moreover, we say that $G$ in $\setTkt$ is {\em generic} if:
\begin{itemize}
	\item there are $2k$ distinct essential vertices (the root, the $k$ distinguished vertices and $k-1$ closest
	common ancestors of pairs of distinguished vertices);
	\item the distances between successive essential vertices 
	are exactly $2t+1$.
\end{itemize}

We note that the edges with (length-)label
are middle edges of the paths of length $2t+1$ between essential vertices,
and hence depend only on the shape of the tree.
We can therefore encode these labels as a vector in $\RR^{2k-1}$, that has entries equal to $0$ whenever the corresponding essential vertices are at distance $2t$ or less.
Finally, $s.R^{[t]}(T, \myvec{v})$ can be seen as an element of
\[ \setTkt \times (\RR_+)^{2k-1}.\]
(A similar identification is done by Aldous throughout the article \cite{MR1207226}.)

Using the discrete topology on $\setTkt$ and the usual one on $\RR^{2k-1}$,
this gives a topology on $\setTkt \times (\RR_+)^{2k-1}$, 
and then it makes sense to speak of convergence in distribution in this space.
We can also speak of {\em density}, taking as reference measure
the product of the counting measure on $\setTkt$ and the Lebesgue measure on $(\RR_+)^{2k-1}$.
Finally we denote by $\Sh$ and $\Lab$ the natural projections
from $\setTkt \times (\RR_+)^{2k-1}$ to $\setTkt$ and $(\RR_+)^{2k-1}$, respectively.
In words $\Sh$ erases the labels and output the {\em shape} of the tree,
while $\Lab$ outputs the vector of labels.

\subsection{The limit tree}
\label{sec:limit_tree}

Throughout Section~\ref{sec:skeleton} we let $\bm T$ denote a (non-degenerate) critical Galton--Watson tree having an aperiodic offspring distribution $\xi$.
We also assume that $\xi$ has finite variance $\sigma^2$.
We fix a subset $\Omega \subseteq \mathbb{N}_0$ satisfying
\begin{align}
\Proba(\xi \in \Omega) > 0.
\end{align}
Given a rooted plane tree $T$, we let $|T|_\Omega$ denote the number of vertices $v \in T$ that have outdegree $d_T^+(v) \in \Omega$. 
For any value $n$ that the number $|\bm T|_\Omega$  can have with positive probability,
we let $\bm T_n^\Omega$ denote the result of conditioning the tree $\bm T$ on $|\bm T|_\Omega = n$.
The goal is to describe the limit of $R^{[t]}(\bm T_n^\Omega,\myvec{\rv{v}})$,
where $\myvec{\rv{v}}=(\bm v_1,\dots,\bm v_k)$ are independently and uniformly chosen vertices of $\bm T_n^\Omega$,
{\em conditioned} to have outdegree in $\Omega$.
\medskip

We first recall the definition of simply and doubly size-biased versions of $\xi$,
namely the random variables $\hat{\xi}$ and $\xi^*$ with distributions
\begin{align}
\label{eq:xihat}
\Proba(\hat{\xi} = i) &= i \Proba(\xi = i), \\
\label{eq:xistar}
\Proba(\xi^* = i) &= i(i-1)\Proba(\xi = i) / \sigma^2.
\end{align}

Furthermore, for any fixed integer $k \ge 1$ we say a \emph{proper $k$-tree} is a (rooted) plane tree that has precisely $k$ leaves, labelled from $1$ to $k$, such that the root has outdegree $1$ and all other internal vertices have outdegree $2$.  Note that each such tree has $2k -1$ edges
and that there are $k! \mathsf{Cat}_{k-1}= 2^{k-1} \prod_{i=1}^{k-1}(2i-1)$ such trees.
Indeed, up to the single edge attached to the root, these trees are complete binary trees with $k$ leaves and a labelling of these leaves.
In the following, we order the edges of proper $k$-tree in some canonical order (e.g. depth first search order),
so that we can speak of the $i$-th edge of the tree; the chosen order is not relevant though. 

\begin{figure}[htbp]
	\centering
	\includegraphics[height=7.5cm]{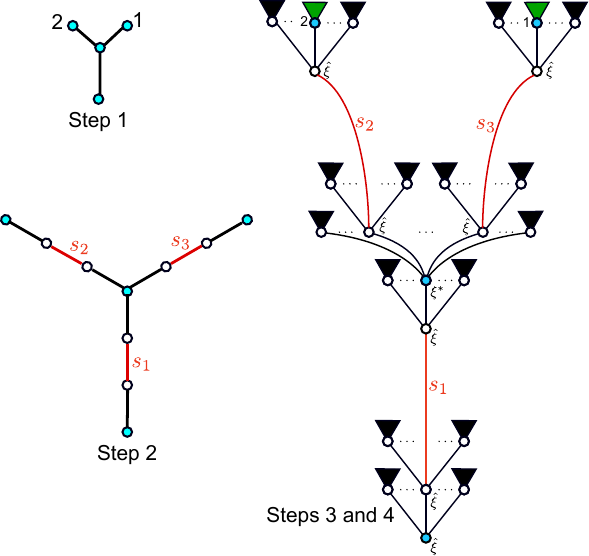}
	\caption{The construction of the limit tree ${\bm T}^{k,t}_\Omega$ for $k=2$ and $t=1$. The essential vertices are coloured blue, and the middle edges are coloured red. Each occurrence of $\hat{\xi}$ or $\xi^*$ at the side of a vertex represents that this vertex receives offspring according to an independent copy of the corresponding random variable (step 3).
		Each black triangle represents an independent copy of the Galton--Watson tree $\bm T$ (step 4).
		The green triangles represent independent copies of $\bm T$ conditioned on having root degree in $\Omega$
		(step 4). }
	\label{fig:ktree}
\end{figure}

For each integer $t \ge 1$ we can now construct a random rooted plane tree
${\bm T}^{k,t}_\Omega$ with $k$ distinguished vertices labelled from $1$ to $k$ having outdegree in $\Omega$, 
and $2k-1$ edges having length-labels.
We will prove later that this tree is the limit of $R^{[t]}(\bm T_n^\Omega,\myvec{\rv{v}})$.
A special case of this construction is illustrated in \cref{fig:ktree}. 
The general procedure goes as follows:
\begin{enumerate}[1.]
	\item \emph{(Pick a skeleton)} Draw a proper $k$-tree uniformly at random. 
	Its leaves will correspond to the distinguished labelled vertices of ${\bm T}^{k,t}_\Omega$.
	Each possible outcome of this step is attained with probability
	\[
	\frac{1}{2^{k-1} \prod_{i=1}^{k-1}(2i-1)}.
	\]
	\item \emph{(Stretch it)} Select a vector $\myvec{\rv{s}} = (s_i)_i \in \mathbb{R}_{>0}^{2k-1}$ at random with density 
	\begin{align}
	\label{eq:density}
	\textstyle  (3 \cdot 5 \cdots (2k-3)) \, (\sum_i s_i) \, \exp \left( - \tfrac{(\sum_i s_i)^2}{2} \right).
	\end{align}
	It is easy to check that this defines a probability distribution,
	using classical expressions for absolute moments of Gaussian distribution.
	For each $1 \le i \le 2k-1$, we replace the $i$-th edge of the $k$-tree by a path of length $2t+1$
	and assign label $s_i$ to the central edge of this path. 
	\item \emph{(Thicken it)} Each internal vertex receives additional offspring, independently from the rest. Here vertices with outdegree $1$ receive additional offspring according to an independent copy of $\hat{\xi} -1$, 
	while vertices with outdegree $2$ receive additional offspring according to an independent copy of $\xi^*-2$. An ordering of the total offspring that respects the ordering of the pre-existing offspring is chosen uniformly at random.
	\item \emph{(Graft branches)}
	Each distinguished vertex (\emph{i.e.,} each leaf of the original $k$-tree)
	becomes the root of an independent copy of $\bm T$ conditioned on having root-degree in $\Omega$.
	Other leaves of the tree resulting from step 3 
	become the roots of independent copies of Galton--Watson trees $\bm T$, 
	without conditioning.
\end{enumerate}

\begin{lemma}
	Seen as an element in $\setTkt \times (\RR_+)^{2k-1}$,
	the random tree ${\bm T}^{k,t}_\Omega$ has density 
	\[ f(G,\myvec u) = p_G \, h(\myvec u)\,  \One[G \text{ generic}],\]
	where, for a generic $G$ in $\setTkt$ and $\myvec u$ in $(\RR_+)^{2k-1}$, we have
	\begin{align}
	p_G&:=  \frac{\Proba(\xi \in \Omega)^{-k}}{\sigma^{2k-2} \, \prod_{i=1}^{k-1}(2i-1)}
	\prod_{v \in V_G} \Proba(\xi= d^+_G(v)); \label{eq:def_pG}\\
	h(\myvec u)&:=
	\textstyle \left[ \prod_{i=1}^{k-1}(2i-1) \right] \, (\sum_i u_i) \, \exp \left( - \tfrac{(\sum_i u_i)^2}{2} \right).\label{eq:def_h}
	\end{align}
	\label{lem:distrib_Tkt}
\end{lemma}
\begin{proof}
	By construction, it is clear that $\Sh({\bm T}^{k,t}_\Omega)$ and $\Lab({\bm T}^{k,t}_\Omega)$
	are independent, that $\Sh({\bm T}^{k,t}_\Omega)$ is generic
	and that $\Lab({\bm T}^{k,t}_\Omega)$ has density $h$.
	We only need to prove that, for any generic $G$ in $\setTkt$,
	$\Proba\big[ \Sh(\bm T^{k,t}_\Omega)=G\big] = p_G$.
	
	We follow the construction of $\bm T^{k,t}_\Omega$.
	The event $\Sh(\bm{T}^{k,t}_\Omega) = G$ holds
	if and only if the following events occur.
	\begin{itemize}
		\item At step 1, we choose the proper $k$-tree corresponding to the genealogy of the distinguished vertices of $G$:
		this happens with probability 
		\[\frac{1}{2^{k-1} \prod_{i=1}^{k-1}(2i-1)}.\]
		\item If we choose the correct proper $k$-tree,
		after step 2, the vertices of the resulting tree correspond to the
		vertices of $R(G ,\myvec{u})$.
		Then, at step 3, we need to choose for each of them the correct number of children
		and the correct ordering of these children.
		For a branching vertex of outdegree $d$ in $R(G ,\myvec{u})$,
		the correct number of children is chosen with probability $\Proba(\xi^* = d)$
		and the correct ordering with probability $\binom{d}{2}^{-1}$.
		Multiplying these probabilities gives 
		\[\frac{\Proba(\xi^* = d)}{\binom{d}{2}} = \frac{2}{\sigma^2} \Proba(\xi= d).\]
		For a non-branching internal vertex of outdegree $d$ in $R(G ,\myvec{u})$,    
		the correct number of children is chosen with probability $\Proba(\hat{\xi} = d)$
		and the correct ordering with probability $d^{-1}$.
		Again, multiplying these two, we get
		\[ \frac{\Proba(\hat{\xi} = d)}{d} = \Proba(\xi= d).\]
		\item In step 4 of the construction, we need to choose copies of $\bm T$ or $\bm T$ conditioned 
		to have root outdegree in $\Omega$ (the black and green triangles in \cref{fig:ktree}) corresponding to that in $G$.
		The probability of this event is given as a product as follows.
		For each distinguished vertex $v$ of outdegree $d$, we have a factor $\Proba(\xi=d)/\Proba(\xi \in \Omega)$ (the denominator comes from the conditioning
		that the outdegree of such vertex is in $\Omega$).
		For vertices in $G \setminus R(G ,\myvec{u})$ of outdegree $d$, we have a factor $\Proba(\xi=d)$.
	\end{itemize}
	Summing up, since there are $k-1$ branching vertices and $k$ distinguished ones, we get that
	\begin{equation}
	\Proba\big[ \Sh(\bm T^{k,t}_\Omega)=G\big] =  \frac{1}{2^{k-1} \prod_{i=1}^{k-1}(2i-1)} \frac{2^{k-1}}{\sigma^{2k-2}} \Proba(\xi \in \Omega)^{-k}
	\prod_{v \in V_G} \Proba(\xi= d^+_G(v)),
	\label{eq:Proba_LimitG}
	\end{equation}
	and the factors $2^{k-1}$ in the numerator and denominator cancel out.
\end{proof}

\subsection{Convergence}

\label{sec:semiconv}

The following lemma extends Aldous' skeleton decomposition~\cite[Eq. (49)]{MR1207226} by keeping track of $o(\sqrt{n})$-neighbourhoods 
near the essential vertices of the skeleton. 
The $o(\sqrt{n})$-threshold is sharp 
(for the applications in this paper,
the convergence of $t_n$-neighbourhoods for any sequence $t_n$ 
tending to infinity would suffice).
We note that $o(\sqrt{n})$-neighbourhoods of the root
have been previously considered in the literature,
{\em e.g.}\ by Aldous \cite{MR1085326,MR1166406}
and Kersting \cite{Kersting};
see also \cite[Theorem 5.2]{stufler2019}
for a result on the $o(\sqrt{n})$-neighbourhood
of a uniform random vertex in the tree.
Besides,
Lemma~\ref{le:semilocal} is also related to scaling limits obtained by Kortchemski~\cite{MR2946438}
and Rizzolo~\cite{MR3335013}, that imply  convergence of $R(\bm T_n^\Omega, \myvec{\rv{v}})$.

We recall that we see trees of the from $s.R^{[t]}(T,\myvec{u})$ and ${\bm T}^{k, t}_\Omega$
as elements of the space $\setTkt \times (\RR_+)^{2k-1}$ as explained in \cref{ssec:skeleton_space}.
\begin{lemma}
	\label{le:semilocal}
	Suppose that the offspring distribution $\xi$ is critical, aperiodic, and has finite variance $\sigma^2$. Let $\myvec{\rv{v}}$ be a vector of $k \ge 1$ independently and uniformly selected vertices with outdegree in $\Omega$ of the conditioned tree $\bm T_n^\Omega$. Then for each constant positive integer $t$ it holds that
	\begin{align}
	\label{eq:weaker}
	c_\Omega \sigma n^{-1/2}. R^{[t]}(\bm T_n^\Omega, \myvec{\rv{v}})	\convdis {\bm T}^{k, t}_\Omega
	\end{align}
	with $c_\Omega= \sqrt{\Proba(\xi \in \Omega)}$. Even stronger, for each sequence $t_n = o(\sqrt{n})$ of positive integers it holds that
	\begin{align}
	\label{eq:toshow}
	\sup_{A,B} \left| \Proba\big[ c_\Omega \sigma n^{-1/2}.R^{[t_n]}(\bm T_n^\Omega, \myvec{\rv{v}})
	\in A \times B\big] - \Proba\big[{\bm T}^{k, t_n}_\Omega \in A \times B\big] \right| \to 0,
	\end{align}
	with $A$ ranging over all subsets of $\mathcal T_{k,\Omega}^{[t_n]}$,
	and $B$ over open intervals of $(\RR_+)^{2k-1}$.
\end{lemma}
\begin{proof}
	We fix some sequence $t_n$ and let, for each $n$,
	$(G,\myvec{x})$ be an element in $\setTktn \times (\RR_+)^{2k-1}$,
	with $G$ generic and $\myvec{x}$ taking integer coordinates.
	
	We also fix constants $b>a>0$ and a sequence $s_n$ with $s_n=o(n)$. 
	The core of the proof consists in establishing
	that, as $n \to \infty$, we have
	\begin{align}
	\label{eq:almostthere}
	\Proba\big[ 1.R^{[t_n]}(\bm T_n^\Omega, \myvec{\rv{v}}) = (G, \myvec{x}) \big] 
	\sim \left( \frac{\sigma  c_\Omega}{\sqrt{n}}\right)^{2k-1}  
	h\left(\sigma c_\Omega n^{-1/2} \myvec{x}\right) p_G,
	\end{align}
	uniformly on pairs $(G, \myvec{x})$  
	such that $\sum_{i=1}^{2k-1} x_i$ is in $[a\sqrt{n},b\sqrt{n}]$ and $|G|_{\Omega} \le s_n$
	(recall that $p_G$ is defined in \eqref{eq:def_pG} and $h(\cdot)$ in \eqref{eq:def_h}).
	
	Assume temporarily \eqref{eq:almostthere}.
	Summing over all possible values of $\myvec{x}$ (such that $\sum_{i=1}^{2k-1} x_i$ lies in $[a\sqrt{n},b\sqrt{n}]$,
	and making $a$ go to $0$, $b$ go to $+\infty$), we have
	\begin{equation}
	\Proba\big[\Sh(1.R^{[t_n]}(\bm T_n^\Omega, \myvec{\rv{v}})) =G \big]=
	\Proba\big[\Sh(c_\Omega \sigma n^{-1/2} .R^{[t_n]}(\bm T_n^\Omega, \myvec{\rv{v}}))=G \big]
	\sim p_G,
	\label{eq:Sh_pG}
	\end{equation}
	uniformly on trees $G$ with $|G|_{\Omega} \le s_n$.
	Moreover, conditionally on the shape of this skeleton being $G$,
	\cref{eq:almostthere} gives a local limit theorem for $\Lab(1.R^{[t_n]}(\bm T_n^\Omega, \myvec{\rv{v}}))$
	with scaling factor $c_\Omega \sigma n^{-1/2}$ and limiting distribution with density $h$.
	This implies convergence in distribution of 
	$\Lab(c_\Omega \sigma n^{-1/2} .R^{[t_n]}(\bm T_n^\Omega, \myvec{\rv{v}}))$
	to a random variable of density $h$.
	Comparing with \cref{lem:distrib_Tkt}, we see that \cref{eq:almostthere} implies 
	\cref{eq:weaker}.
	Proving \cref{eq:toshow} needs an extra ingredient and we come back to it at the end of the proof.
	\medskip
	
	To prove \cref{eq:almostthere}, we need some additional notation.
	First, we write $\ell=\sum_{i=1}^{2k-1} x_i$.
	Additionally, we let $(\rv{X}_i, \hat{\xi}_i)_{i \ge 1}$ be independent copies of $|\bm T|_\Omega$ and $\hat{\xi}$.
	Finally, we also set 
	\begin{equation}
	\rv{S}_m \coloneqq \rv{X}_1 + \ldots + \rv{X}_m \text{ (for $m \ge 0$)}, \quad
	\rv{Q} \coloneqq \sum_{i=1}^{\ell} \mathbb{1}_{\hat{\xi}_i \in \Omega}
	\  \text{ and }\ \rv{L} := \sum_{i=1}^{\ell}(\hat{\xi}_i -1).
	\label{eq:notation}
	\end{equation}
	
	The proof of \cref{eq:almostthere} is splitted in two parts,
	respectively of combinatorial and analytic nature.
	The combinatorial part shows that
	\begin{multline}
	\Proba\big[ 1.R^{[t_n]}(\bm T_n^\Omega, \myvec{\rv{v}}) =  (G, \myvec{x}) \big] 
	= n^{-k} \left( \prod_{v \in V_G} \Proba(\xi= d^+_G(v)) \right) 
	\frac{\Proba\left(\rv{S}_{ \rv{L}} = n - |G|_\Omega - \rv{Q} \right)}{\Proba(|\bm T|_\Omega =n)}
	\label{eq:distrib_Rtn}\\
	=    n^{-k} c_\Omega^{2k} \sigma^{2k-2} \left( \prod_{i=1}^{k-1}(2i-1)\right) p_G  \frac{\Proba\left(\rv{S}_{ \rv{L}} = n - |G|_\Omega - \rv{Q} \right)}{\Proba(|\bm T|_\Omega =n)}.
	\end{multline}
	We do this by decomposing combinatorially pairs $(T_\star, \myvec v_\star)$ (i.e. trees with distinguished vertices)
	such that $1.R^{[t_n]}(T_\star, \myvec v_\star)=(G, \myvec{x})$.
	
	The analytic part, based on a standard local limit lemma,
	then analyzes the numerator of the last factor and shows that
	\begin{align}
	\label{eq:corpsereviver2}
	\Proba\left(\rv{S}_{ \rv{L}} = n - |G|_\Omega - \rv{Q}\right ) \sim (2 \pi)^{-1/2} \sigma  c_\Omega \ell n^{-3/2} \exp \left( - \frac{\sigma^2 \Proba(\xi \in \Omega) \ell^2}{ 2n}\right)
	\end{align}
	uniformly on integers $\ell$ in $[a\sqrt{n},b\sqrt{n}]$, and on trees $G$ with $|G|_\Omega \le s_n$.
	
	Finally, an estimate for the denominator in \eqref{eq:distrib_Rtn}  
	is given e.g. in \cite[Thm. 8.1]{MR2946438}:
	\begin{align}
	\label{eq:partitionfunction}
	\Proba(|\bm T|_\Omega = n) \sim \tfrac{c_\Omega}{\sqrt{2 \pi} \sigma} n^{-3/2}.
	\end{align}
	We leave the reader check that, after many obvious cancellations,
	plugging in the estimates \eqref{eq:corpsereviver2} and \eqref{eq:partitionfunction}
	into \eqref{eq:distrib_Rtn} gives indeed \eqref{eq:almostthere}.
	\medskip
	
	{\em The combinatorial part: proof of \eqref{eq:distrib_Rtn}.}
	We first consider the unconditioned Galton--Watson tree $\bm T$, and conditionally on $\bm T$,
	a uniform list $\myvec{\rv{v}}$ of $k$ vertices with outdegree in $\Omega$ in $\bm T$ (possibly with repetitions).
	A pair $(T_\star, \myvec v_\star)$, where $\myvec v_\star$ is a list of $k$ vertices with outdegree in $\Omega$
	in the tree $T_\star$, is called {\em good} if 
	$|T_\star|_\Omega=n$ and $1.R^{[t_n]}(T_\star,\myvec v_\star)=(G, \myvec{x})$.
	Then we have
	\begin{align}
	\Proba \big[ 1.R^{[t_n]}(\bm T, \myvec{\rv{v}}) =  (G, \myvec{x}), |\bm T|_\Omega=n \big]
	&=\sum_{(T_\star,\myvec v_\star) \text{ good}} \Proba(\bm T=T_\star) \Proba\big( \myvec{\rv{v}}=\myvec v_\star|\bm T=T_\star\big) \nonumber\\
	&=n^{-k} \sum_{(T_\star,\myvec v_\star) \text{ good}} \Proba(\bm T=T_\star).
	\label{eq:Tech}
	\end{align}
	Good pairs $(T_\star,\myvec v_\star)$ can be constructed as follows.
	\begin{enumerate}
		\item We start from $(G, \myvec{x})$ and replace each edge with a length label $x_i$
		by a path with $x_i$ internal vertices;
		in total, this operation creates $\ell$ new vertices, which we will refer to as the {\em remote} vertices.
		\item We choose the outdegrees $(d_i)_{i \le \ell}$ in $T_\star$ of the $\ell$ remote vertices of $(T, \myvec{u})$;
		\item For each remote vertex, choose the distinguished offspring along which we have to proceed 
		to get to the first descendant that is an essential vertex ($d_i$ possible choices).
		\item On each of the $m:=\sum d_i-1$ other children of the remote vertices,
		we attach a fringe subtree tree $(A_j)_{j \le m}$.
		\item To ensure that $|T_\star|_\Omega=n$,
		the degrees $(d_i)_{i \le \ell}$ and the subtrees $(A_j)_{j \le m}$ should be chosen such that  
		$|G|_\Omega+\sum_{i=1}^\ell \mathbb{1}_{d_i \in \Omega}+\sum_{j=1}^m |A_j|_{\Omega}=n$.
	\end{enumerate}
	Moreover, if $(T_\star,\myvec v_\star)$ corresponds in this construction
	to sequences $(d_i)_{i \le \ell}$ and $(A_j)_{j \le m}$, then we have 
	($V_G$ denoting the set of vertices of the tree $G$)
	\[\Proba(\bm T=T_\star) = \left( \prod_{v \in V_G} \Proba\big[\xi= d^+_G(v)\big] \right)
	\, \left( \prod_{i=1}^\ell \Proba\big[ \xi=d_i \big] \right)
	\, \left( \prod_{j=1}^m  \Proba\big[\bm T= A_j \big] \right).\]
	(This probability is independent from the choices made in step iii) ).
	The sum over good pairs $(T_\star,\myvec v_\star)$ in \cref{eq:Tech} can be rewritten as a sum over sequences
	of positive integers $(d_i)_{i \le \ell}$ and sequences of trees $(A_j)_{j \le m}$,
	with an extra factor $\prod_i d_i$ coming from the choices in item iii) above.
	We get
	\begin{multline*}
	\Proba\big[ 1.R^{[t_n]}(\bm T, \myvec{\rv{v}}) =  (G, \myvec{x}), |\bm T|_\Omega=n \big]\\
	= n^{-k} \left( \prod_{v \in V_G} \Proba\big[\xi= d^+_G(v)\big] \right) 
	\sum_{d_1,\dots,d_\ell \ge 1} \left[  \prod_{i=1}^\ell d_i \Proba\big[ \xi=d_i \big]  \right.\\
	\qquad \cdot \left. \sum_{A_1,\dots,A_m \text{ trees}} \One\!\!
	\left[\sum_{j=1}^m |A_j|_{\Omega}=n-|G|_\Omega-\sum_{i=1}^\ell \mathbb{1}_{d_i \in \Omega}\right] 
	\, \prod_{j=1}^m  \Proba\big[\bm T= A_j \big] \right].
	\end{multline*}
	The sum in the last line is the probability that the total number of vertices with outdegree in $\Omega$
	in $m$ independent copies of $\bm T$ is $n-|G|_\Omega-\sum_{i=1}^\ell \mathbb{1}_{d_i \in \Omega}$,
	\emph{i.e.,} with the notation \cref{eq:notation}, this is $\Proba(\bm S_m=n-|G|_\Omega-\sum_{i=1}^\ell \mathbb{1}_{d_i \in \Omega})$.
	Recalling that by definition, $m=\sum_i d_i-1$, we get
	\begin{multline*}
	\Proba\big[ 1.R^{[t_n]}(\bm T, \myvec{\rv{v}}) =  (G, \myvec{x}), |\bm T|_\Omega=n \big]\
	= n^{-k} \left( \prod_{v \in V_G} \Proba\big[\xi= d^+_G(v)\big] \right) \cdot\\
	\left[ \sum_{d_1,\dots,d_\ell \ge 1} \left(\prod_{i=1}^\ell d_i \Proba\big[ \xi=d_i \big]\right)
	\, \Proba\left( \rv{S}_{\sum_{i=1}^{\ell} (d_i - 1)} = n - |G|_\Omega  -  \sum_{i=1}^\ell \mathbb{1}_{d_i \in \Omega} \right)\right]. \\
	\end{multline*}
	With the notation \cref{eq:notation}, the right-hand side can be simplified as
	\[\Proba\big[ 1.R^{[t_n]}(\bm T, \myvec{\rv{v}}) =  (G, \myvec{x}), |\bm T|_\Omega=n \big]\,
	= n^{-k}
	\left( \prod_{v \in V_G} \Proba\big[\xi= d^+_G(v)\big] \right) \Proba\left(\rv{S}_{ \rv{L}} = n - |G|_\Omega - \rv{Q} \right).\]
	Dividing by $\Proba(|\bm T|_\Omega =n)$ gives \cref{eq:distrib_Rtn}, as wanted.
	\medskip
	
	{\em The analytic part: proof of \eqref{eq:corpsereviver2}}.
	We are now looking for an asymptotic estimates for the probability
	$\Proba\left(\rv{S}_{ \rv{L}} = n - |G|_\Omega - \rv{Q} \right)$.
	Since $\bm S_m$ is a sum of $m$ i.i.d.\ random variables with the same law as $|\bm T|_\Omega$,
	this asymptotics depends on the tail of the distribution of $|\bm T|_\Omega$.
	We recall from \eqref{eq:partitionfunction} that
	\[
	\Proba(|\bm T|_\Omega = n) \sim (2 \pi \sigma^2_\Omega)^{-1/2} n^{-3/2},
	\]
	with 
	$\sigma_\Omega^2 =  \sigma^2/ \Proba(\xi \in \Omega)=\sigma^2/c_\Omega^2$.
	This implies that $|\bm T|_\Omega$ lies in the domain of attraction of the positive $(1/2)$-stable law. 
	Hence by~\cite[Sec. 50]{MR0062975},
	\begin{align}
	\label{eq:LLT_S}
	\lim_{m \to \infty} \sup_{r \ge 0} \left|m^2 \Proba\left( \rv{S}_m = r\right) - \sigma^2_\Omega g(\sigma^2_\Omega r / m^2) \right| = 0,
	\end{align}
	with $g$ denoting the positive $(1/2)$-stable density given by
	\begin{align}
	g(x) = (2 \pi)^{-1/2} x^{-3/2} \exp(-1/(2x)).
	\end{align}
	In particular, since this density is bounded, we have that, for $m$ large enough,
	\begin{equation}
	\sup_{r \ge 0} \left|\Proba\left( \rv{S}_m = r\right) \right| =O(m^{-2}).
	\label{eq:Bound_local_S}
	\end{equation}
	The law of large numbers tells us that 
	\[
	\ell^{-1} \rv{L} \convp \E[\hat{\xi}] - 1 = \E[\xi^2] -1 = \sigma^2.
	\]
	Moreover, from standard deviation estimates,
	there is a sequence $\epsilon_n \to 0$ with 
	\[
	\Proba\left(\rv{L} \notin (1 \pm \epsilon_n)\sigma^2 \ell \right) \to 0,
	\]
	and this sequence can be chosen uniformly for all $\ell$ in $[a\sqrt{n},b\sqrt{n}]$.
	It follows by conditioning on $\rv{L}$ and using~\eqref{eq:Bound_local_S} that
	\begin{multline*}
	\Proba\Big[\rv{S}_{ \rv{L}} = n - |G|_\Omega - \rv{Q}, \rv{L} \in (\sigma^2 \ell/2, (1 - \epsilon_n)\sigma^2 \ell)
	\text{ or } \rv{L} >(1 + \epsilon_n)\sigma^2 \ell \Big]  \\
	= O(\ell^{-2}) \Proba\left(\rv{L} \notin (1 \pm \epsilon_n)\sigma^2 \ell \right) = o(\ell^{-2}) = o(n^{-1}).
	\end{multline*}
	Moreover, the Azuma--Hoeffding inequality implies that for large enough $M>0$
	\[
	\Proba(\rv{L} < \sigma^2\ell/2) \le \Proba\left( \sum_{i=1}^{\ell}(\hat{\xi}_i \mathbb{1}_{\hat{\xi} \le M} -1)  \le \sigma^2 \ell/2 \right) \le \exp(-\Theta(\ell)) = \exp(-\Theta(\sqrt{n})).
	\]
	Thus we obtain
	\[
	\Proba(\rv{S}_\rv{L} = n- |G|_\Omega - \rv{Q}) 
	= o(n^{-1}) + \Proba(\rv{S}_\rv{L} = n - |G|_\Omega - \rv{Q}, \rv{L} \in (1 \pm \epsilon_n)\sigma^2 \ell).\]
	By definition, we have $\rv{Q} \le \ell$ a.s.,
	so that $n- |G|_\Omega - \rv{Q}=n-o(1)$, unformly
	when $|G|_\Omega \le s_n$ and $\ell$ in $[a\sqrt{n},b\sqrt{n}]$.
	Using \eqref{eq:LLT_S}, we have
	\[\Proba(\rv{S}_\rv{L} = n- |G|_\Omega - \rv{Q})
	\sim  (2 \pi)^{-1/2}\sigma_{\Omega}^{-1} \sigma^2 \ell n^{-3/2}  \exp(- \ell^2 \sigma^4 / (2n\sigma_\Omega^2));
	\]
	indeed, the right-hand side being of order $\Theta(n^{-1})$ (for $\ell$ in $[a \sqrt{n},b\sqrt{n}]$),
	the error term $o(n^{-1})$ above can be forgotten.
	This verifies \cref{eq:corpsereviver2}.
	\medskip
	
	{\em Proof of the statement with $t_n=o(\sqrt{n})$ (\cref{eq:toshow}).}
	Since the equivalent in \cref{eq:almostthere} is uniform on trees $G$ in $\setTkt$ with $|G|_{|\Omega|}\le s_n$,
	\cref{eq:almostthere} implies a weaker version of \cref{eq:toshow}, 
	where we let $A$ range only over subsets of $\setTkt$ containing only trees
	with $\Omega$-size at most $s_n$.  
	It remains to verify that there exists a sequence $s_n=o(n)$ that additionally satisfies
	\begin{align}
	\label{eq:lastpart}
	\Proba\big[|\Sh(\bm T^{k,t_n}_\Omega)|_\Omega \ge s_n\big] \to 0.
	\end{align}
	Since by \eqref{eq:Sh_pG}, we have
	\[\Proba\big[|\Sh(\bm T^{k,t_n}_\Omega)|_\Omega \le s_n \big] -
	\Proba\big[|\Sh(c_\Omega n^{-1/2}.R^{[t_n]}(\bm T_n^\Omega, \myvec{\rv{v}}) )|_\Omega \le s_n \big] \to 0,\]
	\eqref{eq:lastpart} would imply
	\[
	\Proba(|\Sh(\sigma  c_\Omega n^{-1/2}.R^{[t_n]}(\bm T_n^\Omega, \myvec{\rv{v}}))|_\Omega \ge s_n) \to 0
	\]
	and hence complete the proof of \cref{eq:toshow}. 
	
	Let us check~\eqref{eq:lastpart}. By construction of $\bm T^{k,t_n}_\Omega$ (see \cref{fig:ktree}),
	we have: 
	\begin{equation}
	|\Sh(\bm T^{k,t_n}_\Omega)|_\Omega \leq 4 k t_n + \rv{S}_{\rv{M}} + \rv{S'}_k,
	\label{eq:Control_Size_TOmega}
	\end{equation}
	where the summands in the right-hand side are independent and distributed as follows:
	\begin{itemize}
		\item $\rv{M}$ satisfies
		\[\rv{M} \eqdist  \sum_{i=1}^{k-1} (\xi^*_i-2) + \sum_{i=1}^{1 + t_n(4k-2)} (\hat{\xi}_i -1),   
		\]
		and, as above, $\rv{S}_{\rv{M}}$ denotes the sum of $\rv{M}$ independent random variables $(X_i)_{i \le \bm M}$
		of law $|\bm T|_\Omega$, the $X_i$ being also independent of $\rv{M}$;
		\item $\rv{S'}_k$ is the sum of $k$ i.i.d.\ random variables of law $|\bm T|_\Omega$, conditioned on the root
		of $\bm T$ having outdegree in $\Omega$;
		\item $4 k t_n$ is an upper bound for the number of vertices on the stretched skeleton  of $\bm T^{k,t_n}_\Omega$
		having outdegree in $\Omega$.
	\end{itemize}
	
	It is known (see Rizzolo~\cite[Thm. 6]{MR3335013})
	that $|\bm T|_\Omega \in \{0, 1, \ldots\}$ may be stochastically bounded by the number of vertices of Galton--Watson tree with a different offspring distribution	that is also critical and has finite variance. It follows by a general result for the size of Galton--Watson forests, there is a constant $C>0$ such that 
	\begin{align}
	\label{eq:upperbound}
	\Proba(\rv{S}_m \ge x) \le C m x^{-1/2}.
	\end{align}
	for all $m$ and $x$. See Devroye and Janson \cite[Lem. 4.3]{MR2829308} and Janson \cite[Lem. 2.1]{MR2245498}. Let $\epsilon >0$ be given. 
	We choose a constant $K>0$ such that 
	\[\Proba\big[\rv{S}_{\sum_{i=1}^{k-1} (\xi^*_i-2)} + \rv{S'}_k > K \big] \le \epsilon/2.\] 
	Letting $\rv{M'} \eqdist \sum_{i=1}^{1 + t_n(4k-2)} (\hat{\xi}_i -1)$ be independent from the family $(\rv{X_i})_{i \ge 1}$, 
	it follows by  Inequalities~\eqref{eq:Control_Size_TOmega} and~\eqref{eq:upperbound}  that
	\begin{multline*}
	\Proba\big[|\Sh(\bm T^{k,t_n}_\Omega)|_\Omega \ge s_n \big]
	\le   \Proba\big[4 k t_n + \rv{S}_{\rv{M}} + \rv{S'}_k \ge s_n \big]
	\le \epsilon/2 + \Proba\big[ \rv{S}_\rv{M'} \ge s_n-K-4kt_n \big] \\
	\le \epsilon/2 + C (1 + t_n(4k -2) )(\E[\hat{\xi}] -1)(s_n- K-4kt_n)^{-1/2}.
	\end{multline*}
	Hence~\eqref{eq:lastpart} holds if we select $s_n = o(n)$ such that $t_n / \sqrt{s_n} \to 0$,
	which is clearly possible since $t_n = o(\sqrt{n})$.  This completes the proof.
\end{proof}

The above proof essentially also gives a local version of Lemma~\ref{le:semilocal}, 
which we believe to be of independent interest, and state below as \cref{le:semilocal2}. 
\begin{lemma}
	\label{le:semilocal2}
	Let the offspring distribution $\xi$ be critical, aperiodic, and have a finite variance. 
	Let $\myvec{\rv{v}}$ be a vector of $k \ge 1$ independently and
	uniformly selected vertices 
	with outdegree in $\Omega$  
	of the conditioned tree $\bm T_n^\Omega$. 
	Besides, we fix sequences $t_n$ and $s_n$ of non-negative integers satisfying
	$t_n^2=o(s_n)$ and $s_n=o(n)$.
	Then there exists sequence $a_n$ and $b_n$ tending to $0$ and $+\infty$ respectively
	such that:
	\begin{enumerate}
		\item the estimates
		\begin{align}
		\label{eq:local2}
		\Proba\big[ 1.R^{[t_n]}(\bm T_n^\Omega, \myvec{\rv{v}}) = (G, \myvec{x}) \big] 
		\sim \left( \frac{\sigma  c_\Omega}{\sqrt{n}}\right)^{2k-1}  
		h\left(\sigma c_\Omega n^{-1/2} \myvec{x}\right) p_G,
		\end{align}
		holds uniformly for all pairs $(G, \myvec{x})$ in $\setTktn \times (R_+)^{2k-1}$
		with $G$ generic verifying the conditions
		$|G|_\Omega \le s_n$ and $a_n \sqrt{n} \le \| \myvec{x} \|_1 \le b_n \sqrt{n}$;
		\item and, if we write 
		\[ (\bm G,\bm{\myvec x}):=1.R^{[t_n]}(\bm T_n^\Omega, \myvec{\rv{v}}),\]
		then the following properties hold with high probability as $n$ becomes large:
		$\bm G$ is generic, $|\bm G|_\Omega \le s_n$ and $a_n \sqrt{n} \le \| \bm{\myvec{x}} \|_1 \le b_n \sqrt{n}$.
	\end{enumerate}
\end{lemma}
\begin{proof}
	The estimates i) with a fixed $a$ instead of $a_n$ and a fixed $b$ instead of $b_n$
	has been proved in \cref{eq:almostthere} above.
	The existence of sequences $a_n$ and $b_n$ such that i) holds 
	is a direct consequence, using the following elementary analysis lemma.  
	\begin{quote}
		Let $F(A,n)$ be a bivariate function, nonincreasing in $A$.
		We assume that for any $A>0$, we have $\lim_{n \to \infty} F(A,n)=0$.
		Then, there exists a sequence $A_n$ tending to $0$ such that $F(A_n,n)$ tends to $0$.
	\end{quote} 
	Finally ii) holds for any sequences $a_n$ and $b_n$ and any $s_n$ with $t_n^2=o(s_n)$,
	as a consequence of \cref{eq:toshow,eq:lastpart}.
\end{proof}

Finally, the following statement will be useful (with $\Omega = \{0\}$, \emph{i.e.} marking leaves) in the special case of separable permutations. 
It can either be proved as a corollary of \cref{le:semilocal2}, 
or similarly to Lemma 6.2 in~\cite{MR3813988}. 

\begin{corollary}
	Let the offspring distribution $\xi$ be critical, aperiodic, and have a finite variance. 
	Let $\myvec{\rv{v}}$ be a vector of $k \ge 1$ independently and
	uniformly selected vertices with outdegree in $\Omega$  of the conditioned tree $\bm T_n^\Omega$. 
	Then, for any fixed $t$, asymptotically as $n \to \infty$, the parities of the heights of the essential vertices induced by $\myvec{\rv{v}}$ (except the root of $\bm T_n^\Omega$) converge to  
	Bernoulli random variables of parameter $1/2$, independent among themselves, 
	and from the tree $\Sh(1.R^{[t]}(\bm T_n^\Omega, \myvec{\rv{v}}))$. 
	\label{cor:semilocal3}
\end{corollary}

\section{Scaling limits}
\label{sec:scaling}

\subsection{Background on permuton convergence}
\label{subsec:obs_perm}

As said in introduction, a \emph{permuton} $\mu$ is a Borel probability measure on the unit square $[0,1]^2$ with uniform marginals, that is
\[
\mu( [0,1] \times [a,b] ) = \mu( [a,b] \times [0,1] ) = b-a
\]
for all $0 \le a \le b\le 1$. Any permutation $\nu$ of size $n \ge 1$ may be interpreted as a permuton $\mu_\nu$ given by the linear combination of area measures
\[
\mu_\nu= n \sum_{i=1}^n \delta_{[(i-1)/n, i/n]\times[(\nu(i)-1)/n,\nu(i)/n]}.
\]

By definition, a random permutation $\bm{\nu}_n$ converges weakly to a random permuton $\bm{\mu}$ as $n \to \infty$ if the random probability measure $\bm{\mu}_{\bm{\nu}_n}$ converges weakly to $\bm{\mu}$.  There are different characterisations for this form of convergence \cite[Thm. 2.5]{bassino2017universal}. In particular, if $\bm{\nu}_n$ has size $n$, then the following statements are equivalent:
\begin{enumerate}
	\item There exists a permuton $\bm{\mu}$ such that $\bm{\mu}_{\bm{\nu}_n} \convdis \bm{\mu}$.
	\item For any integer $k \ge 1$ the pattern $\pat_{\rv{I}_{n,k}}(\bm{\nu}_n)$ induced by a uniform random $k$-element subset $\rv{I}_{n,k} \subseteq [n]$ admits a distributional limit $\bm{\rho}_k$.
\end{enumerate} 
In this case, the limit family $(\bm{\rho}_k)_{k}$ is consistent in the sense that $\bm{\rho}_k$ has size $k$ a.s.\ for all $k$ and  $\pat_{\rv{I}_{n,k}}(\bm{\rho}_n) \eqdist \bm{\rho}_k$ for all $1 \le k \le n$.  The permuton $\bm{\mu}$ may be constructed from the family $(\bm{\rho}_k)_{k \ge 1}$, and is hence uniquely determined by it. 
In fact, there is a bijection between random permutons and consistent families~\cite[Prop. 2.9]{bassino2017universal}.
(Compare  with a similar result for random trees~\cite[Thm. 18]{MR1207226}.)

The following permutons were introduced in~\cite{bassino2017universal,MR3813988} 
where they were proved to be the limit of some substitution-closed classes. (See also  \cite{maazoun17BrownianPermuton} for some properties of these permutons.)
\begin{enumerate}
	\item The Brownian separable permuton corresponds to the case where $\bm{\rho}_k$ is 
	the image by $\CanTree^{-1}$ of a uniform binary plane tree with $k$ leaves
	with uniform independent decorations from $\{\oplus, \ominus\}$ on its internal vertices.
	(Recall from \cref{rk:CT-1} that $\CanTree^{-1}$ can be applied to $\{\oplus, \ominus\}$-decorated trees,
	where neighbours may have the same sign.) 
	\item Let $0<p<1$ be a constant. The biased Brownian separable permuton of parameter $p$ is constructed in the same way, but instead of assigning the $\oplus$ / $\ominus$ decorations via fair coin flips, we toss a biased coin that shows $\oplus$ with probability $p$.
\end{enumerate}\medskip

Putting together the pattern characterization 
of permuton convergence (recalled above), 
this description of $\bm{\rho}_k$, 
and the connection between patterns and subtrees explained in \cref{ssec:patterns_subtrees},
we get a convenient sufficient condition for the convergence to a
(biased) separable Brownian permuton.

To state it, we recall that, if $\myvec\ell$ is an ordered sequence of marked leaves in a tree $T$,
then $R(T,\myvec\ell)$ denotes the subtree consisting of these marked leaves
and all their ancestors.
In addition, we denote by $R^\star(T,\myvec\ell)$ the tree obtained from $R(T,\myvec\ell)$
by successively removing all non-root vertices of outdegree $1$, merging their two adjacent edges.
\begin{lemma}
	\label{lem:convergence_via_induced_subtrees}
	Let $p$ be a constant in $[0,1]$ and, for each $n \ge 1$, $\bm{\nu}_n$ be a random permutation of size $n$.
	For each fixed $k \ge 1$, we take a uniform random sequence $\myvec{\bm{\ell}}= (\bm{\ell}_1, \ldots, \bm{\ell}_k)$ of $k$ leaves
	in the canonical tree $\bm{T}_n$ of $\bm{\nu}_n$.
	We make the following assumptions.
	\begin{itemize}
		\item The tree $R^\star(\bm{T}_n,\myvec{\bm{\ell}})$ should converge
		(in distribution) to a proper $k$-tree.
		\item For each non-root internal vertex $\bm u$ of $R^\star(\bm{T}_n,\myvec{\bm{\ell}})$, 
		we choose arbitrarily two leaves from $\myvec{\bm{\ell}}$, say $\bm{\ell}_{i_u}$ and $\bm{\ell}_{j_u}$,
		whose common ancestor is $\bm u$.
		We then assume that $\bm{\ell}_{i_u}$ and $\bm{\ell}_{j_u}$
		form a non-inversion asymptotically
		with probability $p$, and that,
		when $\bm u$ runs over non-root internal vertices of $R^\star(\bm{T}_n,\myvec{\bm{\ell}})$,
		these events are asymptotically
		independent from each other and from the shape $R^\star(\bm{T}_n,\myvec{\bm{\ell}})$.
	\end{itemize}
	Then $\bm{\nu}_n$ converges to the biased separable Brownian permuton of parameter $p$.
\end{lemma}

The arbitrary choices made in the second item above are irrelevant.
Indeed, when $\bm u$ has out-degree $2$ in $R^\star(\bm{T}_n,\myvec{\bm{\ell}})$
(which is the case with high probability under the first assumption),
the fact that $\bm{\ell}_{i_u}$ and $\bm{\ell}_{j_u}$ form an inversion
or not does not depend on the choice of $\bm{\ell}_{i_u}$ and $\bm{\ell}_{j_u}$
(this an easy consequence of the discussion from \cref{ssec:patterns_subtrees}).

\subsection{Permuton convergence of random permutations from substitution-closed classes}

We now prove our first main theorem, Theorem~\ref{thm:scaling_intro}. 
We start by stating this theorem more precisely. 

\begin{theorem}
	\label{thm:scaling}
	Let $\bm{\nu}_n$ be the uniform $n$-sized permutation from a proper substitution-closed class of permutations $\cC$. 
	Let $\xi$ be the offspring distribution of the associated Galton--Watson tree model. 
	Suppose that $\E[\xi]=1$ and $\V[\xi] < \infty$. That is, either
	\[
	\cS'(\rho_\cS) > \frac{2}{(1 +\rho_\cS)^2} -1,
	\]
	or
	\[
	\cS'(\rho_\cS)  = \frac{2}{(1 +\rho_\cS)^2} -1 \qquad \text{and} \qquad 	\cS''(\rho_\cS)  < \infty.
	\]
	Then
	\[
	\bm{\mu}_{\bm{\nu}_n} \convdis \bm{\mu}^{(p)},
	\]
	with $\bm{\mu}^{(p)}$ denoting the biased Brownian separable permuton with parameter
	\begin{equation}
	\label{eq:parm_p}
	p = \frac{2}{\sigma^2}\big(\kappa(1+\kappa)^3\Occ_{12}(\kappa)+\kappa\big),
	\end{equation}
	where $\kappa$ and $\sigma^2=\V[\xi]$ are defined in \cref{prop: offspring_distr_charact} and $\Occ_{12}(z)=\sum_{\alpha\in\mathcal{S}}\socc(12,\alpha)z^{|\alpha|-2}$, 
	with $\socc(12,\alpha)$ being the number of occurrences of the pattern $12$ in $\alpha$.
	
	This includes the case where $\cC$ is the class of separable permutations, for which $\cS(z)=0$ and $p=1/2$.
\end{theorem}
\begin{proof}
	By Proposition~\ref{prop:giant_comp_perm}, it suffices to show that the uniform $n$-sized permutation $\bm{\nu}_n$ from $\cC_{\nonp}$ satisfies
	\[
	\bm{\mu}_{\bm{\nu}_n} \convdis \bm{\mu}^{(p)}.
	\]

	We first consider the separable case $\cS= \emptyset$. 
	Let $\bm T_n$ be the canonical tree of $\bm \nu_n$.
	Here a vertex of $\bm T_n$ is decorated with $\ominus$ if and only if it has even height. 
	Without its decorations, $\bm T_n$ has the law of a critical Galton-Watson tree with finite variance
	conditioned on having $n$ leaves
	(see \cite[Sec. 2.2]{MR3813988} or \cref{sec:indecomposable_GW}; 
	for the separable case, packed trees and canonical trees only differ by their decorations).
	
	Let $k \ge 1$ be given and $\myvec{\bm{\ell}} = (\bm{\ell}_1, \ldots, \bm{\ell}_k)$ be a uniform random sequence of leaves in $\bm T_n$. 
	It follows from Lemma~\ref{le:semilocal} 
	that $R^\star(\bm T_n, \myvec{\bm{\ell}})$ is asymptotically a uniform random proper $k$-tree.
	Corollary~\ref{cor:semilocal3} yields the additional information
	that the parities of the lengths of the $2k-1$ paths in $\bm T_n$ 
	corresponding to the edges of $R^\star(\bm T_n, \myvec{\bm{\ell}})$ 
	converge jointly to $2k-1$ independent fair coin flips,
	independently of the shape $R^\star(\bm T_n, \myvec{\bm{\ell}})$. 
	Hence in the limit as $n \to \infty$ each non-root internal vertex of $R^\star(\bm T_n, \myvec{\bm{\ell}})$
	receives a sign $\oplus$ or $\ominus$ with probability $1/2$
	(meaning that the corresponding leaves, 
	in the sense of the second item of \cref{lem:convergence_via_induced_subtrees},
	form an inversion with probability $1/2$);
	moreover, these events are asymptotically independent from each other
	and from the shape $R^\star(\bm T_n, \myvec{\bm{\ell}})$. 
	As this holds for all $k\ge 1$, thanks to \cref{lem:convergence_via_induced_subtrees},
	it follows that $\bm{\nu}_n$ converges in distribution to the Brownian separable permuton $\bm{\mu}^{(1/2)}$.\vspace{3 mm}
	
	Let us now consider the case $\cS \neq \emptyset$. 
	In this case, it is more convenient to work with packed trees rather than canonical trees
	(note however that both trees have the same set of leaves).
	In particular, the random packed tree $\bm{P}_n=(\bm{T}_n,\bm{\lambda}_{\bm{T}_n})$
	associated with the uniform permutation $\bm \nu_n$ in $\cC_{\nonp}$ 
	is a Galton-Watson tree with a specific offspring distribution $\xi$
	conditioned on having $n$ leaves, with independent random decorations on each vertex
	(see \cref{sec:indecomposable_GW}). 
	As before, we fix $k \ge 1$ and consider a uniformly selected set of distinct leaves 
	$\myvec{\bm{\ell}} = (\bm{\ell}_1, \ldots, \bm{\ell}_k)$ in $\bm T_n$.
	
	By \cref{le:semilocal}, we know that
	$R^\star(\bm{T}_n,\myvec{\bm{\ell}})$ converges
	(in distribution) to a uniform proper $k$-tree
	(recall that $\xi$ is always aperiodic and that
	it has expectation $1$ and finite variance by assumption, as needed to apply \cref{le:semilocal}).
	In particular, the tree $R^\star(\bm{T}_n,\myvec{\bm{\ell}})$ is a proper tree 
	(with a root of outdegree $1$ and other internal vertices of outdegree $2$) 
	with high probability, as $n \to \infty$.
	When this is the case, since the packing construction only merges internal vertices,
	$R^\star(\bm{T}_n,\myvec{\bm{\ell}})$ coincide with 
	$R^\star(\tilde{\bm{T}}_n,\myvec{\bm{\ell}})$,
	where $\tilde{\bm{T}}_n$ is the {\em canonical} tree associated with $\bm{\nu}_n$.
	Therefore, although \cref{lem:convergence_via_induced_subtrees} is stated with the canonical tree $\tilde{\bm{T}}_n$,
	we can use it here with the packed tree $\bm T_n$ instead.
	
	Using the notation of \cref{lem:convergence_via_induced_subtrees},
	it remains to analyse whether $\bm\ell_{i_u}$ and $\bm\ell_{j_u}$ form an inversion
	or not (for non-root internal vertices $\bm u$ of $R^\star(\bm{T}_n,\myvec{\bm{\ell}})$).
	
	We recall (see \cref{ssec:patterns_subtrees})
	that if $\bm u$ is decorated with an $\cS$-gadget,
	then whether $\bm \ell_{i_u}$ and $\bm \ell_{j_u}$ form an inversion or not
	is determined by the decoration of $\bm u$ and by which branches attached to $\bm u$
	contain $\bm \ell_{i_u}$ and $\bm \ell_{j_u}$.
	This information is contained in $\Sh(s.R^{[0]}(\bm{T}_n,\myvec{\bm{\ell}}))$ for any $s>0$.
	
	On the other hand, if $\bm u$ is decorated with $\circledast$, 
	then in order to determine whether $\bm \ell_{i_u}$ and $\bm \ell_{j_u}$ form an inversion or not, 
	we have to recover the parity of the distance of $\bm u$
	to its first ancestor decorated with an $\cS$-gadget 
	(if it exists, otherwise to the root of $\bm T_n$). 
	
	Take $t_n$ tending to infinity, but with $t_n=o(\sqrt{n})$.
	By Lemma~\ref{le:semilocal}, $\bm u$ has asymptotically $t_n$ ancestors 
	with out-degrees $\hat{\xi}_1, \hat{\xi}_2, \ldots, \hat{\xi}_{t_n}$ 
	being independent copies of $\hat{\xi}$ defined in~\eqref{eq:xihat}. 
	Moreover the vertex $\bm u$ and each of its ancestors receive a decoration
	that gets drawn independently and uniformly at random 
	among all $\widehat{\GGG(\SSS)}$-decorations with size equal to the out-degree of the vertex. 
	In this setting, with high probability, one of the $t_n$ ancestors will receive an $\cS$-gadget as decoration. 
	Therefore, with high probability, whether $\bm \ell_{i_u}$ and $\bm \ell_{j_u}$ form an inversion
	is determined by $\Sh(s.R^{[t_n]}(\bm{T}_n,\myvec{\bm{\ell}}))$ for any $s>0$.
	
	We say that two families (indexed by $\NN$) of probability distributions are \emph{close} 
	when their total variation distance tends to $0$ as $n$ tends to infinity. 
	By Lemma~\ref{le:semilocal}, the distributions of the random trees $\Sh(s_n.R^{[t_n]}(\bm{T}_n,\myvec{\bm{\ell}}))$
	and $\Sh(\bm{T}^{k,t_n}_{\{0\}})$ are close, 
	for a well-chosen sequence $s_n$. 
	From the above discussion, this implies that
	the joint distributions of 
	$R^\star(\bm{T}_n,\myvec{\bm{\ell}})$ and
	\[\Big(\One\Big[\text{$\bm \ell_{i_u}$ and $\bm \ell_{j_u}$ form an inversion
		in $(\bm{T}_n,\myvec{\bm{\ell}})$}\Big]\Big)_{\bm u}\]
	are close to the distributions of the same variables in the limiting tree $\bm{T}^{k,t_n}_{\{0\}}$.
	When $n$ tends to infinity, these tend a.s.\ (with the obvious coupling between the $\bm{T}^{k,t_n}_{\{0\}}$)
	to the same variables in $\bm{T}^{k,\bm{t}^*}_{\{0\}}$,
	where $\bm{t}^*$ denotes the minimal radius such that each internal 
	$\circledast$-decorated essential vertex (different from the root) has an ancestor decorated by an $\cS$-gadget.
	
	In the limiting tree $\bm{T}^{k,\bm{t}^*}_{\{0\}}$,
	the neighbourhoods of the essential vertices $\bm u$ 
	are independent from each other
	and all have the same distribution (which does not depend on $k$,
	nor on the shape $R^\star(\bm{T}^{k,\bm{t}^*}_{\{0\}},\myvec{\bm{\ell}})$,
	the latter being the proper $k$-tree taken at step 1 of the construction) .
	Therefore the probability that $\bm \ell_{i_u}$ and $\bm \ell_{j_u}$ form a non-inversion
	tend to some parameter $p$ in $[0,1]$,
	which depends only on the permutation class $\mathcal C$ we are working with.
	Moreover these events are asymptotically independent from each other 
	and from the shape $R^\star(\bm{T}_n,\myvec{\bm{\ell}})$.
	From \cref{lem:convergence_via_induced_subtrees}, this implies that $\bm{\mu}_{\bm{\nu}_n} \convdis \bm{\mu}^{(p)}$.\medskip

	It remains to calculate an explicit expression for the limiting probability $p$. 
	For this, we consider $k=2$, \emph{i.e.,} $p$ is the probability that, 
	in the limiting tree $\bm{T}^{2,\bm{t}^*}_{\{0\}}$, 
	the two marked leaves $\bm\ell_1$ and $\bm\ell_2$ form a non-inversion.
	
	For each integer $m \ge 1$, let $\rv{G}_m$ be drawn uniformly at random among all $m$-sized $\widehat{\GGG(\SSS)}$-objects, \emph{i.e.,} 
	$$\Proba(\rv{G}_m=G)=1/q_m,\quad \text{for all}\quad G\in\widehat{\GGG(\SSS)} \text{ of size }m,$$ 
	where we recall that $\mathcal Q(z)=\widehat{\GGG(\SSS)}(s)=\sum_{k\geq 2}q_kz^k$ is the generating function in \cref{eq:G}. 
	We also recall the following three distributions (see \cref{eq:offspring_distribution_packed_tree,eq:xihat,eq:xistar})
	\begin{equation}
	\Proba(\xi=k)=q_kt_0^{k-1}, \quad \Proba(\hat{\xi}=k)=kq_kt_0^{k-1},\quad \Proba(\xi^*=k)=kq_k(k-1)t_0^{k-1}/\sigma^2,
	\end{equation}
	where $t_0=\frac{\kappa}{1+\kappa}$ is the parameter determined in \cref{prop: offspring_distr_charact} as the unique number such that
	\begin{equation}
	\label{eq:proper_KAPPA}
	\cS'(\kappa) = 2/(1+\kappa)^2 -1,   
	\end{equation} 
	$\cS(z)$ being the generating functions for simple permutations in the considered substitution-closed class $\mathcal{C}$.
	
	To determine whether $\bm\ell_1$ and $\bm\ell_2$ form an inversion or not, 
	there are two cases to consider, depending on whether the decoration $\bm{\lambda}_{\bm{T}^{k,\bm{t}^*}_{\{0\}}}(\bm u)\eqqcolon\bm{\lambda}(\bm u)$ of 
	the closest common ancestor $\bm u$ of $\bm\ell_1$ and $\bm\ell_2$ is an $\cS$-gadget or not. 
	
	We start with the case where it is not. Let $\bm u'$ be the closest ancestor of $\bm u$ that is decorated with an $\cS$-gadget. 
	The limiting probability for $\bm\ell_1$ and $\bm\ell_2$ to form a non-inversion in this case is given by
	\begin{multline*}
	\Proba\big(\bm{\lambda}(\bm u)=\circledast,d(\bm u,\bm u') \text{ is even and}>0\big)\\
	=\sum_{k \ge 2}\Proba\big(d(\bm u,\bm u') \text{ is even}\big|\bm{\lambda}(\bm u)=\circledast\big)\Prb{\rv{G}_{k}=\circledast}\Proba(\xi^*=k)
	\end{multline*}
	where we used that $\bm u$ take offsprings according to $\xi^*$. Since the ancestors of $\bm u$ (between $\bm u$ and $\bm u'$) take offsprings according to $\hat{\xi},$ we have
	\[
	\Proba\big(d(\bm u,\bm u') \text{ is even and}>0\big|\bm{\lambda}(\bm u)=\circledast\big)=\Proba(\text{Geom}(\eta) \text{ is even and}>0)=\frac{\eta-1}{\eta-2},
	\]
	where
	\begin{multline*}
	\eta=\Proba(\bm{G}_{\hat{\xi}}\neq\circledast)=1-\sum_{k \ge 2}\Proba(\bm{G}_{k}=\circledast)\Proba(\hat{\xi}=k)\\
	=1-\sum_{k \ge 2}\frac{1}{q_k}kq_kt_0^{k-1}=1-\frac{t_0(2-t_0)}{(1-t_0)^2}=1-\kappa(\kappa+2).
	\end{multline*}
	Summing-up
	\begin{equation}
	\begin{split}
	\label{eq:keyres1}
	\Proba\big(\bm{\lambda}(\bm u)=\circledast,d(\bm u,\bm u')\text{ is even}\big)&=\sum_{k \ge 2}\frac{\eta-1}{\eta-2}\Prb{\rv{G}_{k}=\circledast}\Proba(\xi^*=k)\\
	&=\frac{1}{\sigma^2}\frac{\kappa(\kappa+2)}{(1+\kappa)^2}\sum_{k \ge 2}k(k-1)t_0^{k-1}=\frac{2}{\sigma^2}\kappa^2(\kappa+2),
	\end{split}
	\end{equation}
	where in the last equality we used that $\sum_{k \ge 2}k(k-1)t_0^{k-1}=\frac{2t_0}{(1-t_0)^3}$ and $t_0=\frac{\kappa}{1+\kappa}$.
	
	Now consider the case where the decoration $\bm{\lambda}(\bm u)$ is an $\cS$-gadget.  
	That is, it consists of a root decorated with a simple permutation with several branches, each of which may be a leaf or a $\oplus$-decorated vertex to which at least $2$ leaves are attached. 
	By definition, the leaves $\bm \ell_1$ and $\bm \ell_2$ are descendants of different children of $\bm u$, say the $\bm i_1$-th and $\bm i_2$-th. 
	These $\bm i_1$-th and $\bm i_2$-th branches attached to $\bm u$ identify two leaves (the $\bm i_1$-th and $\bm i_2$-th) of 
	the $\cS$-gadget $\bm{\lambda}(\bm u)$ decorating $\bm u$. 
	If these two leaves belong to the same branch attached to the root of $\bm{\lambda}(\bm u)$, 
	then $\bm \ell_1$ and $\bm \ell_2$ do not induce an inversion, as their closest common ancestor in $\bm{\lambda}(\bm u)$ is decorated by $\oplus$ (see also \cref{ssec:patterns_subtrees}). 
	Otherwise, when they belong to two different branches attached to the root of $\bm{\lambda}(\bm u)$ (say the $\bm j_1$-th and $\bm j_2$-th), 
	it depends on the simple permutation $\bm \alpha$ appearing in the root of $\bm{\lambda}(\bm u)$: 
	$\bm \ell_1$ and $\bm \ell_2$ do not induce an inversion if and only if $\pat_{\{\bm j_1,\bm j_2\}}(\bm \alpha)=12$. 
	
	Therefore, in the case where the decoration $\bm{\lambda}(\bm u)$ is an $\cS$-gadget, the limiting probability for $\bm\ell_1$ and $\bm\ell_2$ to form a non-inversion is given by
	\begin{multline}
	\label{eq:split_prob}
	\Proba\big(\bm{\lambda}(\bm u)\neq\circledast,\;\bm i_1,\bm i_2 \text{ are in the same branch of }\bm{\lambda}(\bm u)\big)\\
	+\Proba\big(\bm{\lambda}(\bm u)\neq\circledast,\;\bm i_1,\bm i_2 \text{ are not in the same branch of } \bm{\lambda}(\bm u),\pat_{\{\bm j_1,\bm j_2\}}(\bm \alpha)=12\big).
	\end{multline}
	We start by computing the first probability
	\begin{multline*}
	\Proba\big(\bm{\lambda}(\bm u)\neq\circledast,\;\bm i_1,\bm i_2 \text{ are in the same branch of }\bm{\lambda}(\bm u)\big)\\
	=\sum_{k \ge 4}\Proba\big(\bm i_1,\bm i_2 \text{ are in the same branch of }\bm{\lambda}(\bm u)\big|\bm{\lambda}(\bm u)\neq\circledast,d^{+}(\bm u)=k\big)\Prb{\bm{\lambda}(\bm u)\neq\circledast,d^{+}(\bm u)=k} \\
	=\sum_{k \ge 4}\Proba\big(\bm i_1,\bm i_2 \text{ are in the same branch of }\bm{\lambda}(\bm u)\big|\bm{\lambda}(\bm u)\neq\circledast,d^{+}(\bm u)=k\big)\Prb{\rv{G}_{k}\neq\circledast}\Proba(\xi^*=k).
	\end{multline*}
	Note that, denoting with $\text{root-deg}(\bm{\lambda}(\bm u))$ the degree of the root of the decoration of the vertex $\bm u,$ we have
	\begin{multline*}
	\Proba\big(\bm i_1,\bm i_2\text{ are in the same branch of }\bm{\lambda}(\bm u)\big|\bm{\lambda}(\bm u)\neq\circledast,d^{+}(\bm u)=k\big)=\\
	\sum_{a=4}^{k}\frac{\binom{k-1}{a-1}s_{a}}{q_k-1}\Proba\big(\bm i_1,\bm i_2 \text{ are in the same branch of }\bm{\lambda}(\bm u)\big|\bm{\lambda}(\bm u)\neq\circledast,d^{+}(\bm u)=k,\text{root-deg}(\bm{\lambda}(\bm u))=a\big),
	\end{multline*}
	where we use that $\Proba\big(\text{root-deg}(\bm{\lambda}(\bm u))=a|\bm{\lambda}(\bm u)\neq\circledast,d^{+}(\bm u)=k\big)=s_{a}\binom{k-1}{a-1}\tfrac{1}{q_k-1}$, 
	since the number of $\cS$-gadgets of size $k$ with root-degree $a$ is equal to $s_{a}$ multiplied by the number of compositions of $k$ in $a$ parts. The probability 
	$$\Proba\big(\bm i_1,\bm i_2 \text{ are in the same branch of }\bm{\lambda}(\bm u)\big|\bm{\lambda}(\bm u)\neq\circledast,d^{+}(\bm u)=k,\text{root-deg}(\bm{\lambda}(\bm u))=a\big)$$
	is simply the probability that two uniformly chosen
	elements a uniform random composition of $k$ in $a$ parts
	(seen as a list of $k$ elements and $a-1$ bars)
	are in the same part.
	Summing over the positions of the two uniformly chose,
	this probability is easily seen to be
	\begin{equation}
	\label{eq:comb_comput}
	\binom{k-1}{a-1}^{-1} \binom{k}{2}^{-1} \sum_{i<j}\binom{k-(j-i)-1}{a-1} = \frac{2}{a+1}\cdot\frac{k-a}{k-1},
	\end{equation}
	where the last equality is obtained via a computer algebra system.
	Summing-up,
	\begin{multline*}
	\Proba\big(\bm i_1,\bm i_2 \text{ are in the same branch of }\bm{\lambda}(\bm u)\big|\bm{\lambda}(\bm u)\neq\circledast,d^{+}(\bm u)=k\big)\\
	=\frac{1}{q_k-1}\sum_{a=4}^{k}s_{a}\binom{k-1}{a-1}\frac{2}{a+1}\cdot\frac{k-a}{k-1}=\frac{2}{q_k-1}\sum_{a=4}^{k}\frac{s_{a}}{a+1}\binom{k-2}{a-1},
	\end{multline*}
	and so
	\begin{equation}
	\begin{split}
	\label{eq:keyres2}
	\Proba\big(\bm{\lambda}(\bm u)\neq&\circledast,\;\bm i_1,\bm i_2 \text{ are in the same branch of }\bm{\lambda}(\bm u)\big)\\
	&=\sum_{k \ge 4}\frac{2}{q_k-1}\sum_{a=4}^{k}\frac{s_{a}}{a+1}\binom{k-2}{a-1} \Prb{\rv{G}_{k}\neq\circledast}\Proba(\xi^*=k)\\
	&=\frac{2}{\sigma^{2}}\sum_{a\ge 4}\frac{s_{a}}{a+1}\sum_{k \ge a}k(k-1)\binom{k-2}{a-1}t_{0}^{k-1}\\
	&=\frac{2}{\sigma^{2}}\frac{t_0}{(1-t_0)^3}\sum_{a\ge 4}s_{a}a(\tfrac{t_0}{1-t_0})^{a-1}=\frac{2}{\sigma^{2}}\frac{t_0}{(1-t_0)^3}\mathcal{S}'(\tfrac{t_0}{1-t_0})=\frac{2}{\sigma^{2}}(\kappa-\kappa^2(\kappa+2)),
	\end{split}
	\end{equation}
	where we used the formal power series identity $\sum_{k \ge a}k(k-1)\binom{k-2}{a-1}t^{k-1} = a(a+1) \frac{t^a}{(1-t)^{a+2}}$ to go from the third to the fourth line, 
	and in the last equality we used \cref{eq:proper_KAPPA} and $\kappa = \frac{t_0}{1-t_0}$.	
	
	It remains to compute the second term in \cref{eq:split_prob}. 
	Using the obvious notation $\mathcal{S}^{\leq k} = \cup_{j \leq k} \mathcal{S}^{j}$,
	we start by determining
	\begin{multline}
	\label{eq:third_part_prob}
	\Proba\big(\bm i_1,\bm i_2 \text{ are not in the same branch of } \bm{\lambda}(\bm u),\pat_{\{\bm j_1,\bm j_2\}}(\bm \alpha)=12\big|\bm{\lambda}(\bm u)\neq\circledast,d^{+}(\bm u)=k\big)\\
	=\sum_{\alpha\in\mathcal{S}^{\leq k}}\Proba\big(\bm i_1,\bm i_2 \text{ are not in the same branch }\text{of } \bm{\lambda}(\bm u),\pat_{\{\bm j_1,\bm j_2\}}(\bm \alpha)=12\big|\bm{\lambda}(\bm u)\in G_{\alpha}^k\big)\\
	\cdot\Proba\big(\bm{\lambda}(\bm u) \in G_{\alpha}^k\big|\bm{\lambda}(\bm u)\neq\circledast,d^{+}(\bm u)=k\big),
	\end{multline}
	where $G_{\alpha}^k$ denotes the set of $\cS$-gadgets of size $k$ with root-label $\alpha$. Trivially, 
	recalling that $\binom{k-1}{|\alpha|-1}$ is the number of $\cS$-gadgets with $k$ leaves and root decorated by $\alpha$, we have 
	\[
	\Proba\big(\bm{\lambda}(\bm u)\in G_{\alpha}^k\big|\bm{\lambda}(\bm u)\neq\circledast,d^{+}(\bm u)=k\big)=\frac{\binom{k-1}{|\alpha|-1}}{q_k-1}.
	\]	
	Using again the formula \eqref{eq:comb_comput},
	\begin{multline*}
	\Proba\big(\bm i_1,\bm i_2 \text{ are not in the same branch of }\bm{\lambda}(\bm u),\pat_{\{\bm j_1,\bm j_2\}}(\bm \alpha)=12\big|\bm{\lambda}(\bm u)\in G_{\alpha}^k\big)\\
	=\occ(12,\alpha)\Bigg(1-\frac{2}{{|\alpha|}+1}\cdot\frac{k-{|\alpha|}}{k-1}\Bigg),
	\end{multline*}
	where $\occ(12,\alpha)=\binom{|\alpha|}{2}^{-1} \socc(12,\alpha)$
	is the probability that two random elements of $\alpha$ do not form an inversion.
	Substituting the last two equations in \cref{eq:third_part_prob} we have
	\begin{multline*}
	\Proba\big(\bm i_1,\bm i_2 \text{ are not in the same branch of } \bm{\lambda}(\bm u),\pat_{\{\bm j_1,\bm j_2\}}(\bm \alpha)=12\big|\bm{\lambda}(\bm u)\neq\circledast,d^{+}(\bm u)=k\big)\\
	=\sum_{\alpha\in\mathcal{S}^{\leq k}}\occ(12,\alpha)\Bigg(1-\frac{2}{{|\alpha|}+1}\cdot\frac{k-{|\alpha|}}{k-1}\Bigg)\frac{\binom{k-1}{|\alpha|-1}}{q_k-1}.
	\end{multline*}
	Recalling that $\Prb{\bm{\lambda}(\bm u)\neq\circledast,d^{+}(\bm u)=k} = \Prb{\rv{G}_{k}\neq\circledast}\Proba(\xi^*=k) = (q_k-1)k(k-1)t_0^{k-1}/\sigma^2$, 
	it follows that
	\begin{equation}
	\begin{split}
	\label{eq:keyres3}
	\Proba\big(\bm{\lambda}(\bm u)\neq&\circledast,\;\bm i_1,\bm i_2 \text{ are not in the same branch of } \bm{\lambda}(\bm u),\pat_{\{\bm j_1,\bm j_2\}}(\bm \alpha)=12\big)\\
	&=\frac{1}{\sigma^2}\sum_{k \ge 4}\sum_{\alpha\in\mathcal{S}^{\leq k}}\occ(12,\alpha)\Bigg(1-\frac{2}{{|\alpha|}+1}\cdot\frac{k-{|\alpha|}}{k-1}\Bigg)\binom{k-1}{|\alpha|-1}k(k-1)t_0^{k-1}\\
	&=\frac{1}{\sigma^2}\sum_{\alpha\in\mathcal{S}}\occ(12,\alpha)\sum_{k \ge |\alpha|}\Bigg(\binom{k-1}{|\alpha|-1}-\frac{2}{{|\alpha|}+1}\cdot\binom{k-2}{|\alpha|-1}\Bigg)k(k-1)t_0^{k-1}\\
	&=\frac{2}{\sigma^2}\frac{t_0}{(1-t_0)^{4}}\sum_{\alpha\in\mathcal{S}}\socc(12,\alpha)\big(\tfrac{t_0}{1-t_0}\big)^{|\alpha|-2}=\frac{2}{\sigma^2}\kappa(1+\kappa)^3\Occ_{12}(\kappa), 
	\end{split}
	\end{equation}	
	where, to go from the third to the fourth line, we used 
	that $\socc(12,\alpha) = \occ(12,\alpha) \binom{|\alpha|}{2}$ and 
	the formal power series identity $\sum_{k \ge a}\Bigg(\binom{k-1}{a-1}-\frac{2}{{a}+1}\cdot\binom{k-2}{a-1}\Bigg)k(k-1)t^{k-1} = (a-1) a\frac{t^{a-1}}{(1-t)^{a+2}}$, 
	and in the last equality we used that $\kappa = \frac{t_0}{1-t_0}$ (\emph{i.e.} $t_0 = \frac{\kappa}{1+\kappa}$) and the definition of $\Occ_{12}$.
	Summing up the results in \cref{eq:keyres1,eq:keyres2,eq:keyres3} we conclude that \cref{eq:parm_p} holds.	
\end{proof}

\section{Local convergence}
\label{sec:local}

In this section we investigate the local limits of uniform permutations in a fixed substitution-closed class $\mathcal{C}$. 
We work under the following assumption.
\begin{assumption}
	\label{ass:type1}
	Consider the associated random variable $\xi$ defined by \cref{eq:offspring_distribution_packed_tree}.
	We assume that $\esper[\xi]=1$.
\end{assumption}
We highlight that in this section we do not assume the finite variance hypothesis (as done in \cref{sec:scaling}). See also \cref{prop: offspring_distr_charact} for an explicit characterization of this assumption.

Before stating our results, we recall in the following two sections the notions of 
local convergence for permutations and trees.

\subsection{Local limits for permutations}
\label{ssec:reminder_local_limits}
In this section we recall the definition of local topology for permutations recently introduced by Borga in \cite{Borga2019}.
We start by defining finite and infinite {\em rooted} permutations. 
Then we introduce a local distance and the corresponding notion of convergence for deterministic sequences of rooted and unrooted permutations.
Finally, we extend this notion of convergence
(in two non-equivalent ways) to sequences of \emph{random} unrooted permutations.

\begin{definition}
	A \emph{finite rooted permutation} is a pair $(\nu,i),$ where $\nu\in\SG^n$ and $i\in[n]$ for some $n\in\NN.$
\end{definition}

We denote  with $\SG^n_{\bullet}$ the set of rooted permutations of size $n$ and with $\SG_{\bullet}:=\bigcup_{n\in\NN}\SG^n_{\bullet}$ the set of finite rooted permutations. We write sequences of finite rooted permutations in $\SG_{\bullet}$ as $(\nu_n,i_n)_{n\in\NN}.$ 

To a rooted permutation $(\nu,i),$ we associate (as shown in the right-hand side of \cref{fig:rest}) the pair $(\Asi,\leqsi),$  where $\Asi\coloneqq[-i+1,|\nu|-i]$ is a finite interval containing 0 and $\leqsi$ is a total order on $\Asi,$ defined for all $\ell,j\in \Asi$ by
\begin{equation*}
\ell\leqsi j\qquad\text{if and only if}\qquad \nu({\ell+i})\leq\nu({j+i})\;.
\end{equation*} 
Informally, the elements of $\Asi$ should be thought of as the column indices of the diagram of $\nu$,
shifted so that the root is in column $0$.
The order $\leqsi$ then corresponds to the vertical order on the dots in the corresponding columns.
\begin{figure}[htbp]
	\centering
	\includegraphics[scale=.70]{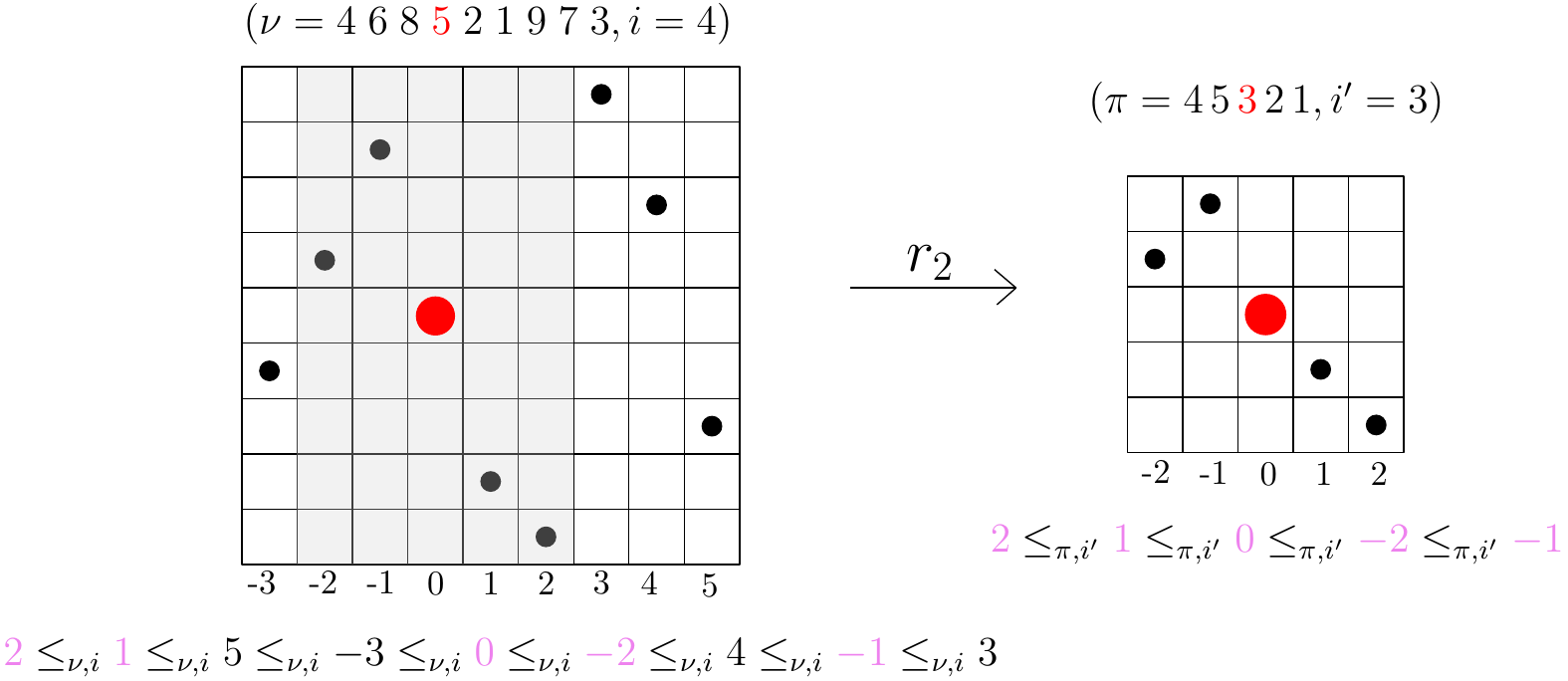}\\
	\caption{Two rooted permutations and the associated total orders. The big red dot indicates the root of the permutation. The vertical grey strip and the relation between the two rooted permutations will be clarified later.  \label{fig:rest}}
\end{figure}

Clearly this map is a bijection from the space of finite rooted permutations $\SG_{\bullet}$ to the space of total orders on finite integer intervals containing zero. Consequently and throughout the paper, we identify every rooted permutation  $(\nu,i)$ with the total order $(\Asi,\leqsi).$

Thanks to the identification between rooted permutations and total orders, the following definition of infinite rooted permutation is natural. 

\begin{definition}
	We call \emph{infinite rooted permutation} a pair $(A,\preccurlyeq)$ where $A$ is an infinite interval of integers containing 0 and $\preccurlyeq$ is a total order on $A$. We denote the set of infinite rooted permutations by $\SG^{\infty}_\bullet.$
\end{definition} 
We highlight that infinite rooted permutations can be thought of as rooted at 0. We set 
$$\tilde{\SG_{\bullet}}\coloneqq\Sr\cup\SG^\infty_{\bullet},$$
which is the set of all (finite and infinite) rooted permutations. 

We now introduce the following \textit{restriction function around the root} defined, for every $h\in\NN$, as follows
\begin{equation}
\label{rhfunct}
\begin{split}
r_h \colon\quad &\tilde{\SG_{\bullet}}\;\longrightarrow \qquad \;\SG_\bullet;\\
(A,&\preccurlyeq) \mapsto \big(A\cap[-h,h],\preccurlyeq\big)\;.
\end{split}
\end{equation}
We can think of restriction functions as a notion of neighbourhood around the root.
For finite rooted permutations we also have the equivalent description of the restriction functions $r_h$ in terms of consecutive patterns: if $(\nu,i)\in\Sr$ then $r_h(\nu,i)=(\text{pat}_{[a,b]}(\nu),i)$ where we take \hbox{$a=\max\{1,i-h\}$} and $b=\min\{|\nu|,i+h\}.$

The \emph{local distance} $d_p$ on the set of (possibly infinite) rooted permutations $\tilde{\SG_{\bullet}}$ is defined as follows: 
given two rooted permutations $(A_1,\preccurlyeq_1),(A_2,\preccurlyeq_2)\in\tilde{\SG_{\bullet}},$
\begin{equation}
\label{distance}
d_p\big((A_1,\preccurlyeq_1),(A_2,\preccurlyeq_2)\big)=2^{-\sup\big\{h\in\NN\;:\;r_h(A_1,\preccurlyeq_1)=r_h(A_2,\preccurlyeq_2)\big\}},
\end{equation}
with the classical conventions that $\sup\emptyset=0,$ $\sup\NN=+\infty$ and $2^{-\infty}=0.$ 
The metric space $(\tilde{\SG_{\bullet}},d_p)$ is a compact space and so Polish, \emph{i.e.}, separable and complete (see \cite[Theorem 2.16]{Borga2019}).

The above distance entails a notion of convergent sequences of {\em rooted} permutations.
For a sequence $\nu_n$ of {\em unrooted} permutations,
we consider the sequence of {\em random} rooted permutations $(\nu_n,\bm i_n)$,
where $\bm i_n$ is a uniform random index of $\nu_n$.
We say that $\nu_n$ converges in the Benjamini--Schramm
sense if the sequence of {\em random} rooted permutations $(\nu_n,\bm i_n)$ 
converges in distribution for the above distance $d_p$.
This definition is inspired from Benjamini--Schramm convergence for graphs (see \cite{benjamini2001recurrence}).

Benjamini--Schramm convergence can be extended in two different ways 
for sequences of random permutations $(\bm \nu_n)_{n \ge 1}$:
the \emph{annealed} and the \emph{quenched}  version of the Benjamini--Schramm convergence.
These two different versions come from the fact that there are two sources of randomness, 
one for the choice of the random permutation $\bm \nu_n$, and one for the random root $\bm i_n$.
Intuitively, in the annealed version, the random permutation
and the random root are taken simultaneously,
while in the quenched version, the random permutation should be thought as frozen
when we take the random root.

We now give the formal definitions.
In both cases, $(\bm{\nu}_n)_{n\in\NN}$ denotes a sequence of random permutations in $\SG$ and $\bm{i}_n$ denotes a uniform index of $\bm{\nu}_n,$ \emph{i.e.,} a uniform integer in $[1,|\bm{\nu}_n|]$.

\begin{definition}[Annealed version of the Benjamini--Schramm convergence]\label{defn:weakweakconv}
	We say that $(\bm{\nu}_n)_{n\in\NN}$ \emph{converges in the annealed Benjamini--Schramm sense} to a random variable $\bm{\nu}_{\infty}$ with values in $\Sri$ if the sequence of random variables $(\bm{\nu}_n,\bm{i}_n)_{n\in\NN}$ converges in distribution to $\bm{\nu}_{\infty}$ with respect to the local distance $d_p$.
	In this case we write $\bm{\nu}_n\stackrel{aBS}{\longrightarrow}\bm{\nu}_\infty$ instead of  $(\bm{\nu}_n,\bm{i}_n)\convdis\bm{\nu}_\infty.$
\end{definition}

\begin{definition}[Quenched version of the Benjamini--Schramm convergence]
	\label{strongconv}
	We say that $(\bm{\nu}_n)_{n\in\NN}$ \emph{converges in the quenched Benjamini--Schramm sense} 
	to a random measure $\bm{\mu}^\infty$ on $\Sri$
	if the sequence of conditional laws $\big(\mathcal{L}\big((\bm{\nu}_n,\bm{i}_n)\big|\bm{\nu}_n\big)\big)_{n\in\NN}$
	converges in distribution to $\bm{\mu}^\infty$ with respect to the weak topology induced by the local distance $d_p$.
	In this case we write $\bm{\nu}_n\stackrel{qBS}{\longrightarrow}\bm{\mu}^\infty$ instead of $\mathcal{L}\big((\bm{\nu}_n,\bm{i}_n)\big|\bm{\nu}_n\big)\convdis\bm{\mu}^\infty$.
\end{definition}

We highlight that, in the annealed version, the limiting object is a {\em random variable} with values in $\Sri$,
while for the quenched version, the limiting object $\bm{\mu}^{\infty}$ is a {\em random measure} on $\Sri$.

We have the following characterizations of the two versions of the Benjamini--Schramm convergence 
\cite[Section 2.5]{Borga2019}. 
Recall that $\widetilde{\coc}(\pi,\nu)$ denotes the proportion of consecutive occurrences of a pattern $\pi$ in $\nu$, 
namely, 
\[\widetilde{\coc}(\pi,\nu) = \frac{\coc(\pi,\nu)}{n}=\frac{1}{n}\mathrm{card}\big\{\text{intervals } I \subseteq [n] \text{ s.t. pat}_I(\nu)=\pi\big\}.\]

\begin{theorem}
	\label{thm:local_conv_perm_charact}
	For any $n\in\NN$, let $\bm{\nu}_n$ be a random permutation of size $n$. Then
	\begin{enumerate}
		\item The sequence $(\bm{\nu}_n)_{n\in\NN}$ converges in the annealed Benjamini--Schramm sense
		to some $\bm{\nu}_{\infty}$ if and only if
		there exist non-negative real numbers $(\Delta_{\pi})_{\pi\in\SG}$ such that 
		$$\E[\widetilde{\coc}(\pi,\bm{\nu}_n)]\to\Delta_{\pi},\quad\text{for all patterns}\quad\pi\in\SG.$$
		\item The sequence $(\bm{\nu}_n)_{n\in\NN}$ converges in the quenched Benjamini--Schramm sense
		to some $\bm{\mu}^\infty$ if and only if      
		there exist non-negative real random variables $(\bm{\Lambda}_\pi)_{\pi\in\SG}$ such that
		$$\big(\widetilde{\coc}(\pi,\bm{\nu}_n)\big)_{\pi\in\SG}\convdis(\bm{\Lambda}_{\pi})_{\pi\in\SG},$$
		w.r.t. the product topology. 
	\end{enumerate}
\end{theorem}
Since the variables $\widetilde{\coc}(\pi,\bm{\nu}_n)$ take values in $[0,1]$,
the quenched Benjamini--Schramm convergence implies the annealed one.

The goal of the following sections is to prove that
a sequence of uniform permutations in a substitution-closed class converges
in the quenched Benjamini--Schramm sense 
using the packed trees representing permutations.
To this end, we need to introduce a local topology for trees. 

\subsection{Local limits for decorated trees}
In this section we introduce a local topology for decorated trees with a distinguished {\em leaf}
(called pointed trees in the sequel).
This is a straight-forward adaptation
of that for trees with a distinguished {\em vertex}
introduced by Stufler in \cite{stufler2019}.

Following the presentation in \cite[Section 6.3.1]{stufler2016limits}, 
we start by defining an infinite pointed plane tree $\Uinf$ (see \cref{fig:Stufler_tree} below).
This infinite tree is meant to be a pointed analogue of Ulam--Harris tree, so that
pointed trees will be seen as subsets of it.
To construct $\Uinf$, we take a spine $(u_i)_{i\geq 0}$ that grows
downwards, that is, such that $u_i$ is the parent of $u_{i-1}$ for all $i\geq 1$. Any vertex $u_i,$ with $i\geq 1,$ has
an infinite number of children to the left and to the right  of its distinguished offspring $u_{i-1}$. 
The former are ordered from right to left and denoted by $(v^{i}_{L,j})_{j\geq1}$, the latter are ordered from left to right and denoted by $(v^{i}_{R,j})_{j\geq1}$. 
Each of these vertices not belonging to the spine $(u_i)_{i\geq 0}$ is the root of a copy of the Ulam--Harris tree $\mathcal{U}_{\infty}$. 
We always think of $\Uinf$ as a tree with distinguished leaf $u_0.$

\begin{figure}[htbp]
	\centering
	\includegraphics[height=7cm]{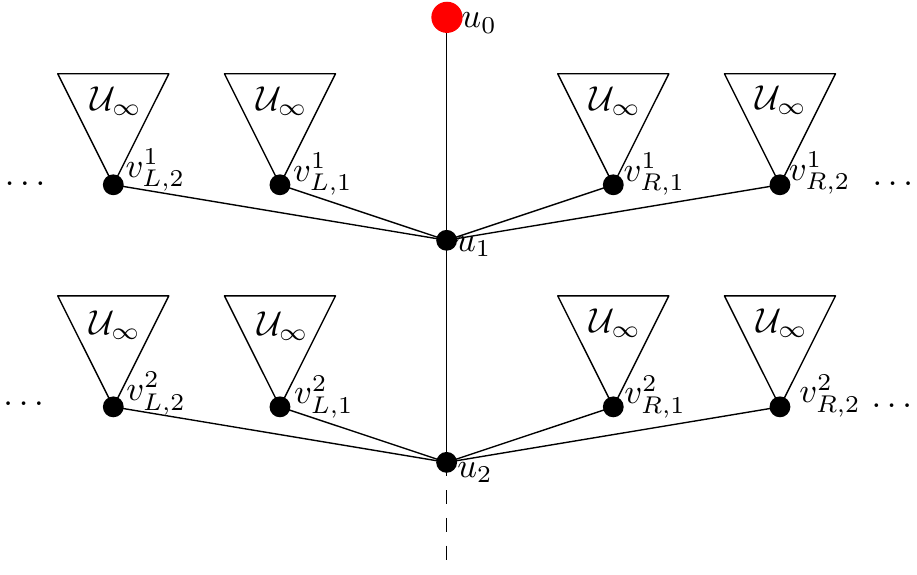}
	\caption{A schema of the infinite plane tree $\Uinf$.}
	\label{fig:Stufler_tree}
\end{figure}

\begin{definition}
	A \emph{(possibly infinite) pointed plane tree} $T^{\bullet}$ is a subset of $\Uinf$ such that
	\begin{itemize}
		\item $u_0\in T^{\bullet}.$
		\item if $u_p\in T^{\bullet}$ then $u_i\in T^{\bullet},$ for all $0\leq i \leq p.$
		\item if $v^{i}_{L,q}\in T^{\bullet}$ (resp. $v^{i}_{R,q}\in T^{\bullet}$) then $u^{i}\in T^{\bullet}$ and $v^{i}_{L,j}\in T^{\bullet}$ (resp. $v^{i}_{R,j}\in T^{\bullet}$) for all $1\leq j \leq q.$
		\item Any maximal subset of $T^{\bullet}$ contained in 
		one of the Ulam--Harris trees $\mathcal{U}_{\infty}$  of $\Uinf$ is a plane tree.
	\end{itemize}
	We denote with $\mathfrak{T}^\bullet$ \label{def:mfTb} the space of (possibly infinite) pointed plane trees.
\end{definition}
We say that a pointed tree $T^{\bullet}$ in $\mathfrak{T}^\bullet$ is {\em locally and upwards finite} if 
every vertex has finite degree and the intersection of $T^\bullet$ with any one of the Ulam--Harris 
trees $\mathcal{U}_{\infty}$  of $\Uinf$ is finite.
The set of locally and upwards finite pointed trees will be denoted by $\lufT$.\label{def:lufT}

Any finite plane tree $T$ together with a distinguished leaf $v_0$
may be interpreted in a canonical way as a pointed plane tree $T^\bullet$, such that $v_0$ is mapped to $u_0$.
In particular, the backward spine $u_0,u_1,\cdots$ of the associated pointed plane tree $T^\bullet$ is finite and ends at the root of $T$.
\medskip

Next, we need to extend this notion to decorated trees.
Let $\DDD$ be a combinatorial class.
We define $\DDD$-decorated locally and upwards finite pointed trees, 
as a tree $T^\bullet$ in $\lufT$, endowed with a decoration function
$\dec: \Vint({T}^\bullet) \to \DDD$, such that, for each $v$ in $\Vint(T^{\bullet})$, the outdegree of $v$ is exactly $\size(\dec(v))$.
We denote such a tree with the pair $(T^\bullet,\dec)$ and the space of such trees as $\lufTD$. \label{def:luftd}
As above, a decorated tree with a distinguished leaf can be identified with an element of this set.

Given a $\DDD$-decorated pointed tree $(T^{\bullet},\lambda_{T^{\bullet}})\in\lufTD$, we denote with $f_h^{\bullet}(T^{\bullet},\lambda_{T^{\bullet}})$ the $\DDD$-decorated pointed tree $(f_h^{\bullet}(T^{\bullet}),f_h^{\bullet}(\lambda_{T^{\bullet}})),$ where $f_h^{\bullet}(T^{\bullet})$ is the pointed fringe subtree rooted at $u_h$ with distinguished leaf $u_0$ (if $u_h$ is not well-defined because $u_h\notin T^\bullet$, we set $f_h^{\bullet}(T^\bullet)=f_m^{\bullet}(T^\bullet),$ where $m$ is the maximal index such that $u_m\in T^\bullet$) 
and $f_h^{\bullet}(\lambda_{T^{\bullet}})$ is $\lambda_{T^{\bullet}}$ restricted to the domain ${\Vint(f_h^{\bullet}(T^{\bullet}))}$.
We note that, for any given $h$, the image set $f_h^{\bullet}(\lufTD)$ is countable.

We endow the space $\lufTD$ with the \emph{local distance} $d_t$ defined, for all $(T^{\bullet}_1,\lambda_{T^{\bullet}_1}),(T^{\bullet}_2,\lambda_{T^{\bullet}_2})\in\lufTD$, by
\begin{equation}
\label{eq:local_dist_for_trees}
d_t\big((T^{\bullet}_1,\lambda_{T^{\bullet}_1}),(T^{\bullet}_2,\lambda_{T^{\bullet}_2})\big)=2^{-\sup\{h\geq0\;:\;f^\bullet_h(T^{\bullet}_1,\lambda_{T^{\bullet}_1})=f^\bullet_h(T^{\bullet}_2,\lambda_{T^{\bullet}_2})\}},
\end{equation}
with the classical conventions that $\sup\emptyset=0,$ $\sup\NN=+\infty$ and $2^{-\infty}=0$.

\begin{remark}
	The distance defined in \cref{eq:local_dist_for_trees} can be trivially restricted also to the space of non-decorated pointed trees. We point out that this distance is not equivalent to the distance considered in \cite[Section 6.3.1]{stufler2016limits} for the space of non-decorated pointed trees. For instance, if $S_n$, $1\leq n\leq \infty$ is a star where the root has outdegree $n$ and its children all have outdegree 0, then the sequence $(S_n)_{n\geq 1}$ does not converge for our metric (and has no convergent subsequences). This implies that our space is not compact. On the contrary, the space of pointed trees endowed with the distance defined by Stufler in \cite{stufler2016limits} is compact. 
	
	We also note (without proof since we do not need this result) that in the subspace of locally finite pointed trees the two distances are topologically equivalent. A proof of this result would be an easy adaptation of \cite[Lemma 6.2]{janson2012simply}.
\end{remark}

\begin{proposition}
	The space $(\lufTD,d_t)$ is a Polish space.
\end{proposition}
\begin{proof}
	The separability is trivial since $\uplus_{h \ge 1} f_h^{\bullet}(\lufTD)$ is a countable dense set.
	The completeness follows from the fact that the space $\lufTD$ is a closed subspace of a countable product of discrete sets (which is complete) via the map $(T^{\bullet},\lambda_{T^{\bullet}})\to\big(f_h^\bullet(T^{\bullet},\lambda_{T^{\bullet}})\big)_{h\geq 1}.$
\end{proof}

We end this section defining two versions 
of the local convergence (similar to those previously defined for permutations) for random decorated trees 
with a uniform random distinguished leaf.
In both definition,
$(\bm{T}_n,\bm{\lambda}_n)_{n\in\mathbb{N}}$ is a sequence of random (finite) $\DDD$-decorated trees
and $\bm \ell_n$ is a uniform random leaf of $(\bm{T}_n,\bm{\lambda}_n)$.

\begin{definition}[Annealed Benjamini--Schramm convergence]
	We say that $(\bm{T}_n,\bm{\lambda}_n)_{n\in\mathbb{N}}$ converges in the annealed Benjamini--Schramm sense to a random variable $(\bm{T}^{\bullet}_\infty,\bm{\lambda}_\infty)$
	with values in $\lufTD$ if the sequence of random pointed $\DDD$-decorated trees $((\bm{T}^{\bullet}_n,\bm{\lambda}_n),\bm \ell_n)_{n\in\mathbb{N}}$ converges in distribution to $(\bm{T}^{\bullet}_\infty,\bm{\lambda}_\infty)$
	with respect to the local distance defined in \cref{eq:local_dist_for_trees}.
\end{definition}

\begin{definition}[Quenched Benjamini--Schramm convergence]
	We say that $(\bm{T}_n,\bm{\lambda}_n)_{n\in\mathbb{N}}$ converges in the quenched Benjamini--Schramm sense to a random measure $\bm{\mu}^{\bullet}_{\infty}$
	on $\lufTD$ if the sequence of conditional probability distributions  $\mathcal{L}\big(((\bm{T}_n,\bm{\lambda}_n),\bm \ell_n)\big|(\bm{T}_n,\bm{\lambda}_n)\big)_{n\in\mathbb{N}}$ converges in distribution to $\bm{\mu}^{\bullet}_\infty$
	with respect to the weak topology induced by the local distance defined in \cref{eq:local_dist_for_trees}.
\end{definition}
Again, the quenched version is stronger than the annealed one.

\begin{remark}
	It would also be natural, and closer to the usual notion of Benjamini--Schramm convergence
	in the literature,
	to distinguish a uniform random vertex $\bm v_n$
	rather than a uniform random leaf $\bm \ell_n$ as above.
	The leaf version is however what we need here for our application to permutations.
\end{remark}

\subsection{Local convergence around a uniform leaf for random packed trees conditioned to the number of leaves}
\label{sec:limiting_local_trees}
We begin this section by constructing the limiting random pointed packed tree $\bm{P}^\bullet_\infty=(\bm{T}^\bullet_\infty,\bm{\lambda}_\infty)$. 
This tree will be the limit of the sequence of uniform packed trees $(\bm{T}_n,\bm{\lambda}_{\bm{T}_n})$ considered in \cref{lem:unif_packed_tree} pointed at a random leaf.

We recall that $\xi$ denotes the random variable defined in \cref{eq:offspring_distribution_packed_tree} and $\bm{T}$ denotes the associated $\xi$-Galton--Watson tree.
Additionally, we recall that the random variable $\hat{\xi}$ defined in \cref{eq:xihat} is the \emph{size-biased} version of $\xi$.

We define the random tree $\bm{T}^\bullet_\infty$ in the space $\lufT$ as follows. Let $u_0$ be the distinguished leaf. For each $i\geq 1$, we let $u_i$ receive offspring according to an independent copy of $\hat{\xi}$. 
The vertex $u_{i-1}$ gets identified with an offspring of $u_i$ chosen uniformly at random. All other offspring vertices of $u_i$ become roots of independent copies of the Galton--Watson tree $\bm{T}$. 

Conditionally on $\bm{T}^\bullet_\infty,$ the random decoration $\bm{\lambda}_{\infty}(v)$ of each internal vertex $v$ of $\bm{T}^\bullet_\infty$ gets drawn uniformly at random among all $d^+_{\bm{T}^\bullet_\infty}(v)$-sized decorations in $\widehat{\GGG(\mathcal{S})}$ independently of all the other decorations ($\widehat{\GGG(\mathcal{S})}$ was introduced after \cref{def:S_plus_decoration}). This construction yields a random infinite locally and upwards finite pointed packed tree.

We refer to the sequence of (decorated) vertices $(u_i)_{i \geq 0}$ as the \emph{infinite spine} of $\bm{P}^\bullet_\infty=(\bm{T}^\bullet_\infty,\bm{\lambda}_\infty)$.

\bigskip

To simplify notation, we denote the space $\mathfrak{T}^{\bullet,\text{\tiny luf}}_{\widehat{\GGG(\SSS)}}$ of (possibly infinite) locally and upwards finite pointed packed trees as $\lufPT$. \label{def:pluft}

\begin{proposition}
	\label{prop:quenched_conv_tree}
	Let $\bm{P}_n=(\bm{T}_n,\bm{\lambda}_{\bm{T}_n})$ be the random packed tree considered in \cref{lem:unif_packed_tree} and  $\bm{P}^\bullet_\infty=(\bm{T}^\bullet_\infty,\bm{\lambda}_{\infty})$ be the limiting random pointed packed tree constructed above. It holds that
	\begin{equation}
	\label{eq:random_meas_conv}
	\mathcal{L}\big((\bm{P}_n,\bm{\ell}_n)|\bm{P}_n\big)\stackrel{P}{\longrightarrow}\mathcal{L}(\bm{P}^\bullet_\infty),
	\end{equation}
	where $\bm{\ell}_n$ is a uniform leaf of $\bm{P}_n$ chosen independently of $\bm{P}_n.$
	
	In particular, $\bm{P}_n$ converges in the quenched Benjamini--Schramm sense 
	to the deterministic measure $\mathcal{L}(\bm{P}^\bullet_\infty)$
	and in the annealed Benjamini--Schramm sense to the random tree $\bm{P}^\bullet_\infty.$  
\end{proposition}

Note that the $\mathcal{L}(\bm{P}^\bullet_\infty)$ is a measure on $\lufPT$.
Since the limiting object in quenched Benjamini--Schramm convergence is in general a random measure on $\lufPT$,
it should be interpreted as a constant random variable, 
equal to the measure $\mathcal{L}(\bm{P}^\bullet_\infty)$.

\begin{proof}[Proof of \cref{prop:quenched_conv_tree}]
	The sequence	$\mathcal{L}\big((\bm{P}_n,\bm{\ell}_n)|\bm{P}_n\big)_{n\in\NN}$ is a sequence of random probability measures on the Polish space $(\lufPT,d_t).$ 
	The set of closed and open balls
	\begin{equation*}
	\label{convergece_determ}
	\mathcal{B}=\Big\{B\big((T^\bullet,\lambda_{T^\bullet}),2^{-h}\big):h\in\NN,(T^\bullet,\lambda_{T^\bullet})\in\lufPT\Big\}
	\end{equation*}
	is a convergence-determining class for the space $(\lufPT,d_t)$, 
	\emph{i.e.,} for every probability measure $\mu$ and every sequence of probability measures $(\mu_n)_{n\in\NN}$ on $\lufPT$, 
	the convergence $\mu_n(B)\to\mu(B)$ for all $B\in\mathcal{B}$ implies $\mu_n\to\mu$ w.r.t. the weak-topology. This is a trivial consequence of the monotone class theorem and the fact that the intersection of two balls in $\lufPT$ is either empty or one of them.
	
	Therefore, using \cite[Theorem 4.11]{kallenberg2017random}, the convergence in \cref{eq:random_meas_conv} is equivalent
	to the following convergence, for all $k\in\NN$ and for all vectors of balls $(B_i)_{1\leq i\leq k}\in\mathcal{B}^k$:
	\begin{equation*}
	\Big(\mathcal{L}\big((\bm{P}_n,\bm{\ell}_n)|\bm{P}_n\big)(B_i)\Big)_{1\leq i\leq k}\stackrel{d}{\longrightarrow}\Big(\mathcal{L}(\bm{P}^\bullet_\infty)(B_i)\Big)_{1\leq i\leq k}.
	\end{equation*}
	Since the limiting vector in the above equation is deterministic, the above convergence in distribution is equivalent to the convergence in probability. Finally, standard properties of the convergence in probability imply that
	it is enough to show the component-wise convergence, \emph{i.e.,} for all $B\in\mathcal{B}$, 
	\begin{equation}
	\label{eq:goaloftheproof}
	\mathcal{L}\big((\bm{P}_n,\bm{\ell}_n)|\bm{P}_n\big)(B)\stackrel{P}{\longrightarrow}\mathcal{L}(\bm{P}^\bullet_\infty)(B).
	\end{equation}
	
	Fix a ball $B=B\big((T^\bullet,\lambda_{T^\bullet}),2^{-h}\big)\in\mathcal{B}$ and note that \cref{eq:goaloftheproof} 
	(which we need to prove) rewrites as
	\begin{equation}
	\label{eq:goal1oftheproof}
	\Proba\big(f_h^{\bullet}(\bm{P}_n,\bm{\ell}_n)=f_h^{\bullet}(T^\bullet,\lambda_{T^\bullet})\big|\bm{P}_n\big)\stackrel{P}{\longrightarrow}\Proba\big(f^\bullet_h(\bm{P}^\bullet_\infty)=f^\bullet_h(T^\bullet,\lambda_{T^\bullet})\big).
	\end{equation}
	(The left-hand side is a function of $\bm{P}_n$, and hence, a random variable; the right-hand side is a number.)
	
	W.l.o.g. we can assume that $f_h^{\bullet}(T^\bullet,\lambda_{T^\bullet})=(T^\bullet,\lambda_{T^\bullet})$. Denoting $\mathfrak{L}(\bm{P}_n)$ the set of leaves of $\bm{P}_n$,
	the left-hand side writes
	\begin{equation}
	\begin{split}
	\Proba\big(f_h^{\bullet}(\bm{P}_n,\bm{\ell}_n)=(T^\bullet,\lambda_{T^\bullet})\big|\bm{P}_n\big)&=\frac{\big|\big\{\ell\in \mathfrak{L}(\bm{P}_n):f_h^{\bullet}(\bm{P}_n,\ell)=(T^\bullet,\lambda_{T^\bullet})\big\}\big|}{n}\\
	&=\frac{1}{n}\sum_{\ell\in \mathfrak{L}(\bm{P}_n)}\mathds{1}_{\{f_h^{\bullet}(\bm{P}_n,\ell)=(T^\bullet,\lambda_{T^\bullet})\}}\\
	&=\frac{1}{n}\sum_{\ell\in \mathfrak{L}(\bm{T}_n)}\mathds{1}_{\{f_h^{\bullet}(\bm{T}_n,\ell)=T^\bullet\}}\mathds{1}_{\{\bm{\lambda}_{f_h^{\bullet}(\bm{T}_n,\ell)}=\lambda_{T^\bullet}\}}.
	\end{split}
	\end{equation}
	For a vertex $v$ of $\bm{T}_n$, we denote by $f(\bm{T}_n,v)$ the fringe subtree rooted at $v$
	and by $f(\bm{\lambda}_{(\bm{T}_n,v)})$ the map $\lambda_{|_{\Vint(f(\bm{T}_n,v))}}$.
	Let also $T$ be the unpointed version of $T^{\bullet}$.
	Note that a leaf $\ell\in \mathfrak{L}(\bm{T}_n)$ satisfies $f_h^{\bullet}(\bm{T}_n,\ell)=T^\bullet$
	if and only if its $h$-th ancestor $v$ satisfies $f(\bm{T}_n,v)=T$.
	Additional, to any $v$ with $f(\bm{T}_n,v)=T$ corresponds exactly one leaf $\ell$ with $f_h^{\bullet}(\bm{T}_n,\ell)=T^\bullet$
	(which is determined by the pointing).
	Therefore we can rewrite the last term of the above equation as
	\begin{equation}
	\frac{1}{n}\sum_{v\in \bm{T}_n}\mathds{1}_{\{f(\bm{T}_n,v)=T\}}\mathds{1}_{\{f(\bm{\lambda}_{(\bm{T}_n,v)})=\lambda_{T}\}}.
	\end{equation}
	By \cite[Rem. 1.9]{stufler2019offspring}, we have that 
	$$\frac{1}{n}\sum_{v\in \bm{T}_n}\mathds{1}_{\{f(\bm{T}_n,v)=T\}}\stackrel{P}{\longrightarrow}\Proba\big(f^\bullet_h(\bm{T}^\bullet_\infty)=f^\bullet_h(T^\bullet)\big).$$
	
	Noting that all fringe subtrees of $\bm{T}_n$ that are equal to $T$ are necessarily disjoint and that, conditioning on $f(\bm{T}_n,v)=T,$ then $f(\bm{\lambda}_{(\bm{T}_n,v)})=\lambda_{T}$ with probability $p,$ independently from the rest 
	(specifically $p = \prod_{u \in T} q_{d^+_T(u)}^{-1}$), we can conclude using Chernoff concentration bounds that 
	\begin{multline*}
	\frac{1}{n}\sum_{v\in \bm{T}_n}\mathds{1}_{\{f(\bm{T}_n,v)=T\}}\mathds{1}_{\{f(\bm{\lambda}_{(\bm{T}_n,v)})=\lambda_{T}\}}\\
	\stackrel{P}{\longrightarrow}p\cdot\Proba\big(f^\bullet_h(\bm{T}^\bullet_\infty)=f^\bullet_h(T^\bullet)\big)=\Proba\big(f^\bullet_h(\bm{P}^\bullet_\infty)=f^\bullet_h(T^\bullet,\lambda_{T^\bullet})\big),
	\end{multline*}
	where the last equality follows from the construction of the map $\bm{\lambda}_{\infty}.$   
\end{proof}

\subsection{The continuity of the bijection between packed trees and $\oplus$-indecomposable permutations}
\label{sec:continuity_bij}

In this section we consider a substitution-closed class $\mathcal{C}$ different from the class of separable permutations. 
The latter case will be considered separately in \cref{sec:map_for_sep_perm}.
We recall that $\DT:=\Pack\circ\CanTree$ is the bijection presented in \cref{le:bij_perm_tree} between $\oplus$-indecomposable permutations of $\mathcal{C}$ and finite packed trees. 

The goal of this section is to extend the bijection $\DT^{-1}$ as a function $\RP$ from the metric space of (possibly infinite) locally and upwards finite pointed packed trees $(\lufPT,d_t)$ to the metric space of (possibly infinite) rooted permutations $(\tilde{\SG}_{\bullet},d_p)$.

First, we need to deal with the introduction of a root in permutations (resp. a pointed leaf in trees) on finite objects. 
This is very simple, and we extend $\DT^{-1}$ as a function $\RP$ 
from finite pointed packed trees to finite rooted permutations as follows. 
We recall (see \cref{rk:Leaves_Elements})
that the $i$-th leaf $\ell$ of a packed tree $P=\DT(\nu)$ corresponds to
the $i$-th element of the permutation $\nu$.
Therefore the following definition is natural:
\begin{equation}
\label{eq:G_extension_1}
\RP(\PackedTree,\ell):=(\DT^{-1}(\PackedTree),i).
\end{equation} 

Given an infinite pointed packed tree $\PackedTree^{\bullet}$ with infinitely many $\cS$-gadget decorations on its infinite spine, we consider the sequence of pointed subtrees 
$$\big(f^{\bullet}_{s(h)}(\PackedTree^{\bullet})\big)_{h\in\mathbb{N}}$$
consisting of all restrictions for $s(h)\in\mathbb{N}$ such that $f^{\bullet}_{s(h)}(\PackedTree^{\bullet})$ has root decorated with an $\cS$-gadget.

\begin{lemma}
	\label{lem:limexistence}
	Let $\PackedTree^{\bullet}$ be an infinite pointed packed tree.
	Then the (deterministic) sequence of rooted permutations $\big(\RP(f^{\bullet}_{s(h)}(\PackedTree^{\bullet}))\big)_{h\in\mathbb{N}}$ converges
	in the Benjamini--Schramm sense, as $h$ tends to $+\infty$.
\end{lemma}
\begin{proof}
	In \cref{ssec:patterns_subtrees},
	we saw that the pattern associated to a set $I$
	of leaves of a packed tree only depends
	on any fringe subtree 
	containing all leaves in $I$
	and rooted at a vertex decorated with an $\cS$-gadget.
	This implies that the family $\big(\RP(f^{\bullet}_{s(h)}(\PackedTree^{\bullet}))\big)_{h\in\mathbb{N}}$ of elements in $\Sr$ is \emph{consistent}, \emph{i.e.,}
	for all $h\in\mathbb{N}$, there exists an integer $k(h)$ (the half-width of the restriction strip) such that $r_{k(h)}(\RP(f^{\bullet}_{s(h+1)}(\PackedTree^{\bullet})))=\RP(f^{\bullet}_{s(h)}(\PackedTree^{\bullet}))$.
	By \cite[Proposition 2.12]{Borga2019}, this implies the existence of a limit,
	which is what we wanted to prove.
\end{proof}
This lemma allows to define, for an infinite pointed packed tree $\PackedTree^{\bullet}$ having infinitely many $\cS$-gadget decorations on its infinite spine,
\begin{equation}
\label{eq:G_extension_2}
\RP(\PackedTree^{\bullet}):=\lim_{h\to\infty}\RP(f^{\bullet}_{s(h)}(\PackedTree^{\bullet})).
\end{equation}

We now investigate the continuity of the function $\RP$ with respect to the local topologies. Note that $\RP$ is defined only for finite pointed packed trees and infinite pointed packed tree with infinitely many $\cS$-gadget decorations on the infinite spine. This will not be an issue: indeed, we will use the map on a sequence of random pointed packed trees that converges to $\bm{P}^\bullet_\infty$ and this limiting pointed packed tree has a.s.\ infinitely many $\cS$-gadget decorations on the infinite spine. We start with a definition and a lemma which characterize a certain regularity property of the map $\RP$. 

\begin{definition}
	Given a finite plane tree $T$ we say that a leaf $\ell_1$ is \emph{before} (resp. \emph{after}) a leaf $\ell_2$ if the
	post-order label of $\ell_1$ (see \cref{rk:Leaves_Elements}) is smaller (resp. greater) than the post-order label of $\ell_2$.
\end{definition}

\begin{lemma}
	\label{lem:local_reg}
	Fix $k \ge 0$.
	Let $\PackedTree^{\bullet}_1,\PackedTree^{\bullet}_2\in\lufPT$ be two pointed packed trees  such that for some $h>0$ it holds that:
	\begin{itemize}
		\item $f_h^{\bullet}(\PackedTree^{\bullet}_1)=f_h^{\bullet}(\PackedTree^{\bullet}_2)$;
		\item $f_h^{\bullet}(\PackedTree^{\bullet}_1)$ contains at least $k$ leaves before and $k$ leaves after the distinguished leaf;
		\item the root of $f_h^{\bullet}(\PackedTree^{\bullet}_1)$ is decorated with an $\cS$-gadget.
	\end{itemize} 
	Then $d_p(\RP(\PackedTree^{\bullet}_1),\RP(\PackedTree^{\bullet}_2))\leq 2^{-k}.$
\end{lemma}

\begin{proof}
	Since $f_h^{\bullet}(\PackedTree^{\bullet}_1)$ contains at least $k$ leaves before and $k$ leaves after the distinguished leaf, and its root is decorated with an $\cS$-gadget, the restriction $r_k(\RP(\PackedTree^{\bullet}_1))$ is equal to the restriction $r_k(\RP(f_h^{\bullet}(\PackedTree^{\bullet}_1)))$.
	Of course, the same holds for $\PackedTree^{\bullet}_2$.
	This follows from the discussion of \cref{ssec:patterns_subtrees}.
	
	Since $f_h^{\bullet}(\PackedTree^{\bullet}_1)=f_h^{\bullet}(\PackedTree^{\bullet}_2)$, it then follows that $r_k(\RP(\PackedTree^{\bullet}_1))=r_k(\RP(\PackedTree^{\bullet}_2))$. So, by definition of~$d_p,$ we conclude that $d_p(\RP(\PackedTree^{\bullet}_1),\RP(\PackedTree^{\bullet}_2))\leq 2^{-k}.$ 
\end{proof}

We now set
\begin{equation}
\begin{split}
C_{\RP}:=\big\{&\PackedTree^{\bullet}\in\lufPT:\forall k>0,\;\exists h(k)>0\text{ s.t. }f^{\bullet}_{h(k)}(\PackedTree^{\bullet})\text{ contains at least } k \text{ leaves before and}\\
&\text{$k$ leaves after the distinguished leaf, and has a root decorated with an $\cS$-gadget}\big\}.
\end{split}
\end{equation}

\begin{proposition}
	\label{prop:continuity_RP}
	$\RP:(\lufPT,d_t)\to(\tilde{\SG}_{\bullet},d_p)$ is continuous on $C_{\RP}$.
\end{proposition}

\begin{proof}
	Let $(\PackedTree^{\bullet}_n)_{n\geq 0}$ be a convergent sequence in $(\lufPT,d_t)$ with limit $\PackedTree^{\bullet}\in C_{\RP}.$ Therefore, for all $k>0$,
	\begin{itemize}
		\item there exist $N(k)>0$ such that $d_t(\PackedTree^{\bullet}_n,\PackedTree^{\bullet})\leq 2^{-k},$ for all $n\geq N(k)$ (since $\PackedTree^{\bullet}_n\to\PackedTree^{\bullet})$;
		\item there exists $h(k)>0$ such that $f^{\bullet}_{h(k)}(\PackedTree^{\bullet})$ contains at least $k$ leaves before and $k$ leaves after the distinguished leaf, and has a root decorated with an $\cS$-gadget (since $\PackedTree^{\bullet}\in C_{\RP}$).
	\end{itemize} 
	In particular, for all $k>0,$ setting $N'(k)=N(h(k))$ we have that
	\begin{itemize}
		\item $f_{h(k)}^{\bullet}(\PackedTree^{\bullet}_n)=f_{h(k)}^{\bullet}(\PackedTree^{\bullet})$ for all $n\geq N'(k)$;
		\item $f_{h(k)}^{\bullet}(\PackedTree^{\bullet})$ contains at least $k$ leaves before and $k$ leaves after the distinguished leaf, and has a root decorated with an $\cS$-gadget.
	\end{itemize} 
	\cref{lem:local_reg} implies that, for $n \ge N'(k)$,
	we have $d_p(\RP(\PackedTree^{\bullet}_n),\RP(\PackedTree^{\bullet}))\leq 2^{-k}$.
	Since such a $N'(k)$ exists for all $k>0$,
	we conclude that $\RP(\PackedTree^{\bullet}_n)\to \RP(\PackedTree^{\bullet})$.
	Therefore the function $\RP$ is continuous on $C_{\RP}$, as claimed.
\end{proof}

As a final preparation result for the proof of \cref{thm:local_intro} in the non-separable case,
we show that the limit object $\bm{P}^\bullet_\infty$ is in the continuity set of $\RP$ with probability $1$.
\begin{proposition}
	\label{prop:prob_discontinuity_RP}
	We have $\Proba(\bm{P}^\bullet_\infty \in C_{\RP})=1$.
\end{proposition}

\begin{proof}
	Obviously we can rewrite $\Proba(\bm{P}^\bullet_\infty \in C_{\RP} )$ as
	\begin{equation*}
	\begin{split}
	\Proba\big(\forall k>0,\;\exists h(k)>0\text{ s.t. }f^{\bullet}_{h(k)}(\bm{P}^\bullet_\infty)\text{ contains at least } k \text{ leaves before}
	\text{ and } k \text{ leaves after } u_0,\\
	\text{and has a root decorated with an $\cS$-gadget}\big).
	\end{split}
	\end{equation*}
	Since the problem is symmetric, it is enough to show that for each fixed $k>0,$
	\begin{equation*}
	\begin{split}
	\Proba\big(\exists h(k) >0\text{ s.t. }f^{\bullet}_{h(k)}(\bm{P}^\bullet_\infty)\text{ contains at least } k \text{ leaves before } u_0&\\
	\text{ and has a root decorated with an $\cS$-gadget}&\big)=1.
	\end{split}
	\end{equation*}
	Note that
	\begin{multline*}
	\Proba\big(\exists h(k)>0\text{ s.t. }f^{\bullet}_{h(k)}(\bm{P}^\bullet_\infty)\text{ contains at least } k \text{ leaves before } u_0\\
	\text{ and has a root decorated with an $\cS$-gadget}\big)\\
	\geq \Proba\big(\bm{P}^\bullet_\infty \text{ has at least } k \text{ vertices } u_i \text{ in the infinite spine having at least one left child}\\
	\text{ and an $\cS$-gadget as decoration}\big).
	\end{multline*}
	Here and after, left child means child to the left of the infinite spine. 
	
	By construction, in the infinite tree $\bm{P}^\bullet_\infty$,
	the vertex $u_i$ has at least one left child when $u_{i-1}$ is not identified with its first offspring.
	Conditioned on $u_i$ having $d$ children (which happens with probability $P(\hat{\xi}=d)$),
	this occurs with probability $1-1/d$. Moreover, conditioning on $u_i$ having $d$ children, the probability that $u_{i}$ has an $\cS$-gadget as decoration is equal to $\tfrac{q_d-1}{q_d}$, where we recall that $\mathcal Q(z)=\widehat{\GGG(\SSS)}(z)=\sum_{k\geq 2}q_kz^k$ is the generating in \cref{eq:G}, 
	and that $q_d >1 $ for some $d$ (since we are not treating the case of separable permutations here).
	Therefore, for all $i\geq1$,
	\begin{multline*}
	\Proba\big(u_i \text{ has at least one left child} \\ \text{and is decorated by an $\cS$-gadget} \big)
	=\sum_{d\geq 2}\Proba(\hat{\xi}=d)(1-1/d)\tfrac{q_d-1}{q_d}>0.
	\end{multline*}
	By construction, all these events (for all $i\geq1$) are independent.
	Since they happen with some positive probability independent of $i$,
	a.s.\ at least $k$ of these events hold.
	Consequently, $\bm{P}^\bullet_\infty$ has a.s.\
	at least $k$ vertices $u_i$ in its infinite spine that have at least one left child  and are decorated by an $\cS$-gadget.
	This concludes the proof.
\end{proof}

\subsection{The separable permutations case}
\label{sec:map_for_sep_perm}
For the class of separable permutations, we cannot extend as before the map $\DT^{-1}$ as a function $\RP$ from the metric space of (possibly infinite) locally and upwards finite pointed packed trees $(\lufPT,d_t)$ to the metric space of (possibly infinite) rooted permutations $(\tilde{\SG}_{\bullet},d_p)$. Indeed, every packed tree obtained from a separable permutation contains only $\circledast$-decorations. 

Instead, in this case, we have to consider two different functions $\RP^+$ and $\RP^-$ from the metric space of (possibly infinite) locally and upwards finite pointed rooted trees 
to the metric space of (possibly infinite) rooted permutations.
We first define the maps for finite rooted trees pointed at a leaf (where all internal vertices are thought of as decorated by $\circledast$). 
Let $(T,\ell)$ be such a tree. 
We denote with $(T_{\oplus},\ell)$ (resp. $(T_{\ominus},\ell)$) the pointed canonical tree obtained from $(T,\ell)$ labelling the parent of $\ell$ with $\oplus$ (resp. $\ominus$) 
and then labelling all the other internal vertices in the unique way that prevents the creation of $\oplus-\oplus$ or $\ominus-\ominus$ edges.
Denoting by $i$ the label of the leaf $\ell$ (in the sense of \cref{rk:Leaves_Elements}), we set
\begin{equation}
\label{rp_functions}
\RP^+(T,\ell)\coloneqq (\CanTree^{-1}(T_{\oplus}),i)\quad\text{and}\quad\RP^-(T,\ell)\coloneqq (\CanTree^{-1}(T_{\ominus}),i).
\end{equation}
Finally, given an infinite pointed tree $T^{\bullet}$ we set 
\[\RP^+(T^{\bullet})\coloneqq \lim_{h\to\infty}	\RP^+(f^{\bullet}_h(T^\bullet))\quad\text{and}\quad\RP^-(T^{\bullet})\coloneqq \lim_{h\to\infty}	\RP^-(f^{\bullet}_h(T^\bullet)).\]
where the existence of the two limits is justified using similar arguments to the ones used in \cref{lem:limexistence}.
We now set
\begin{equation}
\begin{split}
C_{\RP^*}:=\big\{T^{\bullet}\in\lufT:\forall k>0,\;\exists h(k)>0\text{ s.t. }f^{\bullet}_{h(k)}(T^{\bullet})\text{ contains at least } k \text{ leaves before}&\\
\text{and $k$ leaves after the distinguished leaf.}&\big\}.
\end{split}
\end{equation}

With very similar arguments to the ones used in \cref{prop:continuity_RP,prop:prob_discontinuity_RP} we have the following.
\begin{proposition}
	\label{prop:cont_bij_sep_case}
	The functions $\RP^+:(\lufT,d_t)\to(\tilde{\SG}_{\bullet},d_p)$ and $\RP^-:(\lufT,d_t)\to(\tilde{\SG}_{\bullet},d_p)$ are continuous on $C_{\RP^*}$. Moreover, $\Proba(\bm{T}^\bullet_\infty \in C_{\RP^*})=1.$
\end{proposition}

We conclude this section with the following result dealing with the local limit of a uniform canonical tree $\bm{T}_n$ associated with separable permutations, conditioned on having $n$ leaves, 
and where decorations have been removed. 
We note that $\bm T_n$ is distributed as the random packed tree considered in \cref{lem:unif_packed_tree} for the case of separable permutations 
(where decorations, which are all $\circledast$, have also been removed). Therefore all the properties for the offspring distribution $\xi$ are still valid. 
In particular, we remark that $\xi$ has finite variance in the case of separable permutations. 

\begin{proposition}
	\label{prop:quenched_conv_tree_Plus_Signs}
	Let $\bm{T}_n$ be as above and  $\bm{T}^\bullet_\infty$ be the limiting random pointed tree constructed in \cref{sec:limiting_local_trees}. It holds that
	\begin{equation}
	\label{eq:random_meas_conv_sep_case}
	\mathcal{L}\Big(\big((\bm{T}_n,\bm{\ell}_n),(-1)^{\height(\bm{\ell}_n)}\big)|\bm{T}_n\Big)\stackrel{P}{\longrightarrow}\mathcal{L}\big((\bm{T}^\bullet_\infty, \bm{B}_{\pm})\big),
	\end{equation}
	where $\bm{\ell}_n$ is a uniform leaf of $\bm{T}_n$ chosen independently of $\bm{T}_n,$ $\height(\bm{\ell}_n)$ denotes the height of the leaf $\bm{\ell}_n$ and $\bm{B}_{\pm}$ is a Bernoulli random variable on $\{1,-1\}$ independent of $\bm{T}^\bullet_\infty$. 
	
	In particular, $\bm{T}_n$ converges in the quenched Benjamini--Schramm sense 
	to the deterministic measure $\mathcal{L}(\bm{T}^\bullet_\infty)$
	and in the annealed Benjamini--Schramm sense to the random tree $\bm{T}^\bullet_\infty.$  
\end{proposition}

We highlight that since we want also to keep track of the parity of the distance between the pointed leaf and the root of the tree, 
\cref{prop:quenched_conv_tree_Plus_Signs} does not follow as a simple adaptation from the proof of \cref{prop:quenched_conv_tree}.

\begin{proof}
	With similar arguments to the ones used in the first part of proof of \cref{prop:quenched_conv_tree}, in order to prove \cref{eq:random_meas_conv_sep_case}, it is enough to show that for a fixed leaf-pointed tree $T^\bullet$, and for any fixed $h$, 
	\begin{equation}
	\label{eq:goal1oftheproof_sep_case}
	\Proba\big(f_h^{\bullet}(\bm{T}_n,\bm{\ell}_n)=f_h^{\bullet}(T^\bullet),(-1)^{\height(\bm{\ell}_n)}=1\big|\bm{T}_n\big)\stackrel{P}{\longrightarrow}\Proba\big(f^\bullet_h(\bm{T}^\bullet_\infty)=f^\bullet_h(T^\bullet),\bm{B}_{\pm}=1\big).
	\end{equation}
	
	Denoting $\mathfrak{L}(\bm{T}_n)$ for the set of leaves of $\bm{T}_n$,
	the left-hand side writes
	\begin{multline}
	\Proba\big(f_h^{\bullet}(\bm{T}_n,\bm{\ell}_n)=f_h^{\bullet}(T^\bullet),(-1)^{\height(\bm{\ell}_n)}=1\big|\bm{T}_n\big)\\
	=\frac{\big|\big\{\ell\in \mathfrak{L}(\bm{T}_n):f_h^{\bullet}(\bm{T}_n,\ell)=f_h^{\bullet}(T^\bullet),(-1)^{\height(\ell)}=1\big\}\big|}{n}=:\bm{N}_{T^\bullet}(n).
	\end{multline}
	\medskip
	
	In order to prove that $\bm{N}_{T^\bullet}(n)\stackrel{P}{\longrightarrow}\Proba\big(f^\bullet_h(\bm{T}^\bullet_\infty)=f^\bullet_h(T^\bullet),\bm{B}_{\pm}=1\big)$,
	we use the Second moment method. We start by studying the first moment, which is
	\begin{equation*}
	\E[\bm{N}_{T^\bullet}(n)]=\Proba\big(f_h^{\bullet}(\bm{T}_n,\bm{\ell}_n)=f_h^{\bullet}(T^\bullet),(-1)^{\height(\bm{\ell}_n)}=1\big).
	\end{equation*}
	Using the notation of \cref{ssec:skeleton_space} and \cref{le:semilocal},
	we can rewrite this probability as follows:
	\begin{multline*}
	\Proba\big(f_h^{\bullet}(\bm{T}_n,\bm{\ell}_n)=f_h^{\bullet}(T^\bullet),(-1)^{\height(\bm{\ell}_n)}=1\big)\\
	=\Proba\big(\Sh(c_{\{0\}}\sigma n^{-1/2}.R^{[h]}(\bm{T}_n,\bm{\ell}_n))\in A_{f_h^{\bullet}(T^\bullet)},(-1)^{\height(\bm{\ell}_n)}=1\big),
	\end{multline*}
	where $A_{f_h^{\bullet}(T^\bullet)}$ is the set of trees $G$ in $\setToh$
	such that at the $h$-th ancestor of the distinguished leaf of $G$
	is equal to $f_h^{\bullet}(T^\bullet)$. 
	
	Using \cref{le:semilocal} with $\Omega=\{0\}, k=1, t=h$ and offspring distribution equal to the one for separable permutations, and the additional result (given by \cref{cor:semilocal3}) that the parity of the height of $\bm{\ell}_n$ converges to a fair coin flip, we have
	\begin{equation*}
	\Proba\big(f_h^{\bullet}(\bm{T}_n,\bm{\ell}_n)=f_h^{\bullet}(T^\bullet),(-1)^{\height(\bm{\ell}_n)}=1\big)\longrightarrow\Proba\big(\Sh(\bm{T}^{1,h}_{\{0\}})\in A_{f_h^{\bullet}(T^\bullet)},\bm{B}_{\pm}=1\big).
	\end{equation*}
	By comparing the construction of $\bm{T}^{1,h}_{\{0\}}$ in \cref{sec:limit_tree}
	and that of $\bm{T}^\bullet_\infty$, we have 
	\[\Proba\big(\Sh(\bm{T}^{1,h}_{\{0\}})\in A_{f_h^{\bullet}(T^\bullet)},\bm{B}_{\pm}=1\big)=\Proba\big(f^\bullet_h(\bm{T}^\bullet_\infty)=f^\bullet_h(T^\bullet),\bm{B}_{\pm}=1\big).\]
	Bringing everything together yields
	\begin{equation}
	\E[\bm{N}_{T^\bullet}(n)] \longrightarrow \Proba\big(f^\bullet_h(\bm{T}^\bullet_\infty)=f^\bullet_h(T^\bullet),\bm{B}_{\pm}=1\big).
	\label{eq:FirstMoment_sep_case}
	\end{equation}
	\medskip
	
	We now study the second moment.
	We have
	\begin{equation*}
	\E[\bm{N}_{T^\bullet}(n)^2]=\Proba\big(f_h^{\bullet}(\bm{T}_n,\bm{\ell}_n)=f_h^{\bullet}(T^\bullet)=f_h^{\bullet}(\bm{T}_n,\bm{g}_n),(-1)^{\height(\bm{\ell}_n)}=1=(-1)^{\height(\bm{g}_n)}\big),
	\end{equation*}
	where $\bm \ell_n$ and $\bm g_n$ are two uniform random leaves of $\bm T_n$, taking independently conditionally on $\bm T_n$.
	Again, using the notation of \cref{ssec:skeleton_space,le:semilocal}, we can rewrite this probability as follows:
	\begin{multline*}
	\Proba\big(f_h^{\bullet}(\bm{T}_n,\bm{\ell}_n)=f_h^{\bullet}(T^\bullet)=f_h^{\bullet}(\bm{T}_n,\bm{g}_n),(-1)^{\height(\bm{\ell}_n)}=1=(-1)^{\height(\bm{g}_n)}\big)\\
	=\Proba\big(\Sh(c_{\{0\}}\sigma n^{-1/2}.R^{[h]}(\bm{T}_n,(\bm{\ell}_n,\bm{g}_n)))\in B_{f_h^{\bullet}(T^\bullet)},(-1)^{\height(\bm{\ell}_n)}=1=(-1)^{\height(\bm{g}_n)}\big),
	\end{multline*}
	where $B_{f_h^{\bullet}(T^\bullet)}$ is the set of $G$ in $\setTth$ 
	such that the two fringe subtrees rooted at the $h$-th ancestors of
	the two distinguished leaves
	are both equal to $f_h^{\bullet}(T^\bullet)$. 
	
	Using again \cref{le:semilocal} with $\Omega=\{0\}, t=h, k=2$ and
	offspring distribution equal to the one for separable permutations, 
	and the additional result (given by \cref{cor:semilocal3}) that the parities of the height of $\bm{\ell}_n$ and $\bm{g}_n$ converges to two independent fair coin flips, we have
	\begin{multline*}
	\Proba\big(f_h^{\bullet}(\bm{T}_n,\bm{\ell}_n)=f_h^{\bullet}(T^\bullet)=f_h^{\bullet}(\bm{T}_n,\bm{g}_n),(-1)^{\height(\bm{\ell}_n)}=1=(-1)^{\height(\bm{g}_n)}\big)\\
	\longrightarrow\Proba\big(\Sh(\bm{T}^{2,h}_{\{0\}})\in B_{f_h^{\bullet}(T^\bullet)},\bm{B}^1_{\pm}=1,\bm{B}^2_{\pm}=1\big),
	\end{multline*}
	where $\bm{B}^i_{\pm},$ for $i=1,2,$ are two independent copies of $\bm{B}_{\pm}.$
	By construction, in $\bm{T}^{2,h}_{\{0\}}$ the neighbourhoods of the two distinguished vertices (here leaves, since $\Omega=\{0\}$)
	are taken independently so that
	\[\Proba\big(\Sh(\bm{T}^{2,h}_{\{0\}})\in B_{f_h^{\bullet}(T^\bullet)},\bm{B}^1_{\pm}=1,\bm{B}^2_{\pm}=1\big) =
	\Proba\big(f^\bullet_h(\bm{T}^\bullet_\infty)=f^\bullet_h(T^\bullet),\bm{B}_{\pm}=1\big)^2.\]
	Bringing everything together,
	\begin{equation}
	\E[\bm{N}_{T^\bullet}(n)^2] \longrightarrow \Proba\big(f^\bullet_h(\bm{T}^\bullet_\infty)=f^\bullet_h(T^\bullet),\bm{B}_{\pm}=1\big)^2.
	\label{eq:SecondMoment_sep_case}
	\end{equation}
	\medskip
	
	Comparing \cref{eq:FirstMoment_sep_case,eq:SecondMoment_sep_case} and using the standard second moment method, we conclude that
	\begin{equation*}
	\bm{N}_{T^\bullet}(n)\stackrel{P}{\longrightarrow}\Proba\big(f^\bullet_h(\bm{T}^\bullet_\infty)=f^\bullet_h(T^\bullet),\bm{B}_{\pm}=1\big).
	\end{equation*}
	Indeed by Chebyschev's inequality, one has, for any fixed $\varepsilon>0,$
	\begin{equation*}
	\Proba\Big(\big|	\bm{N}_{T^\bullet}(n)-\E\big[\bm{N}_{T^\bullet}(n)\big]\big|\geq\varepsilon\Big)\leq\frac{1}{\varepsilon^2}\cdot\text{Var}\big(\bm{N}_{T^\bullet}(n)\big),
	\end{equation*}
	and the right-hand side tends to zero.
\end{proof}

\subsection{Local limit of uniform permutations in substitution-closed classes}

We now prove a quenched Benjamini--Schramm convergence result
for uniform random permutations in a proper substitution-closed class $\mathcal C$.
As we shall see at the end of the section, this implies our second main result (\cref{thm:local_intro}).

\begin{theorem}
	\label{thm:quenched_BS_cv}
	Let $\bm{\nu}_n$ be a uniform permutation of size $n$ in a proper substitution-closed class $\mathcal{C}$, for all $n\in\NN$. 
	If $\mathcal{C}$ is the class of separable permutations, then
	\begin{align}\bm{\nu}_n\stackrel{qBS}{\longrightarrow}\mathcal{L}\big(\RP^{\bm{B}_{\pm}}(\bm{T}^\bullet_\infty)\big)\quad\text{and}\quad \bm{\nu}_n\stackrel{aBS}{\longrightarrow}\RP^{\bm{B}_{\pm}}(\bm{T}^\bullet_\infty).
	\label{eq:local_limit_sep}
	\end{align}
	If the set $\cS$ of simple permutations in $\cC$ is non-empty and the criticality condition
	\begin{align}
	\cS'(\rho_\cS) \ge \frac{2}{(1 +\rho_\cS)^2} -1
	\end{align}
	is satisfied, then
	\begin{align}
	\bm{\nu}_n\stackrel{qBS}{\longrightarrow}\mathcal{L}\big(\RP(\bm{P}^\bullet_\infty)\big)\quad\text{and}\quad \bm{\nu}_n\stackrel{aBS}{\longrightarrow}\RP(\bm{P}^\bullet_\infty).
	\label{eq:local_lim_nonsep}
	\end{align}
\end{theorem}

Like after \cref{prop:quenched_conv_tree}, we want to emphasize the nature of the limiting objects above. 
The limit $\mathcal{L}\big(\RP(\bm{P}^\bullet_\infty)\big)$ (resp. the limit $\mathcal{L}\big(\RP^{\bm{B}_{\pm}}(\bm{T}^\bullet_\infty)\big)$) is a measure on $\Sri$.
Since the limiting object for the quenched Benjamini--Schramm convergence is in general a random measure on $\Sri$,
it should be interpreted as a constant random variable, 
equal to the measure $\mathcal{L}\big(\RP(\bm{P}^\bullet_\infty)\big)$ (resp. $\mathcal{L}\big(\RP^{\bm{B}_{\pm}}(\bm{T}^\bullet_\infty)\big)$).
\begin{proof}
	We only need to prove the quenched convergence statements,
	the annealed versions being a simple consequence of the quenched one (see \cite[Proposition 2.35]{Borga2019}).
	Moreover, thanks to \cref{prop:giant_comp_perm},
	it is sufficient to prove the statement for a uniform $\oplus$-indecomposable
	permutation $\bm\nu_n$.

	We first consider the case when $\mathcal{C}$ is a proper substitution-closed class different from the class of separable permutations.
	Consider a uniform random leaf $\bm{\ell}_{n}$ in $\bm{P}_{n}$
	and a uniform random element $\bm i_{n}$ in $\bm{\nu}_n$.
	We have the following equality in distribution
	(recall that $\RP$ denotes the          
	extension of the function $(\Pack\circ\CanTree)^{-1}$ to rooted permutations):
	\begin{equation}
	\big(\bm{\nu}_n,\bm i_n\big) \stackrel{d}{=}\RP(\bm{P}_n,\bm{\ell}_n).
	\label{eq:randomNonPlus_Tree}
	\end{equation}
	
	We analyse the right-hand side conditionally on $\bm{P}_{n}$.
	By \cref{prop:quenched_conv_tree}, we know that
	\begin{equation*}
	\mathcal{L}\big((\bm{P}_n,\bm{\ell}_n)|\bm{P}_n\big)\stackrel{P}{\longrightarrow}\mathcal{L}(\bm{P}^\bullet_\infty).
	\end{equation*}
	Moreover, by \cref{prop:continuity_RP,prop:prob_discontinuity_RP}, $\RP$ is almost surely continuous at $\bm{P}^\bullet_\infty$. 
	Therefore, using a combination of the results stated in \cite[Theorem 4.11, Lemma 4.12]{kallenberg2017random}\footnote{\label{footnote:KAL}The specific result that we need is a generalization of the \emph{mapping theorem} for random measures: Let $(\bm{\mu}_n)_{n\in\NN}$ be a sequence of random measures on a space $E$ that converges in distribution to a random measure $\bm{\mu}$ on $E.$ Let $F$ be a function from $E$ to a second space $H$ such that the set $D_F$ of discontinuity points of $F$ has measure $\bm{\mu}(D_F)=0$ a.s.. Then the sequence of pushforward random measures $(\bm{\mu}_n\circ F^{-1})_{n\in\NN}$ converges in distribution to the pushforward random measure $\bm{\mu}\circ F^{-1}.$},
	\begin{equation*}
	\mathcal{L}\big(\RP(\bm{P}_{n},\bm{\ell}_{n})|\bm{P}_{n}\big)\stackrel{P}{\longrightarrow}\mathcal{L}\big(\RP(\bm{P}^\bullet_\infty)\big).
	\end{equation*}
	Note that the result described in footnote~\ref{footnote:KAL} gives convergence \emph{in distribution}; 
	the limit being a deterministic measure, convergence in probability follows. 
	
	Comparing with \cref{eq:randomNonPlus_Tree}, we have that
	\[\mathcal{L}\big( (\bm{\nu}_n,\bm i_n) | \bm{\nu}_n \big) 
	\stackrel{P}{\longrightarrow}\mathcal{L}\big(\RP(\bm{P}^\bullet_\infty)\big),\]
	which is the quenched convergence in \cref{eq:local_lim_nonsep}.
	\medskip
	
	It remains to prove the theorem for the class of separable permutations.
	In this case, we have the following equality in distribution (recall that $\RP^+$ and $\RP^-$ are the maps defined in \cref{rp_functions})
	\begin{equation}
	\big(\bm\nu_n,\bm i_n\big) \stackrel{d}{=}\RP^{\sgn(\bm \ell)}(\bm{T}_{n},\bm{\ell}_{n}),
	\label{eq:randomNonPlus_Tree_sep_case}
	\end{equation}
	where $\bm{T}_{n}$ is a uniform undecorated canonical tree with $n$ leaves, 
	$\bm{\ell}_{n}$ is a uniform leaf of $\bm{T}_{n}$ and
	$\sgn(\bm \ell)$ is the sign $(-1)^{\height(\bm \ell)}$. 
	
	We analyse the right-hand side conditionally on $\bm{T}_n$.
	By \cref{prop:quenched_conv_tree_Plus_Signs}, we know that
	\begin{equation}
	\mathcal{L}\Big(\big((\bm{T}_n,\bm{\ell}_n),(-1)^{\height(\bm{\ell}_n)}\big)|\bm{T}_n\Big)\stackrel{P}{\longrightarrow}\mathcal{L}\big((\bm{T}^\bullet_\infty, \bm{B}_{\pm})\big).
	\end{equation}
	Moreover, by \cref{prop:cont_bij_sep_case} $\RP^{+}$ and $\RP^{-}$ are almost surely continuous at $(\bm{T}^\bullet_\infty, \bm{B}_{\pm})$. Therefore, using again a combination of the results stated in \cite[Theorem 4.11, Lemma 4.12]{kallenberg2017random} 
	\begin{equation*}
	\mathcal{L}\big(\RP^{\sgn(\bm{\ell}_{n})}(\bm{T}_{n},\bm{\ell}_{n})|\bm{T}_{n}\big)\stackrel{P}{\longrightarrow}\mathcal{L}\big(\RP^{\bm{B}_{\pm}}(\bm{T}^\bullet_\infty)\big).
	\end{equation*}
	Comparing with \cref{eq:randomNonPlus_Tree_sep_case}, we have that
	\[\mathcal{L}\big( (\bm{\nu}_n,\bm i_{n}) | \bm{\nu}_n \big)
	\stackrel{P}{\longrightarrow}\mathcal{L}\big(\RP^{\bm{B}_{\pm}}(\bm{T}^\bullet_\infty)\big),\]
	which is exactly the quenched convergence statement in \cref{eq:local_limit_sep}.
\end{proof}

\begin{proof}
	[Proof of \cref{thm:local_intro}]
	With the assumption of \cref{thm:local_intro}, we just proved (\cref{thm:quenched_BS_cv})
	that a uniform permutation $\bm \nu_n$ in $\mathcal C$
	converges in the quenched Benjamini--Schramm sense to some deterministic measure $\mathcal{L}\big(\bm \nu_\infty)$.
	As recalled in \cref{thm:local_conv_perm_charact} above,
	the quenched Benjamini--Schramm convergence imply the (joint) convergence
	of the random variables $\widetilde{\coc}(\pi,\bm{\nu}_n)$ to some random variables $\bm\Lambda_\pi$.
	Additionally, since the quenched Benjamini--Schramm limit is a deterministic measure,
	the random variable $\bm\Lambda_\pi$ are deterministic as well (see \cite[Corollary 2.38]{Borga2019}),
	\emph{i.e.,} they are numbers $\gamma_{\pi,\mathcal C}$ in $[0,1]$.
	This concludes the proof.
\end{proof}
\begin{remark}
	\label{rk:gammas}
	Concretely $\gamma_{\pi,\mathcal C}$ is the probability 
	that the restriction of the random order $\RP(\bm{P}^\bullet_\infty)$ 
	(or, in the case of separable permutations, $\RP^{\bm{B}_{\pm}}(\bm{T}^\bullet_\infty)$)
	on a fixed integer interval of size $|\pi|$
	(e.g. $[0,|\pi|-1]$) is equal to $\pi$ (after the identification between permutations and total order on intervals
	given in \cref{ssec:reminder_local_limits}).
	Computing this number involves a sum over countably many configurations of $\bm{P}^\bullet_\infty$
	and so it is not immediate, even for simple classes $\mathcal C$ and short patterns $\pi$.
\end{remark}




\providecommand{\bysame}{\leavevmode\hbox to3em{\hrulefill}\thinspace}
\providecommand{\MR}{\relax\ifhmode\unskip\space\fi MR }
\providecommand{\MRhref}[2]{%
	\href{http://www.ams.org/mathscinet-getitem?mr=#1}{#2}
}
\providecommand{\href}[2]{#2}


\ACKNO{JB and MB are partially supported by the Swiss National Science Foundation, under grant number 200021-172536.}


\end{document}